\documentclass[10pt,reqno, a4paper]{amsart}
\usepackage{color}
\usepackage[colorlinks=true, allcolors=blue,backref=page]{hyperref}
\usepackage{amsmath, amssymb, amsthm}
\usepackage{mathrsfs}
\usepackage{mathtools}
\usepackage[noabbrev,capitalize,nameinlink]{cleveref}
\crefname{equation}{}{}
\usepackage{fullpage}
\usepackage[noadjust]{cite}
\usepackage{graphics}
\usepackage{pifont}
\usepackage{tikz}
\usepackage{bbm}
\usepackage[T1]{fontenc}

\usetikzlibrary{arrows.meta}

\usepackage{environ}
\usepackage{framed}
\usepackage{url}
\usepackage[linesnumbered,ruled,vlined]{algorithm2e}
\usepackage[noend]{algpseudocode}
\usepackage[labelfont=bf]{caption}
\usepackage{cite}
\usepackage{framed}
\usepackage[framemethod=tikz]{mdframed}
\usepackage{appendix}
\usepackage{graphicx}
\usepackage[textsize=tiny]{todonotes}
\usepackage{tcolorbox}
\usepackage{enumerate}
\allowdisplaybreaks[1]
\usepackage{enumerate}
\usepackage{stmaryrd}
\usepackage[margin=2.3cm]{geometry}

\usepackage[shortlabels]{enumitem}
\crefformat{enumi}{#2#1#3}
\crefrangeformat{enumi}{#3#1#4 to~#5#2#6}
\crefmultiformat{enumi}{#2#1#3}
{ and~#2#1#3}{, #2#1#3}{ and~#2#1#3}

\DeclareSymbolFont{symbolsC}{U}{pxsyc}{m}{n}
\SetSymbolFont{symbolsC}{bold}{U}{pxsyc}{bx}{n}
\DeclareFontSubstitution{U}{pxsyc}{m}{n}
\DeclareMathSymbol{\medcircle}{\mathbin}{symbolsC}{7}

\crefname{algocf}{Algorithm}{Algorithms}

\crefname{equation}{}{} 
\AtBeginEnvironment{appendices}{\crefalias{section}{appendix}} 

\usepackage[color,final]{showkeys} 

\colorlet{refkey}{orange!20}
\colorlet{labelkey}{blue!30}

\crefname{algocf}{Algorithm}{Algorithms}

\numberwithin{equation}{section}
\newtheorem{theorem}{Theorem}[section]
\newtheorem{proposition}[theorem]{Proposition}
\newtheorem{lemma}[theorem]{Lemma}

\crefname{claim}{Claim}{Claims}

\newtheorem{corollary}[theorem]{Corollary}

\newtheorem*{question*}{Question}

\theoremstyle{definition}
\newtheorem{definition}[theorem]{Definition}

\newtheorem*{definition*}{Definition}

\theoremstyle{remark}
\newtheorem*{remark}{Remark}

\newcommand{\pvec}[1]{{\vec{#1}\,}'}

\newcommand{\mb}{\mathbb}
\newcommand{\mbf}{\mathbf}
\newcommand{\mbm}{\mathbbm}
\newcommand{\mc}{\mathcal}
\newcommand{\mf}{\mathfrak}
\newcommand{\mr}{\mathrm}

\newcommand{\on}{\operatorname}

\newcommand{\wt}{\widetilde}

\newcommand{\E}{\mathbb{E}}

\let\originalleft\left
\let\originalright\right
\renewcommand{\left}{\mathopen{}\mathclose\bgroup\originalleft}
\renewcommand{\right}{\aftergroup\egroup\originalright}

\allowdisplaybreaks

\newif\ifpublic

\ifpublic

\newcommand{\ignore}[1]{}

\else

\fi

\title{High-Girth Steiner Triple Systems}

\author[Kwan]{Matthew Kwan}
\address{Institute of Science and Technology Austria (IST Austria), 3400 Klosterneuburg, Austria}
\email{matthew.kwan@ist.ac.at}

\author[A2]{Ashwin Sah}
\author[A3]{Mehtaab Sawhney}
\address{Department of Mathematics, Massachusetts Institute of Technology, Cambridge, MA 02139, USA}
\email{\{asah,msawhney\}@mit.edu}

\author[Simkin]{Michael Simkin}
\address{Center of Mathematical Sciences and Applications, Harvard University, Cambridge, MA 02138, USA}
\email{msimkin@mit.edu}

\thanks{Sah and Sawhney were supported by NSF Graduate Research Fellowship Program DGE-1745302. Sah was supported by the PD Soros Fellowship. Simkin was supported by the Center of Mathematical Sciences and Applications at Harvard University.}

\begin{document}

\maketitle
\begin{abstract}
We prove a 1973 conjecture due to Erd\H{o}s on the existence of Steiner triple systems with arbitrarily high girth.
\end{abstract}

\section{Introduction}
Extremal combinatorics is largely concerned with the extent to which global statistics of an object (such as its \emph{volume} or \emph{density}) control the existence of local substructures. Classical examples include Szemer\'edi's theorem on arithmetic progressions in dense subsets of the integers, and the Erd\H os--Stone--Simonovits theorem on the density threshold which guarantees the existence of a given subgraph.

In the extremal theory of \emph{hypergraphs}, one of the most fundamental research directions, initiated by Brown, Erd\H os, and S\'os~\cite{BES73}, concerns the density conditions which guarantee the existence of local configurations with specific statistics. To be precise (and specializing to 3-uniform hypergraphs, also known as \emph{triple systems}), a \emph{$(j,\ell)$-configuration} is a set of $\ell$ triples which span at most $j$ vertices. Perhaps the most well-known theorem in this area concerns the so-called \emph{$(6,3)$-problem}: it was proved by Rusza and Szemer\'edi~\cite{RS78} that every $N$-vertex triple system with no $(6,3)$-configuration has only $o(N^{2})$ triples. This innocent-sounding fact has a number of surprising consequences, including Roth's theorem on 3-term arithmetic progressions in dense sets of integers.

In general, problems of this type are formidably difficult (for example, it is a famous open problem to prove more generally that for any fixed $j\in\mathbb{N}$, $N$-vertex triple systems with no $(j,j-3)$-configuration have $o(N^{2})$ edges), and it seems far out of reach to precisely characterize the number of triples which forces the existence of a given type of configuration (or even the order of magnitude of this number), in general. However, certain special cases are tractable. In particular, Brown, Erd\H os, and S\'os found the order of magnitude of the solution to the $(j,j-2)$-problem, for any $j\ge 4$: there are constants $C_{j}>c_{j}>0$ such that any $N$-vertex triple system with at least $C_{j}N^{2}$ triples contains a $(j,j-2)$-configuration, while on the other hand, for any $N$ there is an $N$-vertex triple system with $c_{j}N^{2}$ triples and no $(j,j-2)$-configuration\footnote{The only other nontrivial class of $(j,\ell)$-problems for which the order of magnitude is known are $(2\ell+1,\ell)$-problems, which are somewhat degenerate (in a triple system with no $(2\ell+1,\ell)$-configuration, every connected component has at most $\ell-1$ triples).}. They also observed that it is possible to give a more-or-less exact solution to the $(4,2)$-problem. Indeed, the property of having no $(4,2)$-configuration is equivalent to the property that every pair of vertices is included in at most one triple (which implies that every vertex is included in at most $(N-1)/2$ triples, and the total number of triples is at most $\binom{N}{2}/3$). As long as $(N-1)/2$ and $\binom{N}{2}/3$ are integers\footnote{This happens if and only if $N$ is congruent to $1$ or $3\pmod{6}$; such $N$ are said to be \emph{admissible}.}, it is a classical fact\footnote{This is actually one of the oldest theorems in combinatorics, having been proved by Kirkman~\cite{Kir47} in 1847.} that there exists an $N$-vertex \emph{Steiner triple system}, in which every pair of vertices is included in \emph{exactly} one triple (such a triple system has exactly $\binom{N}{2}/3$ triples).

As a far-reaching generalization of the above two facts, Erd\H os conjectured~\cite{Erd76} that for any $g\in \mb N$, if $N$ is admissible and sufficiently large in terms of $g$, then there is a Steiner triple system which has no $(j,j-2)$-configuration for any $4\le j\le g$. That is to say, the $(4,2)$-problem is the most restrictive of the nontrivial $(j,j-2)$-problems, in the very strong sense that given the constraint of having no $(4,2)$-configurations, it makes basically no difference to further impose that there is no $(j,j-2)$-configuration for any other $j$. Erd\H os' conjecture has been reiterated by many authors in different contexts; see for example~\cite{EL14,FR13,LPR93,Lin18,KL20,GP96,Bee02,CF09,GKO21,ST20,Glo19,KKLS18,Bee01,ST21,BN20,Col09,GGAHWW18}. In this paper we prove Erd\H os' conjecture.

\begin{theorem}\label{thm:main}
Given $g\in\mb{N}$, there is $N_{\ref{thm:main}}(g)\in\mb{N}$ such that if $N \ge N_{\ref{thm:main}}(g)$ and $N$ is congruent to $1$ or $3\pmod{6}$, then there exists a Steiner triple system of order $N$ which contains no $(j,j-2)$-configuration for any $4\le j\le g$.
\end{theorem}

It is natural to view \cref{thm:main} as being about the existence of Steiner triple systems which are ``locally sparse'' or have ``high girth'' (recall that the girth of a graph is the length of its shortest cycle, or equivalently the smallest integer $g$ for which there is a set of $g$ vertices containing at least $g$ edges).

\begin{definition}
\label{def:erd-config}
The \emph{girth} of a triple system is the smallest integer $g\ge 4$ for which it has a $(g,g-2)$-configuration. If no such configuration exists, the girth is defined to be infinite. If a triple system has girth strictly greater than $r+2$, it is said to be \emph{$r$-sparse}.
\end{definition}

In the language of \cref{def:erd-config}, \cref{thm:main} says that for any $r$ and sufficiently large admissible $N$, there exists an $r$-sparse Steiner triple system of order $N$. Interpreted in this way, \cref{thm:main} is of interest beyond its relevance to extremal hypergraph theory. Indeed, in the intervening years since Erd\H os made his conjecture, it has become notorious in design theory (roughly speaking, combinatorial designs are set systems with very strong regularity properties; Steiner triple systems are fundamental examples). Also, Erd\H os' conjecture has recently been reiterated in the context of the \emph{high-dimensional combinatorics} program spearheaded by Linial and collaborators (see for example \cite{Lin18}), in which a Steiner triple system can be viewed as a 2-dimensional generalization of a graph perfect matching.

\subsection{Previous work}
In design theory, there has been quite intensive effort towards studying the existence of $r$-sparse Steiner triple systems for small $r$. Every Steiner triple system is $2$-sparse, and in fact also $3$-sparse (one can show that every $(5,3)$-configuration contains a $(4,2)$-configuration). The $4$-sparse case of Erd\H os' conjecture was resolved by Grannell, Griggs, and Whitehead~\cite{4d} following previous partial results by various authors~\cite{4a,4b,4c} (4-sparse Steiner triple systems are also said to be \emph{anti-Pasch} or \emph{quadrilateral-free}, and exist for all admissible orders except $7$ and $13$). No further cases of Erd\H os' conjecture were known until now, but there had been significant progress in the cases $r=5$ and $r=6$: Wolfe~\cite{5d} proved that 5-sparse Steiner triple systems exist for ``almost all'' admissible orders, in a certain asymptotic sense (see \cite{5a,5b,5c} for previous results), and Forbes, Grannell and Griggs~\cite{6a} found a recursive construction that produces infinitely many 6-sparse Steiner triple systems. Until now no 7-sparse Steiner triple systems were known to exist.

In a different direction, Glock, K\"uhn, Lo, and Osthus~\cite{GKLO20} and Bohman and Warnke~\cite{BW19} recently proved an \emph{approximate} version of \cref{thm:main}:
namely, for any $r\in \mb N$ there is an $r$-sparse triple system with $(1-o(1))N^2\!/6$ triples (recall that a Steiner triple system has exactly $N(N-1)/6$ triples). The ideas in these proofs play a crucial role in our proof of \cref{thm:main}, as we will next discuss.

\subsection{Overview of key ideas}\label{sub:intro-overview}
In this subsection we discuss some of the key ideas in the proof of \cref{thm:main}, at a very high level. We provide a more detailed proof outline in \cref{sec:overview}. For the benefit of non-expert readers, we start by briefly examining some of the beautiful and very powerful ideas underlying the two independent proofs of the so-called \emph{existence of designs} conjecture, due to Keevash~\cite{Kee14} and Glock, K\"uhn, Lo, and Osthus~\cite{GKLO16}, which revolutionized design theory and will play an essential role in our proof of \cref{thm:main}.

\subsubsection{Approximate decompositions and random greedy algorithms}\label{subsub:intro-approximate-decompositions}
An order-$N$ Steiner triple system is equivalent to a decomposition of the complete graph $K_{N}$ into edge-disjoint triangles. The starting point for both the existence of designs proofs in \cite{Kee14,GKLO16} (specialized to Steiner triple systems) is the fact that it is possible to find an \emph{approximate} triangle-decomposition of $K_{N}$: a collection of edge-disjoint triangles
that cover a $(1-o(1))$-fraction of the edges of $K_{N}$.

Famously, this fact may be proved via the \emph{probabilistic method}: for example, we can consider the \emph{random greedy} algorithm which builds a collection of triangles one-by-one, at each step choosing a uniformly random triangle among those which are edge-disjoint from previously chosen triangles. Spencer~\cite{Spe95} and R\"odl and Thoma~\cite{RT96} proved that this process is likely to produce an approximate triangle-decomposition, following a foundational paper of R\"odl~\cite{Rod85} that considered a slightly different random process (the so-called \emph{R\"odl nibble}). These ideas have had an enormous impact throughout combinatorics
in the last 30 years.

The aforementioned works of Glock, K\"uhn, Lo, and Osthus~\cite{GKLO20} and Bohman and Warnke~\cite{BW19}, proving an approximate version of \cref{thm:main}, proceed along similar lines, studying a ``high-girth triple process'' that iteratively builds an $r$-sparse Steiner triple system (first suggested by Krivelevich, Kwan, Loh, and Sudakov~\cite{KKLS18}). It is quite surprising that this is possible, as the set of $(j,j-2)$-configurations is diverse and difficult to characterize.

\subsubsection{The absorbing method}\label{subsub:intro-absorbers}

If one's goal is to find a Steiner triple system, it is not obvious that an approximate decomposition is actually helpful. Indeed, if we consider the random greedy algorithms described above, the number of choices for the next triangle decreases at each step, and as we are nearing completion the situation becomes very restricted. One can show that these processes are very unlikely to produce a complete Steiner triple system \cite{BFL15}.

Surprisingly, it turns out that there is a general scheme that allows one to convert approximate packing/matching/decomposition results into exact ones, called the \emph{absorbing method} (pioneered by Erd\H os, Gy\'arf\'as, and Pyber~\cite{EGP91}, and later systematized by R\"odl, Ruci\'nski, and Szemer\'edi~\cite{RRS06}). Roughly speaking, the idea is to prepare an ``absorbing structure'' at the start of the proof which is very ``flexible'' and can contribute to our desired packing in many different ways. Then, one finds an approximate packing of the complement of this absorbing structure, and uses the special properties of the absorbing structure to transform this into an exact packing.

Although the absorbing method has found success in a range of problems over the years, it is difficult to use it in the setting of designs. A key obstacle is that one must create an absorbing structure which itself can be embedded in a design. Keevash's breakthrough (which he calls the method of \emph{randomized algebraic construction}) is that it is possible to construct a suitable absorbing structure by randomly sampling from algebraic objects (for Steiner triple systems, the crucial object is the triple system $\mc T$ consisting of solutions to $x+y+z=0$ among nonzero elements of an abelian 2-group $\mb Z_2^k$; see the exposition in \cite{Kee18} for more details).

Keevash's method of randomized algebraic construction is enormously powerful and has since been used to resolve many other problems in design theory~\cite{Kee18,Kee18c,KB21}. However, it seems to be uniquely unsuitable for Erd\H os' conjecture on high-girth Steiner triple systems. Actually, the triple system $\mc T$ defined above is a Steiner triple system of order $2^k-1$, and it has the maximum possible number of $(6,4)$-configurations among any Steiner triple systems of this order! These $(6,4)$-configurations are in fact crucial to the proof; they allow one to make local changes to $\mc T$, and this flexibility is essential to the utility of the absorbing structure\footnote{It is possible to consider different algebraic structures to avoid $(6,4)$-configurations in particular, but the linear-algebraic nature of Keevash's approach seems to make it fundamentally unsuitable for constructing $r$-sparse Steiner triple systems for $r\ge 7$. Relatedly, we remark that Fujiwara~\cite{Fuj07} proved some nonexistence theorems for high-girth Steiner triple systems with certain algebraic properties.}.

\subsubsection{Iterative absorption}\label{subsub:iterative-absorption}
A few years after Keevash proved the existence of designs conjecture, Glock, K\"uhn, Lo, and Osthus~\cite{GKLO16} discovered a new and completely non-algebraic proof using the method of \emph{iterative absorption} (first introduced by K\"uhn and Osthus~\cite{KO13} and Knox, K\"uhn, and Osthus~\cite{KKO15}, and since refined in various applications to a wide range of different problems; see for example \cite{BKLO17,BKLO16,GKLMO19,JJKO19,BGKLMO20,KB21,PS21}). Roughly speaking, the idea is to use bare-hands combinatorics to construct a very limited absorbing structure that can only ``fix'' very restricted types of defects, and then to iteratively transform our decomposition problem into smaller and smaller decomposition problems until it is small enough to be solved using the absorber.

Specifically, in the setting of Steiner triple systems, it is possible to use a combination of random greedy processes (extending the ideas in \cref{subsub:intro-approximate-decompositions}) to construct an approximate triangle-decomposition in $K_N$ whose leftover edges are ``localized''. Namely, all edges uncovered by triangles are contained inside a particular small set of vertices (say, 1\% of the vertices of $K_N$), and moreover almost none of the edges inside this small set are covered by triangles. This works as long as $N$ is sufficiently large (say $N\ge N_0$). Then, one can ``zoom in'' to our small set of vertices, reducing the situation to a smaller decomposition problem. This can be iterated (about $\log_{0.99} (N_0/N)$ times) until the number of vertices remaining is less than $N_0$. That is to say, in a large complete graph $K_N$ we can find a set of edge-disjoint triangles which covers all edges outside of a specific constant-size set of vertices. Since there are only a constant number of possibilities for the set of leftover edges, it suffices to construct ``exclusive'' absorbers for each: namely, for each possible graph $G$ of uncovered edges, one designs an ``absorbing'' graph $H$ with the property that both $H$ and $H\cup G$ have a triangle-decomposition.

The method of iterative absorption has certain disadvantages compared to Keevash's method of randomized algebraic construction (for example, it seems to have much weaker quantitative aspects), but it has the distinct advantage that the requirements for the absorbing structure are much weaker (we only need to be able to handle a specific, tiny, set of vertices). So, it seems plausible that one might be able to design an absorbing structure which is itself high-girth, suitable for proving Erd\H os' conjecture on high-girth Steiner triple systems.

\subsubsection{Constraint focusing}\label{subsub:intro-constraint-focusing}

Iterative absorption is indeed the approach we take in our proof of \cref{thm:main}, but there are serious issues that need to be overcome. Most obviously, the property of being $r$-sparse is not closed under taking unions: if we consider two edge-disjoint sets of triangles, which each correspond to an $r$-sparse triple system, we cannot guarantee any local sparsity properties of their union. This was highlighted as a key obstacle towards proving Erd\H os' conjecture by Glock, K\"uhn, Lo, and Osthus~\cite{GKLO20}: indeed, iterative absorption fundamentally depends on being able to augment a partial decomposition with another one on a smaller number of vertices.

Superficially, it may seem like this problem can be overcome with ``mere bookkeeping''. Indeed, after the first step of the iteration, there are specific sets of triangles that cannot be used together (as they would create a $(j,j-2)$-configuration for some $j \le r+2$, together with previously chosen triangles). Such sets can now be viewed as additional ``forbidden configurations'', which must be avoided in future steps of the iteration. So, at each step, instead of being presented with the task of finding an $r$-sparse approximate Steiner triple system, we now have the more general task of finding an approximate Steiner triple system avoiding a given collection of forbidden configurations (which one would need to track throughout the process of iterative absorption).

However, there is a fundamental issue with this plan. As we have just explained, at each step of the iteration, we do not only face an $r$-sparseness constraint, but we are also faced with other constraints inherited from previous steps of the iteration. These constraints build up over the iterations: the basic strategy of iterative absorption has the unfortunate side effect of ``focusing'' constraints on a particular small set of vertices. Recall that we need about $\log N$ steps of the iteration, but well before this point we will simply have run out of available triangles which we can use without violating the $r$-sparseness constraint.

\subsubsection{Sparsification and efficient absorption}\label{subsub:intro-sparsification}

We overcome the constraint focusing issue with a combination of two ideas. First, we observe that constraint focusing is at its most severe when performing iterative absorption on the complete graph $K_N$, simply because in every step we must take a very large number of triangles (each of which contributes constraints for future steps). If we were instead interested in finding a triangle-decomposition of (say) a sparse random graph, then the constraints arising from each step of the iteration would be negligible\footnote{This is analogous to the fact that for fixed $j\in \mb N$, in a random triple system with $m=N^2$ triples, the number of $(j,j-2)$-configurations is of the same order $m$, but in a random triple system with $m=o(N^2)$ triples, the number of $(j,j-2)$-configurations is negligible compared to $m$. Actually, this is the key fact underlying Brown, Erd\H os, and S\'os' lower bound for the $(j,j-2)$-problem.}.

So, before we begin our iterative absorption procedure, we run the high-girth triple process to find an approximate triangle-decomposition of $K_N$, in such a way that the uncovered edges resemble a sparse random graph. We then use iterative absorption to find a triangle-decomposition of this graph.

Of course, sparsification comes at a cost: namely, there are fewer edges and triangles to work with (and therefore less freedom to make choices) throughout the iterative absorption process. This issue is amplified by the nature of iterative absorption: in a sparse random graph with density $p=o(1)$, a typical set of $\sqrt{1/p}$ vertices is likely to have almost no edges at all, to say nothing about sets of constant size at the end of the iterative absorption process. The methods in \cite{GKLO16} can be quantified, and it is not strictly necessary for the final absorption step to take place on a constant-size set of vertices, but the quantitative dependence is quite poor: the final set is forced to have size at most about $\sqrt{\log N}$ (and therefore the density $p$ is forced to be at least about $1/\sqrt{\log N}$). Actually, this quantitative limitation is one of the key disadvantages of iterative absorption, compared to Keevash's randomized algebraic construction (which is effective even in quite sparse settings).

In order to overcome the constraint focusing issue, we wish to sparsify quite harshly, and thus we need a more efficient absorbing structure. We accomplish this by taking advantage of short-cycle decompositions obtainable via Euler tours; we believe this type of construction is of independent interest, and we foresee applications to other problems\footnote{Related ideas using Euler tours also appear in papers of Keevash~\cite{Kee18} and Piga and Sanhueza-Matamala~\cite{PS21}.
}.

\subsubsection{Retrospective analysis}\label{subsub:intro-regularization}
There are a significant number of technical challenges one must overcome to implement the strategy described above. For example, a glaring issue is the difficulty of bookkeeping. Generally, in packing arguments involving multiple stages of random choices (including, as far as we know, all applications of iterative absorption), one records ``pseudorandomness'' or ``typicality'' properties at each stage, which may then be used in future stages. In our setting, where we need to record the forbidden configurations inherited from previous steps of the iteration, there is a staggering amount of bookkeeping that would be necessary (see \cref{subsub:weight-systems} for a brief discussion of the relevant subtleties). While we do use notions of typicality in various parts of our proof, the forbidden configurations are largely studied via ``retrospective analysis'': instead of attempting to track the complex combinatorics of evolving forbidden configurations, we ``remember the randomness'' that was used in previous stages of the iterative absorption process, and use this to deduce information about forbidden configurations ``on demand''. This analysis is facilitated by our new formalism of \emph{weight systems}, and some general moment estimation lemmas (see \cref{sub:weight-systems}) related to \emph{Kim--Vu polynomial concentration}~\cite{KV00} and R\"odl and Ruci\'nski's \emph{deletion method}~\cite{RR95}. We believe these ideas have great potential for further applications.

\subsection{Further directions}\label{sub:further-directions}

Strengthening \cref{thm:main}, we can also prove a lower bound on the number of $r$-sparse Steiner triple systems on a given set of vertices. Let $\mr{erd}_j$ be the number of labeled $(j,j-2)$-configurations with girth $j$ on the vertex set $\{1,\dots,j\}$, which contain $\{1,2,3\}$ as a triple.
\begin{theorem}\label{thm:counting-lower-bound}
Given $g\in \mb N$, there is $c_{\ref{thm:counting-lower-bound}}(g)>0$ such that the following holds. If $N$ is congruent to $1$ or $3 \pmod 6$, then the number of (labeled) Steiner triple systems with girth greater than $g$, on a specific set of $N$ vertices, is at least
\[\left( \left( 1-N^{-c_{\ref{thm:counting-lower-bound}}(g)} \right) N\exp\left(-2-\sum_{j=6}^{g}\frac{\mr{erd}_j}{(j-2)!}\right)\right)^{\frac{N^2}{6}}.\]
\end{theorem}

We remark that it makes no difference whether we consider labeled or unlabeled Steiner triple systems in \cref{thm:counting-lower-bound}: a factor of $N!$ would be absorbed into the error term. We also remark that the case $g<6$ provides an estimate for the total number of order-$n$ Steiner triple systems, giving an independent proof of a theorem of Keevash~\cite{Kee18} (this was previously not possible with iterative absorption, due to its quantitative limitations).

We believe the lower bound in \cref{thm:counting-lower-bound} is best-possible, as conjectured by Glock, K\"uhn, Lo, and Osthus~\cite{GKLO20} (see also \cite{BW19} for the same conjecture in the case $g=6$). It is not obvious how to prove a matching upper bound, though we remark that the ideas of \cite{KSS21} may be used to prove that for $r\ge 4$, the number of $r$-sparse Steiner triple systems is at most $(cN)^{N^2/6}$ for some constant $c<e^{-2}$ (that is to say, the probability that a random Steiner triple system is $r$-sparse is of the form $e^{-\Omega(N^2)}$).

As a different strengthening of \cref{thm:main}, we believe our methods are also suitable for showing that a sufficiently dense graph $G$ satisfying certain pseudorandomness and divisibility properties admits an $r$-sparse triangle-decomposition (for example, in terms of pseudorandomness it should suffice for $G\subseteq K_N$ to have density $p\ge N^{-c_r}$, and for every set $A$ of at most four vertices to have $(1\pm \xi_r)p^{|A|}N$ neighbors, for appropriate constants $\xi_r,c_r>0$. In terms of divisibility, we need the number of edges in $G$ to be divisible by $3$, and for each degree to be even). This would imply a weak version of a conjecture of Glock, K\"uhn, and Osthus~\cite[Conjecture~7.7]{GKO21}.

Also, it would be of interest to study the optimal dependence between $g$ and $N_{\ref{thm:main}}(g)$ in \cref{thm:main}: what is the maximum possible girth of an order-$N$ Steiner triple system? Inspecting our proof of \cref{thm:main}, it seems that one obtains a lower bound of $(\log \log N)^{c}$ for some constant $c>0$; in any case, $\log \log N$ seems to be a hard barrier for our approach. On the other hand, Lefmann, Phelps, and R\"odl \cite{LPR93} proved an upper bound of order $\log N/\log \log N$.

Finally, we remark that it may be of interest to generalize \cref{thm:main} to higher uniformities (e.g.\ to study Steiner quadruple systems instead of Steiner triple systems), or to consider triple systems other than Steiner triple systems (for example, \emph{Latin squares} can be interpreted as triple systems). Erd\H os made his conjecture only for Steiner triple systems, but generalizations to higher uniformities have been proposed by F\"uredi and Ruszink\'{o}~\cite[Conjecture~1.4]{FR13}, Glock, K\"uhn, Lo, and Osthus~\cite[Conjecture~7.2]{GKLO20}, and Keevash and Long~\cite[Section~3]{KL20}. Generalizations to Latin squares and 1-factorizations have been proposed by Linial~\cite{Lin18}. We are not aware of any fundamental obstructions that would prevent one from generalizing the methods in this paper to any of these different settings (in particular, iterative absorption can be applied in the ``multipartite'' settings of Latin squares and 1-factorizations, see \cite{BKLO17,Mon17}). However, given the level of technicality of this paper, and especially given the considerable additional technicalities that higher-uniformity generalizations would introduce, we hesitate to make any concrete claims.

\subsection*{Subsequent work}
We remark that subsequent to our work, independently Delcourt and Postle \cite{DP22} and Glock, Joos, Kim, K{\"u}hn, and Lichev \cite{GJKKL23} proved powerful new results on hypergraph matching processes with constraints, subsuming the results of \cite{GKLO20,BW19}. These results, for instance, imply the existence of approximate high-girth Steiner systems of all uniformities. (They have also  already found a number of other exciting applications in graph colouring and extremal combinatorics).

\subsection{Notation}
We use standard asymptotic notation throughout, as follows. For functions $f=f(n)$ and $g=g(n)$, we write $f=O(g)$ to mean that there is a constant $C$ such that $|f|\le C|g|$, $f=\Omega(g)$ to mean that there is a constant $c>0$ such that $f(n)\ge c|g(n)|$ for sufficiently large $n$, and $f=o(g)$ to mean that $f/g\to0$ as $n\to\infty$. Subscripts on asymptotic notation indicate quantities that should be treated as constants. Also, following \cite{Kee14}, the notation $f=1\pm\varepsilon$ means
$1-\varepsilon\le f\le1+\varepsilon$.

For a real number $x$, the floor and ceiling functions are denoted $\lfloor x\rfloor=\max\{i\in \mb Z:i\le x\}$ and $\lceil x\rceil =\min\{i\in\mb Z:i\ge x\}$. We will however mostly omit floor and ceiling symbols and assume large numbers are integers, wherever divisibility considerations are not important. All logarithms in this paper are in base $e$, unless specified otherwise.

Also, where it is not too inconvenient we make an effort to be consistent with fonts: as in \cite{GKLO20}, calligraphic letters ($\mc A,\mc B,\mc C$) will generally be used to denote sets of triangles (or sets of triples), and Fraktur letters ($\mf A,\mf B,\mf C$) will be used to denote collections of sets of triangles.

\subsection{Acknowledgements}
The authors thank Yufei Zhao for suggesting a simplification of the proof of \cref{clm:discrepancy}, Stefan Glock for some helpful discussions and clarifications on the method of iterative absorption, and Zach Hunter for pointing out a number of typographical errors. Michael Simkin thanks Nati Linial for introducing him to the problem and for many discussions on the topic of girth. Matthew Kwan thanks Benny Sudakov for introducing him to the problem and for many discussions on random processes. Finally we are grateful to the referees for their detailed comments. 

\section{Proof overview}\label{sec:overview}

We prove \cref{thm:main} by describing a probabilistic algorithm that (for fixed $g\in \mb N$) w.h.p.\footnote{We say an event holds \emph{with high probability}, or \emph{w.h.p.}\ for short, if it holds with probability $1-o(1)$ (asymptotics will be as $N\to \infty$ or as $n\to \infty$; this should be clear from context).} produces a Steiner triple system of order $N$ with girth greater than $g$. Following \cite{Kee18}, \cref{thm:counting-lower-bound} is proved by counting the number of possible paths the algorithm may take.

\subsection{A bird's-eye view of iterative absorption for Steiner triple systems}\label{sec:birds-eye}

The high-level structure of our proof (and our choice of notation, where possible) is modeled after a paper of Barber, Glock, K\"uhn, Lo, Montgomery, and Osthus~\cite{BGKLMO20}, in which the machinery of iterative absorption is applied in the special case of Steiner triple systems. For the convenience of the reader, we start by outlining the strategy used to prove \cite[Theorem~1.3]{BGKLMO20}.

Let $G \subseteq K_N$ be a graph that we wish to decompose into triangles. We assume that $G$ is \emph{triangle-divisible} (the number of edges is divisible by $3$, and every degree is even) and \emph{pseudorandom} (say, every set $A$ of at most 100 vertices has about $p^{|A|}N$ common neighbors, for some ``not-too-small'' $p$).

We begin by fixing a \textit{vortex} $V(K_N) = U_0 \supseteq U_1 \supseteq \cdots \supseteq U_\ell \eqqcolon X$, in such a way that each $G[U_k]$ is itself pseudorandom (for example, the vortex can be chosen randomly). Then, an absorber graph $H \subseteq G$ is set aside, leaving the graph $G_0= G \setminus H$. The crucial property of $H$ is that for every triangle-divisible graph $L$ on the vertex set $X$, the union $L \cup H$ admits a triangle-decomposition. In \cite{BGKLMO20}, this graph $H$ is obtained by taking a disjoint union of about $2^{|X|^2}$ ``exclusive'' absorbers, one for every possible triangle-divisible graph on the vertex set $X$. (In our proof of \cref{thm:main}, we will need a much more efficient construction).

Next is the iterative part: for $0 \leq k < \ell$, we find a set of edge-disjoint triangles $\mc M_k$ in $G_k$, covering all edges of $G_k\setminus G_k[U_{k+1}]$. Then, we let $G_{k+1}$ be the subgraph of $G_k[U_{k+1}]$ obtained by removing all edges covered by $\mc M_k$. After $\ell$ iterations, the graph of uncovered edges $G_\ell \subseteq G_0$ lies fully within the smallest vortex set $U_\ell$. By the defining property of the absorber there exists a triangle-decomposition $\mc T$ of $G_\ell \cup H$, so $\mc T \cup \bigcup_{0 \leq k < \ell} \mc M_k$ is a triangle-decomposition of $G$.

The construction of the ``exclusive absorbers'' comprising $H$ is quite ingenious, but the details are not too relevant to this outline of our proof of \cref{thm:main}. The iteration step is really the heart of the proof and consists of two components: ``cover down'' and ``regularity boosting''.

\subsubsection{Cover down}\label{subsub:gklo-cover-down}

Roughly speaking, the ``cover down'' lemma (cf.\ \cite[Lemma~3.8]{BGKLMO20}) states that if $G$ is a reasonably dense pseudorandom graph on $n$ vertices and $U \subseteq V(G)$ is a not-too-small subset of vertices, then there exists a set of edge-disjoint triangles $\mc M$ in $G$, such that all edges not covered by $\mc M$ lie inside $U$. Furthermore, $\mc M$ covers very few of the edges inside $U$, meaning that the graph of uncovered edges $L \coloneqq G \setminus E(\mc M)$ (on the vertex set $U$) remains pseudorandom. This lemma is the key ingredient used to find the sets of triangles $\mc M_k$ described above.

The cover down lemma is itself proved using a three-stage randomized algorithm. First, we use a random greedy algorithm to find an approximate triangle-decomposition of $G \setminus G[U]$. At this point, apart from the edges in $G[U]$ (which we do not need to cover) there are two types of uncovered edges: those between $V(G)\setminus U$ and $U$ (``crossing edges'') and those inside $V(G)\setminus U$ (``internal edges'').

Second, we consider another random greedy algorithm, whose purpose is to select triangles to cover the leftover internal edges. Indeed, for each uncovered internal edge $e$ (in some arbitrary order), we choose a random triangle consisting of $e$ and two uncovered crossing edges, uniformly at random.

Third, we need to cover the remaining crossing edges. Crucially, this can be viewed as a sequence of \emph{perfect matching} problems: for each vertex $v\in V(G)\setminus U$, we consider the set of vertices $w\in U$ for which $vw$ has not yet been covered; call this set $W_v$ and consider the induced graph $G[W_v]$. An edge in $G[W_v]$ corresponds to a triangle containing $v$ and two uncovered crossing edges, so we can find edge-disjoint triangles to cover the remaining crossing edges by running through each $v\in V(G)\setminus U$ (in some arbitrary order), and finding a perfect matching in $G[W_v]$ (avoiding edges   covered in previous stages of this process). For this third step we can take advantage of the well-known fact that pseudorandom graphs contain perfect matchings (as may be proved using Hall's theorem).

As a technical note, we remark that it is necessary to set aside a small random subset $R$ of the edges in $G$ between $V(G)\setminus U$ and $U$ (this graph $R$ is called the \emph{reserve graph}), before the above three-stage process. If we choose the sampling probability for $R$ appropriately, this ensures that the uncovered edges are (w.h.p.) predominantly comprised of those in $R$. This means that we will have ``plenty of room'' for the second stage, and it means that in the third stage it will be very easy to show the pseudorandomness properties that imply existence of perfect matchings.

\subsubsection{Regularity boosting}\label{subsub:gklo-regularity-boosting}
There is a subtle issue with the above ``cover down'' procedure. When using a random greedy algorithm to find an approximate triangle-decomposition of a graph $G$, we need to assume that $G$ is pseudorandom (actually, ``triangle-regularity'' suffices: we need every edge to be contained in roughly the same number of triangles). How close we are able to get to a genuine triangle-decomposition depends on how triangle-regular $G$ is. A simple calculation shows that the dependence is not in our favor: if we were to iterate the above argument na\"ively, our regularity would deteriorate very rapidly. In hindsight, this is perhaps unsurprising: it would be quite remarkable if one could do better than a single random greedy algorithm just by chaining together multiple different random greedy algorithms.

In order for the iterative absorption strategy to provide any meaningful gain, it is therefore necessary to ``boost'' the regularity before each ``cover down'' step. Specifically, we apply our random greedy algorithm restricted to a certain highly regular subset of triangles. This subset is obtained by randomly sampling from a \emph{fractional} triangle-decomposition. Such a fractional decomposition is obtained in \cite{BGKLMO20} by an ingenious ``weight-shifting'' argument.

\subsection{High-girth iterative absorption}\label{sub:overview-high-girth}

We define an \emph{Erd\H os $j$-configuration} in a triple system (or, equivalently, in a set of triangles) to be a subset of triples which span exactly $j$ vertices and have girth exactly $j$ (that is to say, an Erd\H os configuration is a $(j,j-2)$-configuration which contains no proper subconfiguration that is itself an Erd\H os configuration). Informally, we just use the term ``Erd\H os configuration'' to describe Erd\H os $j$-configurations for $5\le j\le g$, so our task is to find a triangle-decomposition of $K_N$ containing no Erd\H os configurations.

In order to prove \cref{thm:main} using the above scheme, we need each $\mc M_k$ to contain no Erd\H os configuration, and moreover we need to choose each $\mc M_k$ in such a way that its union with previous $\mc M_i$ also has no Erd\H os configuration. Beyond that, we need to choose all the $\mc M_k$ in such a way that the final ``absorber'' triangles in $\mc T$ (which themselves must contain no Erd\H os configuration) do not introduce an Erd\H os configuration at the end.

\subsubsection{Absorbers}
First, we need to choose absorbers in such a way that their corresponding triangle-decompositions never include Erd\H os configurations. This can actually be done by taking the absorbers in \cite{GKLO16} and modifying them in a black-box manner (extending the construction by copies of a gadget we call a ``$g$-sphere''; see \cref{def:sphere-cover}).

Recalling \cref{subsub:intro-absorbers}, the main innovation of our absorber construction is that we are able to choose $H$ in such a way that its number of vertices is only a polynomial function of the size of $X$ (unlike the construction in \cite{BGKLMO20}, which has an exponential dependence). Our starting point for this is that in any triangle-divisible graph $L$ (more generally, in any graph with even degrees) we can find an \emph{Euler tour} passing through all the edges, which gives us a decomposition of $G$ into cycles. By appending short paths between every pair of vertices, we can convert this initial cycle-decomposition into a cycle decomposition in which each cycle has bounded size, and if we strategically arrange ``exclusive absorbers'' of bounded size (as constructed in \cite{BGKLMO20}) on top of this construction, we can convert this cycle-decomposition into a triangle-decomposition.

That is to say, our final absorber graph $H$ is obtained by first carefully arranging a large number of short paths and small exclusive absorbers, and then extending the construction with $g$-spheres.

\subsubsection{A generalized high-girth process}\label{subsub:high-girth-process}
The majority of the triangles chosen in iterative absorption arise from instances of the first stage of ``cover down'', where we find an approximate triangle-decomposition of an $n$-vertex graph $G\setminus G[U]$. As discussed in the introduction, Glock, K\"uhn, Lo, and Osthus~\cite{GKLO20} and Bohman and Warnke~\cite{BW19} studied\footnote{Despite the fact that \cite{BW19,GKLO20} both study the same process with roughly the same methods, ideas exclusive to both of these papers will be useful to us, for different reasons.} a high-girth (i.e., Erd\H os configuration-free) random greedy process in the complete graph $K_n$; we need to adapt their analysis to our setting, in which we are given a general graph to decompose and a general set of forbidden configurations to avoid. Namely, our forbidden configurations include Erd\H os configurations, any sets of triangles that would create an Erd\H os configuration together with previous stages of the iteration, any sets of triangles that would create an Erd\H os configuration together with a potential outcome of the ``absorbing'' triangles $\mc T$, and also some additional ``regularizing'' sets of triangles (which we will discuss later).

The definition of our random greedy process is straightforward, and the same as in \cite{GKLO20,BW19}: we iteratively build a set of edge-disjoint triangles by considering at each step all those triangles consisting of yet-uncovered edges whose addition would not complete a forbidden configuration, and choosing one of these triangles uniformly at random. Though it requires some quite violent changes to the proofs in \cite{GKLO20,BW19}\footnote{For the reader familiar with the proofs in these papers, both papers depend on ``extension estimates'', taking advantage of the fact that the forbidden configurations can be interpreted as all the isomorphic copies of a particular family of triple systems.}, it is possible to permit general collections of forbidden configurations, given some statistical assumptions (which we call ``well-spreadness''). For example, just as there are at most $n^{j-3}$ Erd\H os $j$-configurations (each of which have $j-2$ triples) containing a given triangle $T$, we need to assume that there are at most $O(n^{j-3})$ forbidden configurations which have $j-2$ triples and contain $T$. More generally, we need to assume an upper bound on the number of forbidden configurations with a given number of triples which include a given set of triangles in our graph, and we also need some information about \emph{pairs} of forbidden configurations.

Of course, we also need to make some regularity assumptions: the number of forbidden configurations with a particular number of triples which include a triangle $T$ must be about the same for all $T$. Perhaps surprisingly, it is \emph{not} necessary to assume that this is true for forbidden configurations of a particular isomorphism class; we can treat all forbidden configurations with a given number of triples together, even if they have quite different structure. This is implicit in the analysis by Glock, K\"uhn, Lo, and Osthus~\cite{GKLO20} (but not the analysis by Bohman and Warnke~\cite{BW19}), and is crucial for our regularization argument, which we will discuss shortly. Under appropriate assumptions, we can find an approximate triangle-decomposition containing no forbidden configuration in which only a $n^{-\beta_g}$-fraction of edges remain uncovered, for some constant $\beta_g>0$ (this polynomial dependence on $n$ is actually rather important, and uses ideas exclusive to the analysis in \cite{BW19}).

\subsubsection{Regularizing the forbidden configurations}\label{subsub:threat-regularization}

As discussed in \cref{subsub:gklo-regularity-boosting}, it is crucial to include a regularity-boosting step between each cover down step. The regularity-boosting argument described in \cref{subsub:gklo-regularity-boosting} is still relevant, but it is only suitable for triangle-regularity (ensuring that each edge is contained in the same number of triangles). As described above, we also need our forbidden configuration data to be regular. This is basically unavoidable, because the set of ``available'' triangles that may be selected at each step evolves according to the distribution of forbidden configurations (if a triangle appears in many forbidden configurations, it is likely to quickly become unavailable).

Our regularization task can be viewed as a hypergraph regularization problem: to regularize the collection of dangerous configurations with $j-2$ triples, we consider a certain $(j-2)$-uniform hypergraph and add hyperedges to it to regularize its degrees. This can be accomplished by an appropriate random hypergraph, where the probability of selecting a hyperedge depends on the degrees of its vertices. Actually, by iterating this idea we are able to obtain a sharp general hypergraph regularization lemma that may be of independent interest.

\subsubsection{The remainder of ``cover down''} Apart from the approximate packing stage, there are two other stages in the multi-stage cover down process described in \cref{subsub:gklo-cover-down}: we need to cover leftover internal edges with a simple random greedy process, and we need to cover leftover crossing edges with a perfect matching argument. The internal edges can be handled in basically the same way as in \cite{BGKLMO20} (being careful to make choices that avoid forbidden configurations), but it is not as simple to handle the crossing edges.

Indeed, it no longer suffices to find arbitrary perfect matchings in our quasirandom graphs $G[W_v]$; we need to ensure that the union of our perfect matchings is ``well-distributed'' in that it does not interfere with the well-spreadness properties of the forbidden configurations for the next step of the iteration, and we need to avoid certain configurations of edges that would induce a forbidden configuration of triangles. Hall's theorem is a very powerful tool to show the existence of perfect matchings, but it gives almost no control over the properties of these matchings. 

The solution to this problem is to use an extra round of random sparsification: before applying Hall's theorem, we randomly sparsify our graphs $G[W_v]$, and delete any edges that could possibly contribute to forbidden configurations. Hall's theorem is sufficiently powerful that we can still find perfect matchings even after very harsh sparsification, and if we sparsify sufficiently harshly, it turns out that very few deletions will be required (this is closely related to the sparsification idea described in \cref{subsub:intro-sparsification}, which we will next discuss, and to the Brown--Erd\H os--S\'os lower bound for the $(j,j-2)$-problem).

\subsubsection{Constraint focusing and sparsification}\label{subsec:sparsification}

Recall from \cref{subsub:intro-constraint-focusing} that there is an issue of ``constraint focusing''. Specifically, there is a constant $q_g>0$ such that if we were to execute the strategy described so far on the complete graph $K_N$, at each step $k$ w.h.p.\ about a $q_g^k$-fraction of triangles in $G[U_k]$ would be excluded from consideration due to choices made at previous steps. The performance of the generalized high-girth triple process has a severe dependence on the density of available triangles, which would restrict how rapidly the sizes of our vortex sets $U_0\supseteq U_1\supseteq\cdots\supseteq U_\ell$ can decrease. This would cause a ``parameter race'': we need our vortex sets to decrease as rapidly as possible, so that we can reach a final set $U_\ell$ that is small enough for efficient absorption before the situation gets so sparse that there is simply no room to make choices. As far as we can tell this is a losing race, no matter how carefully one studies the high-girth triple process.

As discussed in \cref{subsub:intro-sparsification}, the solution is to randomly sparsify $K_N$ before we begin the iterative absorption process. Specifically, we start by finding a random set of triangles $\mc I$ with the high-girth triple process (such that the graph of uncovered edges has density $p=|X|^{-\nu}=N^{-\Omega(1)}$, for a very small constant $\nu>0$). We do this \emph{after} fixing the vortex sets $U_k$ and absorbing graph $H$; this order of operations will be important for the ``retrospective analysis'' described in the next section. Doing things in this order does however introduce some subtleties to the analysis of the high-girth triple process (in particular, we need to analyze the process on ``many different scales'' to ensure that the graph of uncovered edges in each $U_k$ is pseudorandom).

\subsubsection{Weight systems and retrospective analysis}\label{subsub:weight-systems}
In \cref{subsub:high-girth-process}, we mentioned that for the generalized high-girth process to succeed, it suffices for our collection of forbidden configurations to satisfy certain ``well-spreadness'' properties. In particular, this involves upper bounds (in terms of $j$ and $|\mc R|$) on the number of forbidden configurations with $j-2$ triples that contain a particular set of triangles $\mc R$. However, statistics of this type do \emph{not} fully capture the evolution of forbidden configurations during the high-girth triple process (in fact, it is not even true that typically all such statistics concentrate around their means).

We do believe that there exists some ensemble of statistics that does concentrate and does fully capture the behavior of the process, but it would be a formidable technical endeavor to characterize these statistics (especially since we need to work with general forbidden configurations, not just Erd\H os configurations). Of course, this raises the question of how we were able to actually analyze the generalized high-girth triple process. The answer is that it is actually not necessary to actively keep track of all relevant information. Instead, we can ``remember the randomness'' of the process: at each step, any given triangle was chosen with probability at most about $1/n^3$, and the number of steps so far was at most $n^2$. So, for any constant $h\in \mb N$, the probability that a particular set of $h$ triangles was chosen by the process is at most about $(1/n)^h$. We can use this to control high moments of various statistics, which allows one to prove probabilistic inequalities about forbidden configurations. This type of argument featured prominently in the analysis of Bohman and Warnke~\cite{BW19} (but not the analysis of Glock, K\"uhn, Lo, and Osthus~\cite{GKLO20}); actually this is the primary reason why \cite{BW19} is a shorter paper than \cite{GKLO20}, despite their main result having stronger quantitative aspects.

In our proof of \cref{thm:main}, we take this kind of retrospective analysis much further. The entire iterative absorption proof (consisting of an initial sparsification process, $\ell$ randomized regularity-boosting steps and $\ell$ ``cover down'' steps, each of which consists of multiple random processes) is viewed as one random super-process, and we keep track of the approximate probabilities of the random choices made throughout. For example, we define the ``level'' $\mr{lev}(T)$ of a triangle $T$ to be the largest $k$ for which $T$ is contained within $U_k$; it turns out that $T$ is selected by our super-process with probability at most about $p/|U_{\mr{lev}(T)}|$.

In order to prove probabilistic inequalities using this distributional information, we introduce the formalism of \emph{weight systems}, which describe the quantities that one needs to estimate in order to control a desired statistic via moment estimation. Our main lemma very closely resembles the Kim--Vu polynomial concentration inequality~\cite{KV00}, though our bounds are weaker and we do not require independence (one needs to control ``planted'' expected values, which can be interpreted as expected partial derivatives of certain polynomials). This lemma is proved using techniques related to R\"odl and Ruci\'nski's deletion method~\cite{RR95}.

\subsection{Organization of the paper} We start with some preliminaries in \cref{sec:preliminaries}; these include well-known facts in addition to our main \emph{weight systems} lemma, which will be used throughout the proof to keep track of forbidden configurations. In \cref{sec:absorbers} we present our high-girth absorber construction; this section includes our new \textit{efficient absorption} ideas. In \cref{sec:regularization} we record a slight generalization of the ``regularity boosting'' lemma from \cite{BGKLMO20}. We also prove a new general-purpose hypergraph regularization lemma in the same section. In \cref{sec:subsample} we record some general lemmas about perfect matchings in random subgraphs of typical graphs.

\cref{sec:preliminaries,sec:absorbers,sec:regularization,sec:subsample} described above are short and self-contained. From \cref{sec:spread} onwards, we get into the meat of the proof, introducing special-purpose concepts and notation related to the algorithm described in \cref{sec:overview}, and computing quantities related to this algorithm. In \cref{sec:spread} we introduce the notion of ``well-spreadness'' of a collection of forbidden configurations with respect to a vortex $U_0\supseteq \dots\supseteq U_\ell$ and prove that this well-spreadness property holds for a certain collection of forbidden configurations induced by our absorber construction. Next, in \cref{sec:weight} we compute ``weights'' associated with various arrangements of forbidden configurations. These will be inputs to our weight system lemma, and will be used to control various quantities throughout the proof. The reader may wish to skip this section on first reading, as it consists mostly of rather technical computations, the details of which are not important to understand the rest of the proof. In \cref{sec:nibble} we study a generalization of the process studied in \cite{BW19,GKLO20}, which builds an approximate triangle-decomposition avoiding a given collection of forbidden configurations. Finally, in \cref{sec:iter,sec:final} we put everything together according to the strategy described in \cref{sec:overview}, and prove \cref{thm:main,thm:counting-lower-bound}.

\section{Preliminaries}\label{sec:preliminaries}

Recall that a Steiner triple system of order $N$ is equivalent to a triangle-decomposition of $K_N$; for the rest of the paper, we will only use the language of triangles and triangle-decompositions. In this language, we recall the set of configurations we will need to avoid.

\begin{definition}
For $g\ge 5$, we define an \emph{Erd\H{o}s $g$-configuration} to be a configuration of triangles which has $g$ vertices and girth exactly $g$. That is to say, an Erd\H{o}s configuration is a ``minimal'' configuration with two more vertices than triangles.
\end{definition}
Note that by this definition, a \emph{diamond} (pair of triangles sharing an edge) is not an Erd\H{o}s $g$-configuration, and in fact the triangles in Erd\H{o}s configurations are edge-disjoint. This is non-standard, but will be convenient for us, as diamonds play a rather special role in the analysis. We see that a Steiner triple system has girth greater than $g$ if and only if it has no Erd\H{o}s $j$-configuration for any $5\le j\le g$.

The main fact that we need about Erd\H{o}s configurations is that they are ``minimal''.

\begin{lemma}\label{lem:erdos-minimality}
Let $\mc{E}$ be an Erd\H{o}s $j$-configuration.
\begin{enumerate}
    \item Every set of $2\le w\le j-3$ triangles in $\mc{E}$ spans at least $w+3$ vertices. Equivalently, every set of $1\le v\le j-4$ vertices of $\mc E$ touches at least $v+1$ of the triangles in $\mc E$.
    \item Every set of $1\le w\le j-2$ triangles in $\mc{E}$ spans at least $w+2$ vertices. Equivalently, every set of $0\le v\le j-2$ vertices of $\mc E$ touches at least $v$ of the triangles in $\mc{E}$.
\end{enumerate}
\end{lemma}
\begin{proof}
Property (2) is a direct consequence of property (1). For property (1), observe that an Erd\H{o}s $j$-configuration can be equivalently defined as a set of $j-2$ triangles on $j$ vertices, such that no subset of $3<v<j$ vertices spans at least $v-2$ triangles.
\end{proof}

For the convenience of the reader, we also recall some standard concentration inequalities. The first is Freedman's inequality. We write $\Delta X(i)=X(i+1)-X(i)$ to denote one-step differences.
\begin{lemma}[Freedman's inequality~\cite{Fre75}]
Let $(X(0),X(1),\ldots)$ be a supermartingale with respect to a filtration $\mbf F(0)\subseteq \mbf F(1)\subseteq\mbf F(2)\subseteq\cdots$. Suppose that $|\Delta X(i)|\le K$ for all $i$ and let $V(i)=\sum_{j=0}^{i-1} \mb{E}[(\Delta X(j))^2|\mbf F(j)]$. Then for any $t,v>0$ we have
\[\Pr[X(i)\ge X(0)+t\emph{ and }V(i)\le v\emph{ for some }i]\le \exp\bigg(-\frac{t^2}{2v+2Kt}\bigg).\]
\end{lemma}
We will really only use the following immediate corollary of Freedman's inequality (to deduce this, note that for any random variable $Y$, if we have $|Y|\le K$ then $\mb{E}Y^2\le K\mb{E}|Y|$).
\begin{corollary}\label{cor:freedman}
Let $(X(0),\dots, X(m))$ be a supermartingale with respect to a filtration $\mbf F(0)\subseteq \mbf F(1)\subseteq\mbf F(2)\subseteq\cdots$. Suppose that $|\Delta X(i)|\le K$ and $\mb{E}\left[\vphantom\int|\Delta X(i)|\middle|\mbf F(i)\right]\le D$ for all $i$. Then for any $t>0$ we have
\[\Pr[X(m)\ge X(0)+t]\le \exp\left(-\frac{t^2}{2mKD+2Kt}\right).\]
\end{corollary}

We also state a Chernoff bound for binomial and hypergeometric distributions (see for example \cite[Theorems~2.1 and~2.10]{JLR00}). This will be used very frequently throughout the paper.

\begin{lemma}[Chernoff bound]\label{lem:chernoff}
Let $X$ be either:
\begin{itemize}
    \item a sum of independent random variables, each of which take values in $\{0,1\}$, or
    \item hypergeometrically distributed (with any parameters).
\end{itemize}
Then for any $\delta>0$ we have
\[\Pr[X\le (1-\delta)\mb{E}X]\le\exp(-\delta^2\mb{E}X/2),\qquad\Pr[X\ge (1+\delta)\mb{E}X]\le\exp(-\delta^2\mb{E}X/(2+\delta)).\]
\end{lemma}

\subsection{Weight systems and moments}\label{sub:weight-systems}

We will need some non-standard probabilistic machinery. Specifically, at many points in this paper, it will be necessary to prove crude high-probability upper bounds on certain random variables, using estimates on their moments. In this subsection we introduce the formalism of ``weight systems'' and use it to state a general lemma which we will use for all of these moment arguments. This general lemma is closely related to the Kim--Vu polynomial concentration inequality \cite{KV00} (the quantities we need to control are almost the same), but the estimates in the proof more closely resemble the \emph{deletion method}~\cite{RR95} of R\"odl and Ruci\'nski.

\begin{definition}\label{def:weight-system}
Fix a finite set $\mc{W}$. We say that a subset of $\mc W$ is a \emph{configuration}. A \emph{weight system} $\vec{\pi}$ with ground set $\mc{W}$ is an assignment of a nonnegative real number $\pi_T$ to each $T\in\mc{W}$. Let $\mf X$ be a multiset of configurations. We define the \emph{weight} of a subset $\mc H\subseteq\mc{W}$ with respect to $\mf{X}$:
\[\psi^{(\vec{\pi})}(\mf{X},\mc H)= \sum_{\mc H\subseteq \mc S\in\mf{X}}\prod_{T\in \mc S\setminus \mc H}\pi_T,\]
where the sum over $\mc{S}$ is with multiplicity. We also define the \emph{maximum weight} of $\mf{X}$ as
\[\kappa^{(\vec{\pi})}(\mf{X}) = \max_{\mc H\subseteq\mc{W}}\psi^{(\vec{\pi})}(\mf{X},\mc H).\]
We suppress dependence on $\vec{\pi}$ where there is no ambiguity. For notational convenience we write $\psi(\mf{X},\emptyset) = \psi(\mf{X})$.
\end{definition}
Note that if $\mc H$ is larger than all the sets in $\mf X$, then $\psi(\mf X,\mc H)=0$. So, in the definition of $\kappa(\mf X)$ it suffices to take a maximum over all sets $\mc H$ whose size is at most $\max_{\mc S\in\mf{X}}|\mc S|$.

\begin{remark}
If each $\pi_T\le 1$, we can consider a random subset of $\mc W$ where each element $T\in \mc W$ is included with probability $\pi_T$ independently. Then, $\psi(\mf{X})$ is the expected number of elements of $\mf{X}$ (with multiplicity) which are fully included in this random subset. More generally $\psi(\mf X,\mc H)$ can be interpreted as a ``planted'' expectation, where we condition on the elements of $\mc H$ lying in our random subset. Or, alternatively, we can interpret these weights in the language of polynomials (as in Kim--Vu polynomial concentration). Indeed, we can encode $\mf X$ as a multilinear polynomial $f_\mf X$ in the variables $(x_T)_{T\in \mc W}$, with nonnegative integer coefficients. In this language, weights of the form $\psi(\mf{X},\mc H)$ can be interpreted as expected values arising from partial derivatives of $f_\mf X$ (where we differentiate with respect to the variables $(x_T)_{T\in \mc H}$).
\end{remark}

For now, the reader may think of $\mc W$ as being a set of triangles, so $\mf X$ is a set of configurations of triangles\footnote{This mostly describes the settings in which we will apply the machinery of weight systems, though we will sometimes consider more general situations (where $\mc W$ is a mixture of edges and triangles).}. Then, $\psi(\mf X)$ is the expected number of configurations we expect to see in an appropriate random set of triangles. However, since we are interested in proving upper bounds that hold with very high probability, we need to account for the fact that ``planting'' a small number of triangles could dramatically increase this expected value, which leads us to the definition of $\kappa(\mf X)$.

In particular, the following lemma says that if a random subset $\mc R\subseteq \mc W$ is in some sense ``bounded'' by a weight system $\vec \pi$ for $\mc W$, then the number of $\mc S\in \mf X$ which are fully included in $\mc R$ is unlikely to be much greater than $\kappa^{(\vec \pi)}(\mf X)$.

\begin{lemma}\label{lem:moments}
Fix integers $d,s\ge 1$ and a real number $C>0$. Consider a finite set $\mc W$, a multiset $\mf X$ of subsets of $\mc W$ that have size at most $d$, and a weight system $\vec \pi$ with ground set $\mc W$. Let $\mc R$ be a random subset of $\mc W$ such that
\[\Pr[\mc T\subseteq\mc R]\le C\prod_{T\in \mc T}\pi_T\]
for all $\mc T\subseteq \mc{W}$ with size at most $ds$. Let $X=X(\mf X)$ be the number of configurations $\mc S\in \mf X$ such that $\mc S\subseteq \mc R$. Then, for any $\gamma>0$ we have
\[\Pr[X\ge\gamma\kappa(\mf{X})]\le\frac{C(ds)^{ds}}{\gamma^s}.\]
\end{lemma}
\begin{proof}
It suffices to prove that 
\begin{equation}\mb{E}X^s\le C(ds)^{ds}\kappa(\mf{X})^s;\label{eq:moment-claim}\end{equation}
the desired result then follows from an application of Markov's inequality to the random variable $X^s$. To prove \cref{eq:moment-claim}, let $x_\mc S\in \mb Z_{\ge 0}$ be the multiplicity of a set $\mc S$ in $\mf X$, write
\[M_t = \sum_{\substack{\mc S_1,\ldots,\mc S_t\subseteq\mc{W}:\\\text{each }|\mc S_i|\le d}}\prod_{i=1}^tx_{\mc S_i}\prod_{T\in \mc S_1\cup\cdots\cup \mc S_t}\pi_T,\]
and note that
\[
\mb{E}X^s = \sum_{\substack{\mc S_1,\ldots,\mc S_s\subseteq\mc{W}:\\\text{each }|\mc S_i|\le d}}\mb{E}\bigg[ \mbm{1}_{\mc S_1\cup\cdots\cup \mc S_s\subseteq\mc R} \prod_{i=1}^sx_{\mc S_i}\bigg]\le CM_s.
\]
Now, we recursively bound $M_t$. Note that $M_1=\psi(\mathfrak{X})\le \kappa(\mf X)$, and for $t\ge 1$ we can sum over all the ways $\mc S_t$ can intersect $\mc S_1\cup\dots\cup \mc S_{t-1}$ to obtain
\begin{align}
M_t &=\sum_{\substack{\mc S_1,\ldots,\mc S_{t-1}\subseteq\mc{W}:\\\text{each }|\mc S_i|\le d}}\left(\prod_{i=1}^{t-1}x_{\mc S_i}\prod_{T\in \mc S_1\cup\cdots\cup \mc S_{t-1}}\pi_T\right)\left(\sum_{\mc L\subseteq \mc S_1\cup\cdots\cup \mc S_{t-1}}\sum_{\substack{\mc S_t\subseteq\mc{W}:\;|\mc S_t|\le d,\\\mc L = \mc S_t\cap(\mc S_1\cup\cdots\cup \mc S_{t-1})}}x_{\mc S_t}\prod_{T\in \mc S_t\setminus \mc L}\pi_T\right)\notag\\
&\le \sum_{\substack{\mc S_1,\ldots,\mc S_{t-1}\subseteq\mc{W}:\\\text{each }|\mc S_i|\le d}}\left(\prod_{i=1}^{t-1}x_{\mc S_i}\prod_{T\in \mc S_1\cup\cdots\cup \mc S_{t-1}}\pi_T\right)\left(\sum_{\mc L\subseteq \mc S_1\cup\cdots\cup \mc S_{t-1}}\psi(\mf{X},\mc L)\right).\label{eq:Mt}
\end{align}

Note that $|\mc{S}_1\cup\dots\cup\mc{S}_{t-1}|\le d(t-1)$, so there are at most $(d(t-1)+1)^d\le (dt)^{d}$ subsets $\mc L\subseteq \mc S_1\cup\dots\cup \mc S_{t-1}$ which have size at most $d$ (if $|\mc L|>d$ then $\psi(\mf X,\mc L)=0$). Also, recall that $\psi(\mf{X},\mc L)\le \kappa(\mf X)$, so $M_t\le (dt)^d\kappa(\mf X)M_{t-1}$, from which we deduce that $M_s\le (ds)^{ds}\kappa(\mf{X})^s$. Then, \cref{eq:moment-claim} follows.
\end{proof}

We also state a convenient asymptotic corollary of \cref{lem:moments}.

\begin{corollary}\label{lem:moments-asymptotic}
Consider a finite set $\mc W$ of size at most $n^{O(1)}$ and a multiset $\mf X$ of subsets of $\mc W$ each of which has size $O(1)$. Consider a weight system $\vec \pi$ with ground set $\mc W$ and let $\mc R$ be a random subset of $\mc W$ such that
\[\Pr[\mc T\subseteq\mc R] = O\bigg(\prod_{T\in \mc T}\pi_T+n^{-\omega(1)}\bigg)\]
for all $\mc T\subseteq \mc{W}$. Let $X=X(\mf X)$ be the number of configurations $\mc S\in \mf X$ such that $\mc S\subseteq \mc R$. Then, with probability $1-n^{-\omega(1)}$ we have $X\le n^{o(1)}\kappa(\mf X)$.
\end{corollary}
\begin{proof}
Note that one may replace each $\pi_T$ with $\on{min}(\pi_T,1)$. Indeed, given a configuration $\mc T$, if $\mc{S} = \{ T \in \mc T : \pi_T \leq 1 \}$ then we have
\[\Pr[\mc T\subseteq\mc R]\le \Pr[\mc S\subseteq\mc R] = O\bigg(\prod_{T\in \mc S}\pi_T+n^{-\omega(1)}\bigg) = O\bigg(\prod_{T\in \mc T}\min(\pi_T,1)+n^{-\omega(1)}\bigg).\]
Hence, we may (and do) assume that $\pi_T\le 1$. We also assume that $\mf X$ is nonempty; otherwise the statement is vacuous. 

Unwinding the asymptotic notation, there exists an absolute constant $K$ such that $|\mc W|\le (2n)^{K}$ and all configurations in $\mf X$ have size at most $K$. Take any integer $t\ge 1$. There is a constant $C_{K,t}>0$ such that 
\[\Pr[\mc T\subseteq\mc R]\le C_{K,t} \prod_{T\in \mc{T}} (\pi_T + (2n)^{-K^2})\]
for all $|T|\le Kt^2$. Taking the modified weight system $\pi'_T = \pi_T + (2n)^{-K^2}$, we have that 
\begin{align*}
\psi^{(\vec{\pi}')}(\mf X,\mc H) &= \sum_{\mc H\subseteq \mc S\in\mf{X}}\prod_{T\in \mc S\setminus \mc H}(\pi_T + (2n)^{-K^2})\le \psi^{\pi}(\mf X,\mc H) + (2n)^{-K^2}\sum_{\mc H\subsetneq \mc H'\subseteq \mc S\in\mf{X}}\prod_{T\in \mc S\setminus \mc H'}\pi_T\\
&\le \psi^{(\vec{\pi})}(\mf X,\mc H) + (2n)^{-K^2}\sum_{\substack{\mc H'\subseteq \mc S\in\mf{X}\\ 0<|\mc H'| \leq K}}\prod_{T\in \mc S\setminus \mc H'}\pi_T\le \psi^{(\vec{\pi})}(\mf X,\mc H) + (2^{K}-1)|\mc W|^{K}(2n)^{-K^2}\kappa^{(\vec{\pi})}(\mf X)\\
&\le 2^{K}\kappa^{(\vec{\pi})}(\mf X)=2^K\kappa(\mf X).
\end{align*}
This implies that $\kappa^{(\vec{\pi}')}(\mf X) \leq 2^K \kappa(\mf X)$. Hence, applying \cref{lem:moments} with $d=K$ and $s = t^2$, we have 
\[\mb{P}[X\ge 2^{K}n^{1/t}\kappa(\mf X)]\le\mb{P}[X\ge n^{1/t}\kappa^{(\vec{\pi}')}(\mf X)]\le C_{K,t}(Kt^2)^{Kt^2}n^{-t}.\]
As this holds for any fixed integer $t\ge 1$, the desired result follows. 
\end{proof}

We conclude the section with a simple application of \cref{lem:moments}, corresponding to the special case where $\mf X$ is a collection of singleton sets and the weights are all equal (our weight systems machinery is not really necessary for such a simple case).

\begin{corollary}\label{cor:moments-asymptotic-simple}
Consider a finite set $\mc W$ of size at most $n^{O(1)}$, and let $\mc R$ be a random subset of $\mc W$, such that for every choice of distinct $w_1,\dots,w_t\in \mc W$ we have $\Pr[w_1,\dots,w_t\in \mc R]\le p^t+n^{-\omega(1)}$. Then, with probability $1-n^{-\omega(1)}$ we have $|\mc R|\le n^{o(1)}\max(p|\mc W|,1)$.
\end{corollary}

\section{Efficient High-Girth Absorbers}\label{sec:absorbers}

Recall that a graph is \emph{triangle-divisible} if all its degrees are even and its number of edges is divisible by $3$. This is a necessary but insufficient condition for triangle-decomposability. In this section we explicitly define a high-girth ``absorbing structure'' that will allow us to find a triangle-decomposition extending any triangle-divisible graph on a specific vertex set. Importantly, the size of this structure is only polynomial in the size of the distinguished vertex set, which is needed for our proof strategy (recall the discussion in \cref{subsub:intro-absorbers}). For a set of triangles $\mc R$, let $V(\mc R)$ be the set of all vertices in these triangles.

The following theorem encapsulates the properties of our absorbing structure. Apart from the fact that we can find a high-girth triangle-decomposition extending any triangle-divisible graph on a specific vertex set, we also need a rather technical property (\cref{AB2}) stating that the triangle-decompositions can be chosen in a way such that there are only a small number of ways for Erd\H{o}s configurations to intersect the triangles in these decompositions. This is needed so we can ensure that at the very end, when the absorber is used to complete the final high-girth system, no Erd\H{o}s configurations are introduced.

\begin{theorem}\label{thm:absorbers}
There is $C_{\ref{thm:absorbers}}\in\mb{N}$ so that for $g\in\mb{N}$ there exists $M_{\ref{thm:absorbers}}(g)\in\mb{N}$ such that the following holds. For any $m\ge 1$, there is a graph $H$ with at most $M_{\ref{thm:absorbers}}(g)m^{C_{\ref{thm:absorbers}}}$ vertices containing a distinguished independent set $X$ of size $m$ and satisfying the following properties.
\begin{enumerate}[{\bfseries{Ab\arabic{enumi}}}]
    \item\label{AB1} For any triangle-divisible graph $L$ on the vertex set $X$ there exists a triangle-decomposition $\mc S_L$ of $L\cup H$ which has girth greater than $g$.
    
    \item\label{AB2} Let $\mc{B} = \bigcup_L\mc{S}_L$ (where the union is over all triangle-divisible graphs $L$ on the vertex set $X$) and consider any graph $K$ containing $H$ as a subgraph. Say that a triangle in $K$ is \emph{nontrivially $H$-intersecting} if it is not one of the triangles in $\mc{B}$, but contains a vertex in $V(H)\setminus X$.
    
    Then, for every set of at most $g$ triangles $\mc{R}$ in $K$, there is a subset $\mc{L}_{\mc{R}}\subseteq\mc{B}$ of at most $M_{\ref{thm:absorbers}}(g)$ triangles such that every Erd\H{o}s configuration $\mc{E}$ on at most $g$ vertices which includes $\mc{R}$ must either satisfy $\mc{E}\cap\mc{B}\subseteq\mc L_\mc R$ or must contain a nontrivially $H$-intersecting triangle $T\notin\mc R$.
\end{enumerate}
\end{theorem}
\begin{remark}
Note that \cref{AB1}, with $L$ as the empty graph on the vertex set $X$, implies that $H$ itself is triangle-decomposable (hence triangle-divisible). In \cref{AB2}, note that each triangle-decomposition $\mc{S}_L$ is viewed as a collection of triangles, and $\mc{B}$ is therefore also a collection of triangles.
\end{remark}

We will prove \cref{thm:absorbers} by chaining together some special-purpose graph operations.

\begin{definition}[Path-cover]\label{def:path-cover}
Let the \emph{path-cover} $\wedge X$ of a vertex set $X$ be the graph obtained as follows. Start with the empty graph on $X$. Then, for every unordered pair of $u,v$ of distinct vertices in $X$, add $6|X|^2$ new paths of length $2$ between $u$ and $v$, introducing a new vertex for each (so in total, we introduce $6|X|^2\binom{|X|}{2}$ new vertices). We call these new length-$2$ paths \emph{augmenting paths}.
\end{definition}

The key point is that for any graph $L$ on the vertex set $X$ with even degrees, the edges of $L\cup\wedge X$ can be decomposed into short cycles.

\begin{lemma}\label{lem:path-cover}
If a graph $L$ on $X$ has even degrees, then $L\cup\wedge X$ can be decomposed into cycles of length at most $5$, such that there are more cycles of length $4$ than of length $5$. Additionally, if $L$ is triangle-divisible, then so is $L\cup\wedge X$.
\end{lemma}
\begin{proof}
By finding an Euler tour in each connected component of $L$, we decompose $L$ into edge-disjoint cycles. There are fewer than $|X|^2$ cycles in this decomposition. Now, consider one of these cycles $C$, and write $v_1,\ldots,v_\ell$ for its vertices. For each $2\le i\le \ell-1$, add two of the augmenting paths between $v_1$ and $v_i$, to obtain an augmented graph $C'\subseteq L\cup\wedge X$ containing $C$. By design, $C'$ has an edge-decomposition into cycles of lengths $3,5,5,\ldots,5,3$. We can do this edge-disjointly for each of the cycles in the cycle-decomposition of $L$, since each cycle only uses at most two augmenting paths between any pair. The remaining unused edges of $L\cup\wedge X$ consist of augmenting paths between vertices. There are an even number of such augmenting paths between each pair of vertices since the above edge-decompositions involve adding pairs of augmenting paths at a time. Therefore, we can decompose the remaining edges into cycles of length $4$. Since we included sufficiently many augmenting paths in $\wedge X$, it follows that a majority of cycles in our final decomposition have length $4$.

Finally, the last statement follows from the fact that $\wedge X$ is triangle-divisible.
\end{proof}

It is proved in \cite[Lemma~3.2]{BGKLMO20} that for any triangle-divisible graph $H$, there is a graph $A(H)$ containing the vertex set of $H$ as an independent set such that $A(H)$ and $A(H)\cup H$ are both triangle-decomposable.

\begin{definition}[Cycle-cover]\label{def:cycle-cover}
Let $\mc{H}$ be the (finite) set of all graphs which are a triangle, or whose edges decompose into one cycle of length $4$ and one cycle of length $5$, or into three cycles of length $4$. All such graphs are triangle-divisible (but might not be triangle-decomposable).

Let the \emph{cycle-cover} of a vertex set $Y$ be the graph $\triangle Y$ obtained as follows. Beginning with $Y$, for every $H\in\mc{H}$ and every injection $f:V(H)\to Y$ we add a copy of $A(H)$, such that each $v\in V(H)\subseteq V(A(H))$ in the copy of $A(H)$ coincides with $f(v)$ in $Y$. We do this in a vertex-disjoint way, introducing $|V(A(H))\setminus V(H)|$ new vertices each time. (Think of the copy of $A(H)$ as being ``rooted'' on a specific set of vertices in $Y$, and otherwise being disjoint from everything else.)
\end{definition}

\begin{lemma}\label{lem:absorber-cycle-cover}
Let $X$ be a set of vertices and $Y = V(\wedge X)$ be the vertex set of the graph $\wedge X$. If a graph $L$ on $X$ is triangle-divisible, then $L\cup\wedge X\cup\triangle Y$ admits a triangle-decomposition.
\end{lemma}
\begin{proof}
By \cref{lem:path-cover}, $L\cup\wedge X$ is triangle-divisible and can be decomposed into cycles of length $3$, $4$, and $5$, with more cycles of length $4$ than $5$. We group the cycles into individual triangles, pairs of a $4$-cycle and $5$-cycle, and triples of $4$-cycles, all of which are contained within $Y$. (Triangle-divisibility of $L \cup \wedge X$ implies that there will be no leftover $4$-cycles.) To obtain a triangle-decomposition we use the defining property of the graphs $A(H)$ and the definition of $\triangle Y$: for every $H\in\mc{H}$ and injection $f:V(H)\to Y$, decompose the corresponding image of $A(H)\cup H$ into triangles if $f(H)$ comprises one of our groups, and otherwise decompose the corresponding image of $A(H)$ into triangles.
\end{proof}

If $|X| = m$, then \cref{lem:absorber-cycle-cover} implies that $\wedge X\cup\triangle(V(\wedge X))$ can ``absorb'' any triangle-divisible graph on $X$, though not necessarily in a high-girth manner. The next transformation will provide us with a girth guarantee.

\begin{definition}[Sphere-cover]\label{def:sphere-cover}
Let the \emph{$g$-sphere-cover} of a vertex set $Z$ be the graph $\medcircle_gZ$ obtained by the following procedure. For every triple $T$ of distinct vertices of $Z$, arbitrarily label these vertices as $a,b_1,b_2$. Then, append a ``$g$-sphere'' to the triple. Namely, first add $2g-1$ new vertices $b_3,\ldots,b_{2g},c$. Then add the edges $ab_j$ for $3\le j\le 2g$, the edges $cb_j$ for $1\le j\le 2g$, the edges $b_jb_{j+1}$ for $2\le j\le 2g-1$, and the edge $b_{2g}b_1$.

Note that every such $g$-sphere $Q$ itself has a triangle-decomposition: specifically, we define the \emph{out-decomposition} to consist of the triangles
\[c b_2 b_3,\,ab_3 b_4,\, c b_4b_5,\, a b_5 b_6,\dots,cb_{2g}b_1.\]
We also identify a particular triangle-decomposition of the edges $Q\cup T$: the \emph{in-decomposition} consists of the triangles 
\[cb_1b_2,\,ab_2b_3,\,cb_3b_4,\,a b_4 b_5,\dots,ab_{2g}b_1.\]
For a triple $T$ in $Z$, let $\mc{B}^\medcircle(T)$ be the set of all triangles in the in- and out-decompositions of the $g$-sphere associated to $T$. We emphasize that $T\notin\mc{B}^\medcircle(T)$.
\end{definition}

\begin{remark}
The $g$-sphere-cover depends on an ordering of every triple of vertices in $Z$, so is technically not unique, but this is of little consequence to us. Also, we remark that the object we call a $g$-sphere is usually called a \emph{cycle of length $2g$} in design theory; we prefer to avoid this terminology as our construction also involves graph cycles.
\end{remark}

The key properties of the $g$-sphere-cover are encapsulated in the following lemma.

\begin{lemma}\label{obs:erdos-in-sphere}
Fix an integer $g>2$. Consider a vertex set $Z$ and its $g$-sphere-cover $\medcircle_gZ$, and a triple of vertices $T$ in $Z$. Let $\mc{E}\subseteq\mc{B}^\medcircle(T)$ be a non-empty set of at most $g$ edge-disjoint triangles.
\begin{enumerate}
    \item There is $v\in V(\mc{B}^\medcircle(T))\setminus V(T)$ which appears in exactly one of the triangles in $\mc{E}$.
    \item $|V(\mc{E})\setminus V(T)|\ge|\mc{E}|$.
\end{enumerate}
\end{lemma}

\begin{proof}
Write $a,b_1,b_2$ for the three vertices of $T$, and let $b_3,\dots,b_{2g},c$ be as in \cref{def:sphere-cover}.

Consider the cycle graph $C_{2g}$ on the vertices $b_1,\dots,b_{2g}$ (running through these vertices in order). By inspecting \cref{def:sphere-cover}, we see that every triangle in $\mc{E}\subseteq\mc{B}^\medcircle(T)$ contains an edge of $C_{2g}$, and no edge of $C_{2g}$ appears in more than one triangle of $\mc{E}$ (since the triangles in $\mc{E}$ are edge-disjoint). So, we may identify each triangle of $\mc{E}$ with a unique edge of $C_{2g}$, to give a nonempty subgraph $S\subseteq C_{2g}$.

For (1), we observe that $S$ has at least two degree-$1$ vertices (since $g<2g$). If the only degree-$1$ vertices of $S$ are $b_1$ and $b_2$, then $\mc{E} = \{cb_1b_2\}$ (since $g < 2g-1$), in which case $c$ satisfies the desired conclusion. Otherwise, there is some degree-$1$ vertex different from $b_1,b_2$ in $S$, and this vertex satisfies the desired conclusion.

For (2), note that $S$ has more vertices than edges, so the desired result follows immediately unless $S$ contains both $b_1$ and $b_2$. But if $S$ contains both these vertices then either it is disconnected (in which case it contains at least two more vertices than edges), or $S$ contains the edge $b_1b_2$ (in which case $\mc{S}$ additionally contains $c$). In all cases, we are done.
\end{proof}

\cref{obs:erdos-in-sphere}(1) implies that $\medcircle_gZ$ can ``transform'' arbitrary triangle-decompositions on $Z$ into high-girth triangle-decompositions, as follows.

\begin{lemma}\label{lem:absorber-sphere}
Let $\mc{C} = \bigcup_{T\in\binom{Z}{3}}\mc{B}^\medcircle(T)$. 
If a graph $L$ on the vertex set $Z$ is triangle-decomposable, then $L\cup\medcircle_gZ$ has a triangle-decomposition using only triangles of $\mc C$, with girth greater than $g$.
\end{lemma}
\begin{proof}
We have a triangle-decomposition of $L$. Classify every triple of vertices in $L$ as ``in'' or ``out'' depending on whether it is in the triangle-decomposition or not. Let $Q(T)$ be the $g$-sphere associated with a triple $T$.

For every ``in'' triple $T$, remove it from our decomposition and instead cover the edges of $T\cup Q(T)$ using the in-decomposition of $Q(T)$. For every ``out'' triple $T$, cover the edges of $Q(T)$ using the out-decomposition of $Q(T)$. Either way, we have only used triangles in $\mc{B}^\medcircle(T)$.

We claim that $\mc{C}$ contains no Erd\H{o}s $j$-configuration $\mc{E}$ for $5\le j\le g$. Suppose for the sake of contradiction that there were such a configuration $\mc E\subseteq \mc C$. Observe that $\mc{E}$ cannot have a vertex which appears in exactly one triangle (recall \cref{lem:erdos-minimality}(1) and that $j\ge 5$). Consider some triple $T$ in $Z$ such that $\mc{E}\cap\mc{B}^\medcircle(T)$ is nonempty. Erd\H{o}s configurations contain edge-disjoint triangles, so \cref{obs:erdos-in-sphere}(1) implies there is a vertex $v\in V(\mc{B}^\medcircle(T))\setminus V(T) = V(\mc{B}^\medcircle(T))\setminus Z$ in exactly one triangle of $\mc{E}\cap\mc{B}^\medcircle(T)$.

However, $\mc{E}\subseteq\mc{C}$ and thus the triangles in $\mc{E}\setminus\mc{B}^\medcircle(T)$ are vertex-disjoint from $V(\mc{B}^\medcircle(T))\setminus Z$ by construction. We conclude that $\mc{E}$ has exactly one triangle containing $v$. This contradicts our earlier observation!
\end{proof}

We are ready to prove \cref{thm:absorbers}.

\begin{proof}[Proof of \cref{thm:absorbers}]
We start with a vertex set $X$ of size $m$. Let $Y = V(\wedge X)$ and $Z = V(\triangle Y)$, and take $H=\wedge X\cup\triangle Y\cup\medcircle_gZ$. \cref{AB1} is a direct consequence of \cref{lem:absorber-cycle-cover,lem:absorber-sphere}. It is also easy to see that $H$ has the claimed size, and that $\mc{B}\subseteq \bigcup_{T\in \binom Z 3} \mc{B}^\medcircle(T)$.

Next, we consider \cref{AB2}. For a triple of vertices $T$ in $Z$, recall that $\mc{B}^\medcircle(T)$ is the collection of all triangles in the in-decomposition and the out-decomposition associated with the $g$-sphere on $T$, and let $W(T)=V(\mc{B}^\medcircle(T))\setminus V(T)\subseteq V(H)\setminus Z$. Note that the $\mc{B}^\medcircle(T)$ are disjoint from each other (having no triangles in common), and the $W(T)=V(\mc{B}^\medcircle(T))\setminus V(T)$ are vertex-disjoint from each other.

Now, given a collection of at most $g$ triangles $\mc{R}$ in $K$, let $\mc{T}_{\mc{R}}$ be the collection of triples $T$ within $Z$ such that $W(T)$ shares a vertex with $\mc{R}$. Let $\mc{L}_{\mc{R}}=\bigcup_{T\in\mc{T}_{\mc{R}}}\mc{B}^\medcircle(T)$. We claim that $\mc{L}_{\mc{R}}$ satisfies the desired property, noting that its size is bounded by a constant depending only on $g$.

Consider an Erd\H{o}s $j$-configuration $\mc{E}$, for $5\le j\le g$ (consisting of triangles in $K$), which includes $\mc{R}$ and contains some $T^\ast\in\mc{B}\setminus\mc{L}_{\mc{R}}$. Suppose for the sake of contradiction that every $T'\in\mc{E}\setminus(\mc{R}\cup\mc{B})$ is disjoint from $V(H)\setminus X$. There is a unique triple $T_0$ within $Z$ with $T^\ast\in\mc{B}^\medcircle(T_0)$. Let $\mc{K} = \mc{E}\cap\mc{B}^\medcircle(T_0)$. Note that $\mc K$ is nonempty (as $T^\ast\in\mc{K}$), and it contains at most $g$ triangles, all edge-disjoint, so by \cref{obs:erdos-in-sphere}(1) there is $v\in W(T_0)$ which appears in exactly one of the triangles in $\mc{K}$. Let $\mc{W}\subseteq \mc E$ be the subset of triangles of $\mc{E}$ which contain $v$.

By assumption, $v\in V(H)\setminus X$ cannot be a vertex in any of the triangles in $\mc{E}\setminus(\mc{R}\cup\mc{B})$, so $\mc{W}\subseteq\mc{E}\cap(\mc{R}\cup\mc{B})$. If $v$ was in one of the triangles in $\mc{R}$, then by definition we would have $T_0\in\mc{T}_{\mc{R}}$, which would imply $T^\ast\in\mc{L}_{\mc{R}}$, a contradiction. Hence we deduce $\mc{W}\subseteq\mc{E}\cap \mc{B}$. But now note that all of the triangles in $\bigcup_{T\neq T_0}\mc{B}^\medcircle(T)$ have all their vertices in $Z\cup\bigcup_{T\neq T_0}W(T)$, which is disjoint from $W(T_0)$ by construction. So $\mc{W}\subseteq\mc{E}\cap\mc{B}^\medcircle(T_0) = \mc{K}$. Finally, $v$ is in exactly one triangle of $\mc{K}$, so we deduce that $|\mc{W}| = 1$. This contradicts \cref{lem:erdos-minimality}(1) (recalling that $j\ge 5$): every vertex of an Erd\H{o}s configuration must be in at least two triangles.
\end{proof}

\section{Regularization}\label{sec:regularization}

\subsection{Triangle regularization}\label{sub:triangle-regularization}

Given a set of triangles with suitable regularity and ``extendability'' properties, the following lemma finds a subset which is substantially more regular. This lemma is closely related to \cite[Lemma~4.2]{BGKLMO20} and the proof is almost the same (the idea is to sample from a fractional triangle-decomposition), however we need to make certain adjustments in our setting.

\begin{lemma}\label{lem:reg-deg}
There is $n_{\ref{lem:reg-deg}}\in\mb{N}$ such that the following holds for all $n\ge n_{\ref{lem:reg-deg}}$. Suppose $C\ge 2$ and $p\in (n^{-1/6},1)$, let $G$ be a graph on $n$ vertices and let $\mc{T}$ be a collection of triangles of $G$, satisfying the following properties.
\begin{enumerate}
    \item Every edge $e\in E(G)$ is in $(1\pm 1/(12C^5))p^2n$ triangles of $\mc{T}$.
    
    \item\label{item:extension bounds} For every set $S\subseteq V(G)$ with $2\le |S|\le 4$ forming a clique in $G$, there are between $C^{-1}p^{|S|}n$ and $Cp^{|S|}n$ vertices $u\in V(G)\setminus S$ which form a triangle in $\mc{T}$ with every distinct pair $v,w\in S$.
\end{enumerate}
Then, there is a subcollection $\mc{T}'\subseteq\mc{T}$ such that every edge $e\in E(G)$ is in $(1\pm n^{-1/4})p^2n/4$ triangles of $\mc{T}'$.
\end{lemma}
\begin{proof}
Let $\mc{T}_5$ be the collection of all copies of $K_5$ in $G$, such that all triangles in this $K_5$ are present in $\mc{T}$. For every edge $e$ in $E(G)$ let $\mc{T}(e)\subseteq\mc{T}$ be the subcollection of triangles in $\mc{T}$ which contain $e$, and let $\mc{T}_5(e)\subseteq\mc{T}_5$ be the subcollection of 5-cliques in $\mc{T}_5$ which contain $e$.

Also, let
\[c_e = \frac{p^2n-|\mc{T}(e)|}{|\mc{T}_5(e)|}.\]
Iteratively applying assumption \ref{item:extension bounds} in the lemma statement (and then compensating for double-counting), we see that $|\mc{T}_5(e)|\ge (C^{-1}p^2n)(C^{-1}p^3n)(C^{-1}p^4n)/6$ for all $e\in E(G)$, implying that
\[|c_e|\le\frac{p^2n/(12C^5)}{C^{-3}p^9n^3/6}\le\frac{1}{2C^2p^7n^2}.\]
For $e\in E(G)$ and $J\in\mc{T}_5(e)$, define $\psi_{e,J}\colon\mc{T}\to\mb{R}$ by
\[\psi_{e,J}(T) = \begin{cases}-1/6&\text{if }|V(T)\cap e| = 1\text{ and }T\subseteq J,\\1/3&\text{if }|V(T)\cap e| \ne 1\text{ and }T\subseteq J,\\0&\text{if }T\not\subseteq J.\end{cases}\]
The key point of this construction is that it provides a simple encoding of certain delta functions: for any $e,e'\in E(G)$ and $J\in\mc{T}_5(e)$, one can check (via case analysis) that
\[\sum_{T\in\mc{T}(e')}\psi_{e,J}(T) = \mbm{1}_{e' = e}.\]

Now, define $\psi\colon\mc{T}\to\mb{R}$ by
\[\psi(T) = \frac{1}{4} + \frac{1}{4}\sum_{e\in E(G)}c_e\sum_{J\in\mc{T}_5(e)}\psi_{e,J}(T).\]
For a given $T\in\mc{T}$, there are at most $(Cp^3n)(Cp^4n)/2$ many $J\in\mc{T}_5$ such that $T\subseteq J$ (to see this, iteratively apply assumption \ref{item:extension bounds} in the lemma statement). We deduce that $\psi(T)\in[0,1]$, since
\begin{align*}
|\psi(T)-1/4|&\le\frac{1}{4}\sum_{e\in E(G),J\in\mc{T}_5(e)}|c_e||\psi_{e,J}(T)|\le\frac{1}{4}\sum_{e\in E(G),J\in\mc{T}_5(e):T\subseteq J}\frac{1}{2C^2p^7n^2}\cdot\frac{1}{3}\\
&\le\Big(\frac{1}{2}C^2p^7n^2\Big)\cdot\binom{5}{2}\cdot\frac{1}{24C^2p^7n^2}\le\frac{1}{4}.
\end{align*}
Also, for any $e\in E(G)$ we have
\begin{align*}
\sum_{T\in\mc{T}(e)}\psi(T) &= \frac{1}{4}|\mc{T}(e)| + \frac{1}{4}\sum_{e'\in E(G)}c_{e'}\sum_{J\in\mc{T}_5(e')}\sum_{T\in\mc{T}(e)}\psi_{e',J}(T)\\
&= \frac{1}{4}|\mc{T}(e)| + \frac{1}{4}\sum_{e'\in E(G)}c_{e'}\sum_{J\in\mc{T}_5(e')}\mbm{1}_{e=e'} \;=\; \frac{1}{4}|\mc{T}(e)|+\frac{1}{4}c_e|\mc{T}_5(e)|.
\end{align*}
By the definition of $c_e$, it follows that
\[\sum_{T\in\mc{T}(e)}\psi(T) = \frac{1}{4}p^2n.\]

Now, we define a random subset $\mc T'\subseteq \mc T$ by including each $T\in\mc{T}$ with probability $\psi(T)$ independently. For every $e\in E(G)$, the expected number of $T\in\mc{T}'$ containing $e$ is exactly $p^2n/4$. By a Chernoff bound, $e$ is in $(1\pm n^{-1/4})p^2n/4$ of the triangles in $\mc{T}'$, with probability at least $1-\exp(-\Omega(n^{-1/2}\cdot p^2n))$. Since $p\ge n^{-1/6}$, a union bound over at most $\binom{n}{2}$ edges $e$ concludes the proof.
\end{proof}

\subsection{General hypergraph regularization}\label{sub:threat-regularization}
In this section we prove a general lemma stating that for any not-too-dense hypergraph $\mf G$, we can add a small number of random hyperedges to $\mf G$ to obtain a nearly-regular hypergraph. Moreover, we can specify a sparse hypergraph $\mf H$ of hyperedges we would like to avoid adding. For the proof of \cref{thm:main}, we will apply this lemma to an auxiliary hypergraph $\mf G$ whose vertices correspond to triangles in some graph $G$, and whose hyperedges correspond to forbidden configurations of triangles.

In the following lemma statement, the term ``degree'' always refers to the degree of a vertex, i.e., the number of hyperedges containing that vertex (in hypergraphs there are ``higher-order'' notions of degree, which will not concern us here).

It strikes us that this lemma may be of independent interest, so we prove a stronger lemma statement than is strictly necessary for the proof of \cref{thm:main} (in particular, the proof is iterative, but only one step of this iteration is actually necessary for us).

\begin{lemma}\label{lem:reg-threat}
Fix a constant $k\in\mb{N}$. Consider $n$-vertex $k$-uniform hypergraphs $\mf G, \mf H$, such that $\mf G\subseteq \mf H$ and such that the maximum degree of $\mf H$ is at most $\binom{n-1}{k-1}/(36\cdot 2^k)$. Let $d_{\max}$ be the maximum degree of $\mf G$. Let $\mb G^{(k)}(n,p)$ be a random $k$-uniform hypergraph on the same vertex set as $\mf G$, where each of the $\binom n k$ possible edges are present with probability $p$ independently. There is a random simple $k$-uniform hypergraph $\mf G'$, containing no edge of $\mf H$, satisfying the following properties.
\begin{enumerate}
    \item[(1)] The maximum degree of $\mf G\cup \mf G'$ is at most $9d_{\max}$ with probability $1-n^{-\omega(1)}$;
    \item[(2)] The vertex degrees of $\mf G\cup \mf G'$ all differ by at most $(\log n)^2$ with probability $1-n^{-\omega(1)}$;
    \item[(3)] $\mf G'$ can be coupled as a subgraph of $\mb G^{(k)}(n,p)$, where $p=2\cdot2^kd_{\max}/\binom{n-1}{k-1}$.
\end{enumerate}
\end{lemma}

\begin{proof}
We will iteratively construct a sequence of random hypergraphs $\mf G(1),\ldots,\mf G(\tau)$, where the number of hypergraphs $\tau\ge 0$ is itself random (we can think of $\tau$ as a stopping time). Specifically, given outcomes of $\mf G(1),\ldots,\mf G(t)$, for some $t\ge 0$, either we will take $\tau=t$ and stop, or we will define a random hypergraph $\mf G(t+1)$ whose distribution depends on the outcomes of $\mf G(1),\dots,\mf G(t)$. This distribution will be chosen in such a way as to typically reduce the difference between the maximum and minimum degree, without introducing too many new edges. At the end, we will take $\mf G'=(\mf G(1)\cup\cdots\cup \mf G(\tau))\setminus \mf H$.

So, consider outcomes of $\mf G(1),\dots,\mf G(t)$. Write $F(t)$ for the difference between the maximum and the minimum degrees of $\mf G\cup((\mf G(1)\cup\cdots\cup \mf G(t))\setminus \mf H)$ (so $F(0)$ is the difference between the maximum and minimum degrees of $\mf G$). If one of the following three criteria is met, we take $\tau=t$ and stop.
\begin{enumerate}
    \item[(A)] $F(t) \le (\log n)^2$;
    \item[(B)] $F(t)\ge F(t-1)/2$;
    \item[(C)] $\mf G(t)$ has maximum degree greater than $4F(t-1)$.
\end{enumerate}
The idea is that if (A) occurs then we have successfully regularized our hypergraph, while if (B) or (C) occurs then we have failed to sufficiently reduce $F(t)$ or have introduced too many new hyperedges (in which case we give up; we will show that this is unlikely). Note that the process cannot continue for more than $\lceil\log_2(d_{\max}/(\log n)^2)\rceil$ steps (before this time, either (A) or (B) must occur).

If none of (A)--(C) hold at step $t$, then we define the next hypergraph $\mf G(t+1)$ as follows. For each vertex $v$ let $d_v(t)$ be the degree of $v$ in $\mf G\cup((\mf G(1)\cup\cdots\cup \mf G(t))\setminus \mf H)$ and let $d_{\max}(t) = \max_{v\in V(\mf G)}d_v(t)$. For each vertex $v$, define its ``weight'' $w_v = d_{\max}(t)+F(t)-d_v(t)$. Hypothetically speaking, if we were able to choose $\mf G(t+1)$ in such a way that the degree of each vertex $v$ in $\mf G(t+1)\setminus (\mf G(1)\cup\cdots\cup \mf G(t)\cup \mf H)$ were precisely equal to $w_v$, then $\mf G\cup((\mf G(1)\cup\cdots\cup \mf G(t+1))\setminus \mf H)$ would be regular, i.e., $F(t+1)=0$. We will not quite be able to manage this, but we will be able to get close by randomly sampling according to the weights $w_v$.

Note that $F(t)\le w_v\le 2F(t)$, so the weights of different vertices differ by a multiplicative factor of at most $2$. Define $W = \sum_{S\in\binom{V(\mf G)}{k-1}}\prod_{v\in S}w_v$ (where $\binom{V(\mf G)}{k-1}$ denotes the collection of all $(k-1)$-sets of vertices in $\mf G$), and define the next random hypergraph $\mf G(t+1)$ by including each $k$-set of vertices $S$ with probability $\left( \prod_{v\in S}w_v \right) /W\le 2^kF(t)/\binom{n-1}{k-1}$ independently.

We now prove (3). Note that, conditional on outcomes of $\mf G(1),\dots,\mf G(t)$ such that (A)--(C) do not hold, the random hypergraph $\mf G(t+1)$ can be coupled as a subgraph of $\mb G^{(k)}(n,p_t)$, where $p_t=2^k(2^{-t}F(0))/\binom{n-1}{k-1}$ (here we are using that $F(t)\le 2^{-t}F(0)$, since $\tau > t$). More precisely, one can sample $\mf G(t+1)$ and $\mb G^{(k)}(n,p_t)$ on a joint probability space such that $\mf G(t+1)$ is deterministically a subgraph of $\mb G^{(k)}(n,p_t)$. It follows that $\bigcup_{t=1}^{\tau} \mb G^{(k)}(n,p_t)$ can be coupled as a subgraph of $G^{(k)}(n,p)$, where 
\[p=\sum_{t=1}^\infty p_t\le 2p_1\le 2\cdot\frac{2^kd_{\max}}{\binom{n-1}{k-1}}\le 1.\]
So, (3) holds.

Now, we claim that (conditioning on any outcomes of $\mf G(1),\dots,\mf G(t)$, but subject to the randomness of $\mf G(t+1)$), with probability $1-n^{-\omega(1)}$ we have that $F(t+1)\le F(t)/2$, and that $\mf G(t+1)$ has maximum degree at most $4F(t)$. Taking a union bound over at most $t\le \lceil\log_2(d_{\max}/(\log n)^2)\rceil$ steps, it will follow that the probability (B) or (C) ever occur is at most $n^{-\omega(1)}$. Note that if (B) and (C) never occur then we have $F(0)+\dots+F(\tau)\le F(0)(1+1/2+1/4+\dots+1/2^\tau)\le 2d_{\max}$ and it is easy to check that (1)--(2) hold. Thus these claims will finish the proof.

So, condition on any outcomes of $\mf G(1),\dots,\mf G(t)$, and consider $\mf G(t+1)$ as defined above. Let $E_v$ be the set of edges containing $v$ in $\mf G(t+1)$. First, we have
\[\mb{E}|E_v|=\frac{w_v\left(\sum_{S\in\binom{V(\mf G)\setminus \{v\}}{k-1}}\prod_{u\in S}w_u\right)}{W}\le w_v\le 2F(t).\]
Noting that $F(t) > (\log n)^2$ (since (A) did not occur), it follows from a Chernoff bound (\cref{lem:chernoff}) that with probability $1-n^{-\omega(1)}$, the maximum degree of $\mf G(t+1)$ is at most $4F(t)$.

Note that only a $\binom{n-1}{k-2}/\binom{n}{k-1} \leq k/n$ fraction of size-$(k-1)$ subsets of $V(\mf G)$ contain $v$, so the above expression for $\mb{E}|E_v|$ also implies that $\mb{E}|E_v|\ge w_v(1-2^k(k/n))$. Since $\mf G(1)\cup\cdots\cup \mf G(t)\cup \mf H$ has maximum degree at most $8d_{\max}+\binom{n-1}{k-1}/(36\cdot 2^k)\le\binom{n-1}{k-1}/(4\cdot 2^k)$, we additionally compute
\[\mb{E}|E_v\cap(\mf G(1)\cup\cdots\cup \mf G(t)\cup \mf H)|\le \frac{\binom{n-1}{k-1}}{4\cdot 2^k}\cdot\frac{w_v2^{k-1}}{\binom{n-1}{k-1}}\le w_v/8.\]

If $n$ is sufficiently large with respect to $k$, it follows that $\mb{E}|E_v\setminus(\mf G(1)\cup\dots\cup \mf G(t)\cup \mf H)|\in[13w_v/16,w_v]$. By a Chernoff bound (\cref{lem:chernoff}) and the fact that $w_v\ge F(t) > (\log n)^2$, with probability $1-n^{-\omega(1)}$ every vertex $v$ has \[\big||E_v\setminus(\mf G(1)\cup\dots\cup \mf G(t)\cup \mf H)|-w_v\big|\le w_v/4\le F(t)/2,\]
in which case $F(t+1)\le F(t)/2$, as desired.
\end{proof}

\section{Subsampled Typical Graphs and Perfect Matchings}\label{sec:subsample}

In this section we prove that graphs which are ``random-like'' have perfect matchings, in a certain robust sense. Our notion of being random-like is as follows.

\begin{definition}\label{def:typical}
An $n$-vertex graph $G$ is \emph{$(p,\xi)$-typical} if every vertex has $(1\pm\xi)pn$ neighbors, and every pair of distinct vertices have $(1\pm\xi)p^2n$ common neighbors.
\end{definition}

Unless $p$ decays very rapidly with $n$, it is straightforward to prove that a random graph $\mb G(n,p)$ is likely to be $(p,o(1))$-typical. Thus, typical graphs resemble random graphs in a certain statistical sense.

In the case where $p\in (0,1)$ is a constant (the ``dense'' case), typicality implies a slightly weaker property called \emph{quasirandomness}, which famously has several equivalent definitions. One of these is that between any two disjoint vertex sets $S,T$ there are about $p|S||T|$ edges (see for example \cite{CGW89}). We will need a quantitative version of this implication, where $p$ is allowed to decay with $n$; the proof is virtually the same as in the dense case, but we include it for completeness as we could not find this specific statement in the literature.

\begin{lemma}\label{clm:discrepancy}
Let $G = (V,E)$ be an $n$-vertex $(p,\xi)$-typical graph, for $p,\xi\in (0,1)$. For every pair of disjoint vertex sets $S,T \subseteq V$, the number of edges $e(S,T)$ between $S$ and $T$ satisfies
\[\big|e(S,T)-p|S||T|\big|\le2(\xi^{1/2}pn+\sqrt{pn})\sqrt{|S||T|}.\]
\end{lemma}
\begin{proof}
We write $\on{codeg}(v,w)$ for the number of common neighbors of $v$ and $w$, and we write $\deg_T(s)$ for the number of neighbors of a vertex $s$ in a set $T$. Note that 
\begin{align*}
\big|e(S,T)-p|S||T|\big|^2 &= \bigg|\sum_{s\in S} \left( \on{deg}_T(s)-p|T| \right) \bigg|^2\le |S|\sum_{s\in S}(\on{deg}_T(s)-p|T|)^2\le |S|\sum_{s\in V}(\on{deg}_T(s)-p|T|)^2\\
&= |S|\sum_{s\in V}\on{deg}_T(s)^2 -2p|T||S|\sum_{s\in V}\on{deg}_T(s)+p^2n|T|^2|S|\\
&= |S|\sum_{t_1,t_2\in T}\on{codeg}(t_1,t_2) -2p|T||S|\sum_{t\in T}\on{deg}(t)+p^2n|T|^2|S|\\
&\le |S|(|T|(1+\xi)pn + |T|(|T|-1)(1+\xi)p^2n) - 2p|T|^2|S|(1-\xi)pn + p^2n|T|^2|S|\\
&\le 2|S||T|pn + 3\xi |T|^2|S|p^2n\le 2|S||T|pn + 3\xi|T||S|p^2n^2,
\end{align*}
and thus
\[\big|e(S,T)-p|S||T|\big|\le 2(\xi^{1/2}pn+\sqrt{pn})\sqrt{|S||T|}.\qedhere\]
\end{proof}

We next give a convenient Hall-type criterion for a bipartite graph to have a perfect matching. It is an immediate consequence of the main theorem in \cite{SS17}.

\begin{lemma}\label{lem:hall}
Let $G = (X\cup Y, E)$ be a bipartite graph with $|X| = |Y| = n$. Suppose that for every $S\subseteq X$, $S'\subseteq Y$ with $|S'|< |S|\le\lceil n/2\rceil$ we have $e(S,Y\setminus S')\neq 0$, and that for every $T'\subseteq X$, $T\subseteq Y$ with $|T'|<|T|\le\lceil n/2\rceil$ we have $e(T,X\setminus T')\neq 0$. Then $G$ has a perfect matching.
\end{lemma}

We now prove that with high probability, a random subgraph of a typical graph is ``robustly matchable'', in the sense that it is not possible to destroy all perfect matchings by deleting a subgraph with small maximum degree. We accomplish this by applying \cref{lem:hall} with a random bipartition of our random subgraph.

\begin{lemma}\label{lem:match}
There is an absolute constant $\xi=\xi_{\ref{lem:match}}>0$ such that the following holds. Fix $\gamma\in(0,1/2)$, and let $n$ be an even number which is sufficiently large in terms of $\gamma$. Let $G$ be an $n$-vertex graph which is $(p,\xi)$-typical with $p\in[n^{-1/3},1]$.

Now, let $R$ be a random subgraph of $G$ obtained by keeping each edge with probability $n^\gamma/(pn)$ independently. With probability at least $1-\exp(-\Omega(n^\gamma))$, this random subgraph $R$ has the property that for any subgraph $F\subseteq R$ with maximum degree at most $\xi n^\gamma$, there is a perfect matching in $R\setminus F$.
\end{lemma}
\begin{proof}
Independently from the randomness of $R$, consider a uniformly random bipartition $X\cup Y$ of the vertex set of $G$, into two parts of size exactly $n/2$. Let $G'$ be the bipartite subgraph of $G$ consisting of edges between $X$ and $Y$. Every vertex $v$ has degree $(1\pm\xi)pn$ in $G$, so by the Chernoff bound for hypergeometric random variables (\cref{lem:chernoff}) and a union bound, with probability $1-n\exp(-\Omega(\xi pn))$ we have $\deg_{G'}(v) = (1\pm 2\xi)p(n/2)$ for each vertex $v$. Now consider some $S\subseteq X, S'\subseteq Y$ satisfying $\lceil n/4\rceil\ge|S|>|S'|$. By \cref{clm:discrepancy}, we have
\begin{align*}
e_{G'}(S,Y\setminus S') &= \sum_{v\in S}\deg_{G'}(v) - e_G(S,S')\ge (1-2\xi)p(n/2)|S| - p|S||S'| - 4\xi^{1/2}pn\sqrt{|S||S'|}\ge pn|S|/10
\end{align*}
for large $n$ and small $\xi$. Similarly, if $T'\subseteq X, T\subseteq Y$ with $|T'|<|T|\le\lceil n/4\rceil$ then $e_{G'}(T,X\setminus T')\ge pn|T|/10$.

Now, fix any outcome of $G'$ satisfying the above properties, and let $R'=G'\cap R$. We observe that with probability at least $1-\exp(-\Omega(n^\gamma))$, we have $e_{R'}(S,Y\setminus S')\ge n^\gamma|S|/20$ for every pair of sets $S\subseteq X, S'\subseteq Y$ with $\lceil n/4\rceil\ge|S|>|S'|$. Indeed, by a union bound over $S,S'$ and the Chernoff bound, the probability that this fails to hold is at most
\[\sum_{k=1}^{\lceil n/4\rceil}\binom{n/2}{k}\sum_{\ell=0}^{k-1} \binom{n/2}{\ell}\exp(-\Omega(kn^\gamma))\le\sum_{k=1}^{\lceil n/4\rceil} n^{2k}\exp(-\Omega(kn^\gamma))\le\exp(-\Omega(n^{\gamma})).\]
Assuming the above event holds, to make $e_{R'}(S,Y\setminus S')=0$ we need to delete at least $n^\gamma/20$ edges incident to some vertex of $S$. By symmetry, the same is true (with probability at least $1-\exp(-\Omega(n^\gamma))$) when switching the roles of $X$ and $Y$. The desired result then follows from \cref{lem:hall}, as long as $\xi<1/20$.
\end{proof}

\section{Well-spread forbidden configurations}\label{sec:spread}

For our generalized version of the high-girth triple process, we will need to make certain statistical ``well-spreadness'' assumptions about our family of forbidden configurations, which mimic some basic statistical properties of the family of Erd\H{o}s configurations. We will have to define these assumptions in a sufficiently loose way that ``sparse enough'' random sets of triangles are likely to satisfy them, since such sets will arise during the regularization step of our proof (recall the outline in \cref{sec:overview}). In particular, we cannot make assumptions that rely too closely on the specific way that triangles are arranged with respect to each other in every forbidden configuration (whereas such information is used rather crucially in \cite{GKLO20,BW19}).

For example, note that in the complete graph $K_n$, the number of Erd\H{o}s $j$-configurations containing a given set $\mc{R}$ of triangles is $O(n^{j-v(\mc R)})$, where $v(\mc{R})$ is the number of vertices spanned by the triangles in $\mc{R}$. We will make an assumption of this type, for general forbidden configurations with $j-2$ triangles, without demanding that each of the configurations have exactly $j$ vertices. This type of assumption is probably the most important for us, but we will also need to make further assumptions about \emph{pairs} of forbidden configurations. A subtle property of Erd\H{o}s $j$-configurations is that their vertex set is determined by any set of $j-3$ triangles (equivalently, deleting a triangle does not delete any vertices), by \cref{lem:erdos-minimality}(1). This has implications for the way that certain pairs of Erd\H{o}s configurations are distributed. For example, if we fix a pair of distinct triangles $T,T'$, then the number of Erd\H{o}s $j$-configurations containing $T$ is $\Theta(n^{j-3})$, but the number of pairs of Erd\H{o}s $j$-configurations $\mc{E},\mc{E}'$ with $T\in\mc{E}$, $T'\in\mc{E}'$ and $\mc{E}\setminus\{T\}=\mc{E}'\setminus\{T'\}$ is only $O(n^{j-4})$ (because both $\mc{E}$ and $\mc{E}'$ must have the same vertex set containing all $v(T\cup T')\ge 4$ vertices of $T$ and $T'$).

It turns out that assumptions of the above type (concerning the numbers of Erd\H{o}s configurations and pairs of Erd\H{o}s configurations containing given sets of triangles) are sufficient for a basic generalization of the theorems in \cite{GKLO20,BW19}. However, we will actually need rather more complicated assumptions: we need a general theorem which we can apply iteratively, ``moving down'' a vortex $U_0\supseteq U_1\supseteq\cdots\supseteq U_\ell$ (recall the discussion in \cref{sec:birds-eye}). We therefore need our family of forbidden configurations to be well-spread, in a somewhat technical way, with respect to our vortex.

\begin{definition}[Well-spread]\label{def:well-spread}
First, for any nonempty set of at most $j-2$ triangles $\mc{R}$, define
\[v^j(\mc{R})=\begin{cases}
|\mc{R}|+2&\text{ if }|\mc{R}|\in\{1,j-2\}\\
|\mc{R}|+3&\text{ if }1<|\mc{R}|<j-2.
\end{cases}\]
The idea is that if $\mc{R}$ is a nonempty subset of an Erd\H{o}s $j$-configuration, then it has at least $v^j(\mc{R})$ vertices (this follows from \cref{lem:erdos-minimality} applied to $\mc{R}$).

Now, fix a descending sequence of subsets $V(K_N)=U_0\supseteq\cdots\supseteq U_k$, and let $n=|U_k|$. Fix some $y,z\in\mb{R}$ (which should be thought of as ``error parameters''). For a triple $T$ in $K_N$, let $\on{lev}(T)$ be the maximum $i$ such that $T$ has all its vertices in $U_i$. For $4\le j\le g$, say that a collection of sets of $j-2$ triangles $\mf{F}_j$ is \emph{$(y,z)$-well-spread} (with respect to our descending sequence of subsets) if the following conditions hold.
\begin{enumerate}[{\bfseries{WS\arabic{enumi}}}]
\setcounter{enumi}{-1}
    \item\label{WS0} Every $\mc{E}\in\mf{F}_j$ is an edge-disjoint collection of triangles.
    \item\label{WS1} For every nonempty set of triangles $\mc{R}$ and any sequence $t_0,\dots,t_{k-1}$, the number of $\mc{E}\in\mf{F}_j$ which include $\mc{R}$ and, for each $0\le i < k$, contain $t_i$ triangles $T\notin\mc{R}$ with $\on{lev}(T)=i$ is at most
    \[zn^{j-(t_0+\dots+t_{k-1})-v^j(\mc R)}\prod_{0\le i<k}|U_i|^{t_i}.\]
    \item\label{WS2} Fix a pair of triangles $T,T'$ and any sequence $t_0,\dots,t_{k-1}$. Consider all the pairs of distinct $\mc{E},\mc{E}'\in \mf{F}_j$ for which $T\in\mc{E}$, $T'\in\mc{E}'$ and $\mc{E}\setminus \{T\}=\mc{E}'\setminus\{T'\}$. The number of such pairs $\mc{E},\mc{E}'$ which contain $t_i$ triangles $T''\in\mc{E}\setminus\{T\}$ with $\on{lev}(T'')=i$, for $0\le i < k$, is at most
    \[zn^{j-(t_0+\dots+t_{k-1})-4}\prod_{0\le i<k}|U_i|^{t_i}.\]
    \item\label{WS3} If $j=4$, then for every triangle $T$ and edge $e\notin E(T)$, there are at most $z$ configurations $\mc{E}\in\mf{F}_j$ which contain $T$ and a second triangle $T'$ with $\on{lev}(T')=k$ and $e\in E(T')$.
    \item\label{WS4} For every triangle $T$ and any sequence $t_0,\dots,t_{k-1}$, the number of $\mc{E}\in\mf{F}_j$ which contain $T$ and for each $0\le i<k$ contain $t_i$ triangles $T'\neq T$ with $\on{lev}(T')=i$ is at most
    \[yn^{j-(t_0+\dots+t_{k-1})-3}\prod_{0\le i<k}|U_i|^{t_i}.\]
\end{enumerate}
Finally, we say that a collection of sets of $j-2$ triangles $\mf{F}_j$ is \emph{$z$-well-spread} if it is $(z,z)$-well-spread.
\end{definition}
\begin{remark}
We make a few important comments regarding the above definition. First, note that we allow $j\in\{4,5\}$, despite the fact that Erd\H{o}s 4-configurations and Erd\H{o}s 5-configurations do not exist. Second, note that \cref{WS4} is simply \cref{WS1} for $|\mc{R}| = 1$, but with an adjusted error parameter. In certain applications, it will be crucial that $y = o(z)$. Third, note that being well-spread for $U_0\supseteq\cdots\supseteq U_{k+1}$ is not strictly stronger than being well-spread for $U_0\supseteq\cdots\supseteq U_k$ (since the definition of $n$ changes); similarly, the definition of $\on{lev}(T)$ depends on our sequence of sets. These are not serious issues, though we must use care when applying the definitions and results in this section. Finally, we note that the definition of being well-spread still makes sense when $k = 0$ (in which case the sequences $t_0,\dots,t_{k-1}$ are empty).
\end{remark}

It is not too difficult to prove that for any descending sequence of subsets $V(K_N)=U_0\supseteq\cdots\supseteq U_k$, the collection of Erd\H{o}s $j$-configurations in $K_N$ is $O_{j,k}(1)$-well-spread. However, we will need a version of this fact that incorporates the influence of the absorbing structure defined in \cref{sec:absorbers}; we are not only interested in Erd\H{o}s configurations, but more generally in sets of triangles which would complete an Erd\H{o}s configuration when combined with some absorber triangles. Furthermore, we will need to account for the additional (random) forbidden configurations which may be added for purposes of regularization. We first tackle the impact of the absorbing structure, which is the more difficult of these considerations.

\begin{lemma}\label{lem:absorber-well-spread}
Let $k\ge 0$, consider a descending sequence of subsets $V(K_N)=U_0\supseteq\dots\supseteq U_k\supseteq X$, and let $n = |U_k|$. Consider an absorbing structure $H$ (with distinguished set $X$) satisfying \cref{AB1} and \cref{AB2} in \cref{thm:absorbers} for $g\in\mb{N}$. Suppose $H$ is embedded in a complete graph $K_N$, and assume that (if $k>0$) we have $V(H)\setminus X\subseteq U_0\setminus U_1$.

For any triangle-divisible graph $L$ on the vertex set $X$ let $\mc{S}_L$ be the associated triangle-decomposition of $L\cup H$ (as in \cref{AB1}), and let $\mc{B} = \bigcup_L\mc{S}_L$. Suppose that $|\mc{B}|^{2g}\le N^\beta$, for some $\beta\le 1$. For $j\ge 4$, let $\mf{F}_j^\mc{B}$ be the collection of all $(j-2)$-sets of triangles such that there is $L$ so that one can obtain an Erd\H{o}s $j'$-configuration for some $5\le j'\le g$ by adding $j'-j$ triangles from $\mc{B}$. Then $\mf{F}_j^\mc{B}$ is $(C_{\ref{lem:absorber-well-spread}}(g,k), C_{\ref{lem:absorber-well-spread}}(g,k) n^\beta)$-well-spread with respect to $U_0\supseteq\cdots\supseteq U_k$.
\end{lemma}

\begin{remark}
Each $\mc{S}_L$ is a triangle-decomposition (which is viewed as a collection of triangles) that may arise in the final step using our absorber, and $\mc{B}$ is the collection of all triangles appearing in some $\mc{S}_L$. Thus, $\mf{F}_j^\mc{B}$ certainly contains all possible ``induced'' forbidden configurations that may come from using the absorber, allowing us to proceed without knowing what $L$ may be in advance. Note that $\mf{F}_j^\mc{B}$ potentially includes ``extra'' configurations in which the $j'-j$ triangles lie in multiple different triangle-decompositions $\mc{S}_L$. However, the strength of \cref{AB2} allows sufficient leeway for such situations.
\end{remark}

\begin{proof}
The configurations in $\mf{F}_j^\mc{B}$ are edge-disjoint by definition. Consider $j\ge 4$ and sets of at most $j-2$ triangles $\mc{R},\mc{K}$ in $K_N$ with $\mc{R}\neq\emptyset$, and consider a sequence of nonnegative integers $t_0,\ldots,t_{k-1}$. Let $\#(\mc{R},\mc{K},j,t_0,\ldots,t_{k-1})$ be the number of Erd\H{o}s configurations $\mc{E}$ with $\mc{E}\cap\mc{B} = \mc{K}$, with $\mc{R}\subseteq\mc{E}\setminus\mc{K}$, which have $j+|\mc{K}|-2$ triangles (so $j+|\mc{K}|$ vertices) in total, with $|\mc{K}|\le g-j$, and which contain $t_i$ triangles $T\notin\mc{K}\cup\mc{R}$ with $\on{lev}(T)=i$. Let
\[\#(\mc{R},j,t_0,\dots,t_{k-1})=\sum_{\substack{\mc{K}\subseteq\mc{B}\\|\mc{K}|\le g-j}}\#(\mc{R},\mc{K},j,t_0,\dots,t_{k-1}),\]
which is an upper bound for the quantity of interest in \cref{WS1} for $\mf{F}_j^\mc{B}$, and for \cref{WS4} in the cases with $|\mc{R}| = 1$.

We start by proving an upper bound on $\#(\mc R,\mc K,j,t_0,\dots,t_{k-1})$. Suppose $\mc E$ contributes to this quantity. For $0\le i\le k$, let $v_i=|(V(\mc{E})\setminus {V(\mc{K}\cup\mc{R})}) \cap (U_i\setminus U_{i+1})|$ be the number of vertices in $\mc{E}$ but not in $\mc{K}\cup\mc{R}$ which lie in $U_i\setminus U_{i+1}$ (here we take $U_{k+1}$ to be the empty set). Recalling the definition of $v^j$ from \cref{def:well-spread}, note that
\begin{equation}\label{eq:wd-vtx-count}
 v_0+\dots+v_k=j+|\mc{K}|-v(\mc{K}\cup\mc{R})\le j-v^j(\mc{R}).
\end{equation}
To see why the inequality holds, note that if $|\mc{R}|\in\{1,j-2\}$, the inequality is equivalent to $v(\mc{K}\cup\mc{R})\ge|\mc{K}\cup\mc{R}|+2$, which holds by \cref{lem:erdos-minimality}(2). Otherwise, if $1 < |\mc{R}| < j-2$, the inequality is equivalent to $v(\mc{K}\cup\mc{R})\ge|\mc{K}\cup\mc{R}|+3$. This holds by \cref{lem:erdos-minimality}(1), since $1 < |\mc{K}\cup\mc{R}| < j+|\mc{K}|-2$.

Now, since $|V(\mc{R})|\ge 3$, we may apply \cref{lem:erdos-minimality}(2) to the vertex set $(V(\mc{E})\setminus V(\mc{K}\cup\mc{R}))\cap(U_0\setminus U_{i+1})$ to obtain $v_0+\cdots+v_i\le t_0+\cdots+t_i$ for $0\le i < k$. Therefore, the contribution to $\#(\mc{R},\mc{K},j,t_0,\dots,t_{k-1})$ from a particular choice of $(v_i)_{0\le i<k}$ is at most
\begin{align}
\prod_{0\le i\le k}|U_i|^{v_i} &= |U_k|^{v_0+\cdots+v_k}\prod_{0\le i<k}\bigg(\frac{|U_i|}{|U_{i+1}|}\bigg)^{v_0+\cdots+v_i}\le n^{j-v^j(\mc{R})}\prod_{0\le i<k}\bigg(\frac{|U_i|}{|U_{i+1}|}\bigg)^{t_0+\cdots+t_i}\notag\\
&= n^{j-v^j(\mc{R})-(t_0+\cdots+t_{k-1})}\prod_{0\le i<k}|U_i|^{t_i}.\label{eq:wd-abel}
\end{align}
Note that there are only $O_{g,k}(1)$ possibilities for $(v_i)_{0\le i<k}$, so
\begin{equation}
\#(\mc{R},\mc{K},j,t_0,\dots,t_{k-1})\le O_{g,k}(1)\cdot n^{j-v^j(\mc{R})-(t_0+\cdots+t_{k-1})}\prod_{0 \leq i<k}|U_i|^{t_i}.\label{eq:almost-well-spread}
\end{equation}

We will now distinguish some cases to prove that
\begin{equation}
\#(\mc{R},j,t_0,\dots,t_{k-1})\le O_{g,k}(1)\cdot n^{\beta+j-v^j(\mc{R})-(t_0+\cdots+t_{k-1})}\prod_{0\le i<k}|U_i|^{t_i},\label{eq:almost-well-spread2}
\end{equation}
which will show that $\mf{F}_j^\mc{B}$ satisfies \cref{WS1} with parameter $z=O_{g,k}(n^\beta)$.
\begin{itemize}
    \item \textit{Case 1: $n\ge N^{1/2}$.} There are at most $|\mc{B}|^g\le N^{\beta/2}\le n^\beta$ ways to choose $\mc{K}$, so we can simply sum \cref{eq:almost-well-spread} over all such choices to obtain \cref{eq:almost-well-spread2}. Note this covers all cases when $k = 0$.
    \item \textit{Case 2: $n < N^{1/2}$ and $t_0=0$.} We claim there are actually only $O_g(1)$ choices for $\mc{K}$ in this case, so we can again sum over all choices. Recall \cref{AB2}. Either $\mc{K}\subseteq\mc{L}_{\mc{R}}$, which provides $O_g(1)$ choices for $\mc{K}$, or there must be $T\in\mc{E}\setminus(\mc{R}\cup\mc{B})$ such that $T$ contains a vertex of $V(H)\setminus X$. In the latter case, note that $\mc{K}\subseteq\mc{B}$ and $V(H)\setminus X\subseteq U_0\setminus U_1$. Then $t_0 > 0$ follows, which is a contradiction.
    \item \textit{Case 3: $n < N^{1/2}$ and $t_0>0$.} Configurations with $\mc{K} = \emptyset$ are bounded by \cref{eq:almost-well-spread}. Otherwise, we obtain a strict inequality $v_0+\cdots+v_i<t_0+\cdots+t_i$ for all $0\le i<k$. Indeed, if $|(V(\mc{E})\setminus V(\mc{K}\cup\mc{R}))\cap(U_0\setminus U_{i+1})| = v_0+\cdots+v_i=0$ then this is trivial since $t_0 > 0$. Otherwise apply \cref{lem:erdos-minimality}(1) to this nonempty vertex set, noting that $\mc{K}\cup\mc{R}$ contains at least four vertices total. There are at most $|\mc{B}|^g$ ways to choose $\mc{K}$, and for each fixed $\mc{K}$ and $(v_i)_{0\le i<k}$, the number of choices for $\mc{E}$ is at most
    \[\prod_{0\le i\le k} |U_i|^{v_i}\le(n/N)n^{j-v^j(\mc{R})-(t_0+\dots+t_{k-1})}\prod_{0\le i<k} |U_i|^{t_i},\]
    mimicking the proof of \cref{eq:wd-abel}. Since $|\mc{B}|^g\le N^{1/2}\le N/n$, the desired bound follows.
\end{itemize}

We have now established \cref{eq:almost-well-spread2} in all cases, proving that $\mf{F}_j^\mc{B}$ satisfies \cref{WS1}. Next, we observe that in the case $|\mc{R}|=1$ and $\mc{K}\neq\emptyset$, the inequality in \cref{eq:wd-vtx-count} is strict (due to \cref{lem:erdos-minimality}(1) and $1 < |\mc{K}\cup\mc{R}|<j+|\mc{K}|-2$). So, in this case, the proof of \cref{eq:almost-well-spread2} can be seen to save a factor of $n$, and actually shows
\begin{equation}
    \#(\mc R,\mc K,j,t_0,\dots,t_{k-1})\le O_{g,k}(1)\cdot n^{\beta+j-4-(t_0+\cdots+t_{k-1})}\prod_{0 \leq i<k}|U_i|^{t_i}.\label{eq:strong-well-spread}
    \end{equation}
Then, \cref{eq:strong-well-spread} (for the case $\mc K\ne \emptyset$) and \cref{eq:almost-well-spread} (for the case $\mc K=\emptyset$) together imply
\begin{equation}\notag
\#(\mc{R},j,t_0,\dots,t_{k-1})\le O_{g,k}(1)\cdot n^{j-3-(t_0+\cdots+t_{k-1})}\prod_{0\le i<k}|U_i|^{t_i},
\end{equation}
for $|\mc{R}| = 1$, which yields \cref{WS4} of $(C_{\ref{lem:absorber-well-spread}}(g,k),C_{\ref{lem:absorber-well-spread}}(g,k)n^\beta)$-well-spreadness. 
Also, \cref{eq:strong-well-spread} proves \cref{WS3} (with room to spare; we do not need to consider the edge $e$ at all).

It remains to prove property \cref{WS2} with the given parameters. Fix $T,T'$. We are counting the number of pairs $\mc E,\mc E'\in\mf F_j^\mc{B}$ satisfying a certain property, and such that there are $t_i$ triangles $T''\in\mc{E}\setminus\{T\}$ with $\on{lev}(T'') = i$ for $0\le i < k$. First, the contribution coming from pairs in which $\mc E$ is a ``derived''  forbidden configurations (i.e., not a complete Erd\H{o}s configuration, induced by some nonempty subset $\mc K\subseteq\mc{B}$) is bounded as desired by \cref{eq:strong-well-spread}, taking $\mc R=\{T\}$. A symmetric argument for $\mc{E}'$ shows we only need to consider the contribution coming from the case where $\mc E,\mc E'$ are both genuine Erd\H{o}s configurations. In this case $\mc E'\setminus \{T'\}$ has the same vertex set as $\mc E'$, so $\mc E'$ has to contain all $v(T\cup T')\ge 4$ vertices of $T$ and $T'$. To count the number of such configurations let $v_i = |(V(\mc E)\setminus V(T\cup T'))\cap (U_i\setminus U_{i+1})|$ for $0\le i \leq k$ (once again taking $U_{k+1} = \emptyset$), and note that 
\[v_0+\cdots+v_k\le j-4.\]
With this modified definition of $v_i$, we still have $v_0+\cdots+v_i\le t_0 + \cdots+t_i$ for $0\le i<k$ by \cref{lem:erdos-minimality}(2) and the number of choices for $\mc{E}$ given a particular choice of $(v_i)_{0\le i<k}$ is at most 
\begin{align*}
\prod_{i\le k}|U_i|^{v_i} &= |U_k|^{v_0+\cdots+v_i}\prod_{i<k}\bigg(\frac{|U_i|}{|U_{i+1}|}\bigg)^{v_0+\cdots+v_i}\!\le n^{j-4}\prod_{i<k}\bigg(\frac{|U_i|}{|U_{i+1}|}\bigg)^{t_0+\cdots+t_i}\!= n^{j-4-(t_0+\cdots+t_{k-1})}\prod_{i<k}|U_i|^{t_i}.
\end{align*}
Summing over $O_{g,k}(1)$ sequences $(v_i)_{0\le i < k}$ and $O_g(1)$ ways to choose $\mc{E}'$ from triangles on $V(\mc{E})$ completes the proof.
\end{proof}

Finally, the following lemma says that if one adds appropriate \emph{random} configurations to a well-spread collection of configurations, the result is still well-spread.

\begin{lemma}\label{lem:random-well-spread}
Fix $\beta\in(0,1/3)$. Let $\mf F_j$ be an $n^\beta$-well-spread collection of sets of $j-2$ triangles with respect to $V(K_N) = U_0\supseteq\cdots\supseteq U_k$. Let $\mf F_j^\mr{rand}$ be a random collection of sets of vertex-disjoint triangles in $K_N[U_k]\cong K_n$, where every set of $j-2$ triangles is included in $\mf F_j^\mr{rand}$ with probability $p_j=n^\beta n^{-2j+6}$. With probability $1-n^{-\omega(1)}$, the collection of sets of triangles $\mf F_j\cup \mf F_j^\mr{rand}$ is $C_{\ref{lem:random-well-spread}}(g,k)n^\beta$-well-spread.
\end{lemma}

\begin{proof}
First, we claim that with probability $1-n^{-\omega(1)}$, $\mf F_j^\mr{rand}$ satisfies property \cref{WS1} of $C_{\ref{lem:random-well-spread}}(g,k)n^\beta$-well-spreadness. Indeed, note that no configuration in $\mf F_j^\mr{rand}$ ever contains a triangle $T$ with $\on{lev}(T)<k$. Fixing nonempty $\mc R$, there are at most $(n^3)^{j-2-|\mc R|}$ possible configurations including $\mc R$, and each is present independently with probability $p_j$. The expected number of such configurations in $\mf F_j^\mr{rand}$ is at most $n^{\beta+j-3|\mc R|}\le n^{\beta + j-v^j(\mc R)}$, so a Chernoff bound proves the claim. Note in this situation, \cref{WS4} is a special case of \cref{WS1}.

Next, we claim that with probability $1-n^{-\omega(1)}$, if $j = 4$ then $\mf F_j^\mr{rand}$ satisfies property \cref{WS3}. Indeed, for a triangle $T$ and an edge $e\notin E(T)$, there are at most $n$ triangles $T'$ with $\on{lev}(T') = k$ containing $e$. So, the expected number of configurations in $\mf F_4^\mr{rand}$ containing $T$ and such a triangle $T'$ is at most $np_4\le n^{\beta-1}$, implying (via a Chernoff bound) that property \cref{WS3} holds for $\mf F_4^\mr{rand}$ with probability $1-n^{-\omega(1)}$.

For property \cref{WS2}, we need to consider separately the case where $\mc E,\mc E'\in \mf F_j^\mr{rand}$ and the case where $\mc E\in \mf F_j^\mr{rand},\mc E'\in \mf F_j$. In both cases we can again use a simple Chernoff bound. The expected number of suitable pairs in the former case is at most $(n^3)^{j-3}p_j^2\le n^{2\beta -j+3}=o(1)$, choosing $j-3$ disjoint triangles for $\mc{E}\setminus\{T\}$. Then, the expected number of suitable pairs in the latter case is at most $O_g(n^{\beta+j-3}p_j)\le n^{2\beta -j+3}=o(1)$. Here we are using that $\mc{E}'\setminus\{T'\}=\mc E\setminus \{T\}$ must have all vertices in $U_k$, so \cref{WS4} applied to $\mf F_j$ implies there are at most $n^{\beta+j-3}$ choices for $\mc E'$. The result follows.
\end{proof}

\section{Bounds for weight systems}\label{sec:weight}
At many points in the proof of \cref{thm:main}, we will prove upper bounds on random variables by defining a weight system $\vec \pi$ and a set of configurations $\mf X$, and applying \cref{lem:moments}. In order to apply \cref{lem:moments} it is necessary to estimate the maximum weight $\kappa^{(\vec \pi)}(\mf X)$, and in many of these applications the case analysis involved in this estimation is quite involved (though mostly quite mechanical). We collect all such computations in this section: in particular, we prove a number of abstract lemmas providing upper bounds on maximum weights in various situations, tailored for various applications in the proof of \cref{thm:main}. The lemmas in this section are likely to seem unmotivated on first reading; the reader is encouraged to skip over this section and refer to it only as various estimates are needed in later sections.

First, for our most straightforward applications of \cref{lem:moments}, our ground set $\mc W$ is a set of triangles in a graph, and our weight system $\vec \pi$ encodes the approximate probability that a given triangle $T$ ends up being chosen in one of the (random) stages of the proof of \cref{thm:main}. This probability depends on the position of $T$ with respect to the vortex $U_0\supseteq\dots\supseteq U_\ell$ (recall the outline in \cref{sec:overview}). The multiset of configurations with which we will apply \cref{lem:moments} will be defined in terms of a well-spread collection of forbidden configurations (for example, we might be interested in a high-probability upper bound on the number of forbidden configurations containing a given triangle $T$, of which a certain number of triangles have so far been chosen, in which case we would apply \cref{lem:moments} to a multiset of configurations obtained by taking appropriate-size subsets of forbidden configurations containing $T$). In the following lemma we collect some general weight estimates which are suitable for studying collections of configurations of this type.

Recall that for a triangle $T$ in $K_N$ and a vortex $V(K_N) = U_0 \supseteq \dots \supseteq U_k$ we write $\on{lev}(T)$ for the maximum $i$ such that $T$ has all its vertices in $U_i$. The functions $v^j(\cdot)$ used throughout the section were defined in \cref{def:well-spread}.

\begin{lemma}\label{lem:general-moments}
Fix positive real numbers $w,y,z$, fix a descending sequence of subsets $V(K_N)=U_{0}\supseteq\dots\supseteq U_{k}$, and let $n=|U_{k}|$. Let $\mf F_4,\dots,\mf F_g$ be collections of sets of triangles of $K_N$, where each configuration in $\mf F_j$ contains $j-2$ triangles and each $\mf F_j$ is $(y,z)$-well-spread with respect to our sequence of sets. Let $\mc W$ be the set of triangles in $K_N$, and let $\vec{\pi}$ be the weight system defined by $\pi_T = w/|U_{\on{lev}(T)}|$ for all $T\in \mc W$.
\begin{enumerate}
    \item Fix a nonempty set of triangles $\mc{Q}$ of $K_N$, and integers $j$ and $f$ such that $4\le j\le g$ and $f\ge 0$. Let $\mf{A}_{\mc Q,j,f}^{(1)}$ be the multiset of sets of triangles constructed as follows. Consider every $\mc{E}\in\mf F_j$ and every partition $\mc{E}=\mc{Q}\cup\mc{Z}\cup\mc{F}$ such that $|\mc{F}| = f$ and all triangles in $\mc{Z}$ are contained in $U_k$. For each such choice of $(\mc{E},\mc{F},\mc{Z})$, we add a copy of $\mc{F}$ to $\mf{A}_{\mc Q,j,f}^{(1)}$. Then
    \[\psi(\mf{A}_{\mc Q,j,f}^{(1)}) = \begin{cases}
    O_{g,k}(zw^fn^{j-v^j(\mc{Q})-f})&\emph{for all }\mc{Q},\\
    O_{g,k}(yw^fn^{j-v^j(\mc{Q})-f})&\emph{if }|\mc Q|=1.
    \end{cases}
    \]

    \item Fix sets of triangles $\mc{Q},\mc Q'$ of $K_N$ with $\mc{Q}\neq\emptyset$, fix $j,j',v',f$ such that $4\le j,j'\le g$ and $v',f\ge 0$, and let $v=v^j(\mc Q)$. Let $\mf{A}_{\mc Q,\mc Q', j,j',v',f}^{(2)}$ be the multiset of sets of triangles constructed as follows. Consider each $\mc{E}\in \mf F_j,\mc{E}'\in \mf F_{j'}$ and every choice of partitions $\mc{E} = \mc{Q}\cup\mc{Z}\cup\mc{F}$ and $\mc{E}' = \mc{Q}'\cup\mc{Z}'\cup\mc{F}'$ such that $\mc{F}\cup\mc{F}'$ is disjoint from $\mc{Q}\cup\mc{Z}\cup\mc{Q}'\cup\mc{Z}'$, $\mc{Q}'\cup(\mc{E}\cap\mc{E}')\neq\emptyset$, $v'=v^{j'}(\mc{E}'\cap(\mc{E}\cup\mc{Q}'))$, $|\mc{F}\cup\mc{F}'|=f$, and all the triangles in $\mc{Z}\cup\mc{Z}'$ are contained in $U_{k}$. (Note $v'$ is well-defined as $\mc{E}'\cap(\mc{E}\cup\mc{Q}') = \mc{Q}'\cup(\mc{E}\cap\mc{E}')$.) For each such choice of $(\mc{E},\mc{E}',\mc{F},\mc{F}',\mc{Z},\mc{Z}')$, we add a copy of $\mc{F}\cup\mc{F}'$ to $\mf{A}_{\mc Q,\mc Q', j,j',v,v',f}^{(2)}$. Then
    \[\psi(\mf{A}_{\mc Q,\mc Q', j,j',v',f}^{(2)}) = O_{g,k}(z^2w^fn^{j+j'-v-v'-f}).\]
    
    \item Fix a pair of (not necessarily distinct) triangles $T,T'$ of $K_N$, and fix some $4\le j\le g$ and $f\ge 0$. Let $\mf{A}_{T,T',j,f}^{(3)}$ be the multiset of sets of triangles constructed as follows. Consider each pair of distinct $\mc{E},\mc{E}'\in \mf F_j$ with $\mc{E}\setminus\{T\}=\mc{E}'\setminus\{T'\}$, and every choice of partition $\mc{E}\setminus\{T\} = \mc{Z}\cup\mc{F}$ such that $|\mc{F}|=f$ and all the triangles in $\mc Z$ are contained in $U_k$. For each such choice of $(\mc{E},\mc{E}',\mc{F},\mc{Z})$, we add a copy of $\mc{F}$ to $\mf{A}_{T,T',j,f}^{(3)}$. Then \[\psi(\mf{A}_{T,T',j,f}^{(3)}) = O_{g,k}(zw^fn^{j-4-f}).\]
\end{enumerate}
\end{lemma}
\begin{proof}
First we consider (1). For some choice of $\mc E,\mc F,\mc Z$ and some $0\le i\le k-1$, let $t_i$ be the number of triangles $T\in\mc{F}$ with $\on{lev}(T)=i$. For any particular profile $(t_i)_{0\le i<k}$, the contribution to $\psi(\mf{A}_{\mc Q,j,f}^{(1)})$ from choices of $\mc E,\mc F,\mc Z$ with that profile is
\[
O_g(1)\cdot\left(zn^{j-(t_0+\cdots+t_{k-1})-v^j(\mc Q)}\prod_{i<k}|U_{i}|^{t_i}\right)\left(w^f|U_k|^{-(f-t_0-\cdots-t_{k-1})}\prod_{i < k}|U_{i}|^{-t_i}\right) = O_g \left( zw^fn^{j-v^j(\mc{Q})-f} \right).
\]
To obtain this estimate, we have applied \cref{WS1} (in the definition of well-spreadness, in \cref{def:well-spread}) to $\mc{R} = \mc{Q}\neq\emptyset$ (since every $T\in\mc{Z}$ satisfies $\on{lev}(T) = k$) in order to count choices of $\mc{E}$, with an extra factor depending only on $g$ to account for the number of choices of $\mc{F}$. The first bound in (1) follows by summing over $O_{g,k}(1)$ possibilities for the profile $(t_i)_{0\le i<k}$. For the alternate bound when $|\mc{Q}|=1$, we simply use \cref{WS4} instead of \cref{WS1}.

Next, (2) is very similar, but with slightly more complicated notation. For $0\le i < k$, let $t_i$ be the number of triangles $T\in\mc{F}$ with $\on{lev}(T)=i$,  and let $t_i'$ be the number of triangles $T\in\mc{F}'\setminus\mc{E} = \mc{F}'\setminus\mc{F}$ with $\on{lev}(T)=i$. The contribution to $\psi(\mf{A}_{\mc Q,\mc Q', j,j',v,v',f}^{(2)})$ coming from a particular choice of the profiles $(t_i)_{0\le i<k}$, $(t_i')_{0\le i<k}$
is at most
\begin{align*}
O_g(1)\cdot&\left(zn^{j-(t_0+\dots+t_{k-1})-v}\prod_{i<k}\left|U_{i}\right|^{t_i}\right)\left(zn^{j'-(t_0'+\dots+t_{k-1}')-v'}\prod_{i<k}\left|U_{i}\right|^{t_i'}\right)\\
&\cdot\left(w^f|U_k|^{-(f-t_0-\cdots-t_{k-1}-t_0'-\cdots-t_{k-1}')}\prod_{i<k}\left|U_{i}\right|^{-(t_i+t_i')}\right) = O_g \left( z^2w^fn^{j+j'-v-v'-f} \right).
\end{align*}
Here we have used \cref{WS1} twice to count choices for $\mc E,\mc E'$, first applied to $\mc{R} = \mc{Q}\neq\emptyset$ and then to $\mc{R} = \mc{E}'\cap(\mc{E}\cup\mc{Q}') = \mc{Q}'\cup(\mc{E}\cap\mc{E}')\neq\emptyset$ (there are $O_g(1)$ choices for such an $\mc R$).
The desired result follows upon summing over possibilities for the profiles $(t_i)_{0\le i<k}$ and $(t_i')_{0\le i<k}$.

Finally, (3) is essentially identical to (1). For $0\le i < k$, let $t_i$ be the number of triangles $T''\in \mc{F}$ with $\on{lev}(T'')=i$. The contribution to $\psi(\mf{A}_{T,T',j,f}^{(3)})$ coming from a particular choice of the profile $(t_i)_{0\le i<k}$ is at most
\[
O_g(1)\cdot\left(z n^{j-(t_0+\cdots+t_{k-1})-4}\prod_{i<k}\left|U_{i}\right|^{t_i}\right)\left(w^fn^{-(f-t_0-\cdots-t_{k-1})}\prod_{i<k}\left|U_i\right|^{-t_i}\right) = O_g \left( zw^fn^{j-4-f} \right),
\]
where we have used \cref{WS2}. As before the desired bound is obtained by summing over choices of $(t_i)_{0\le i<k}$.
\end{proof}

Using the general estimates in \cref{lem:general-moments}, we now estimate some maximum weights $\kappa(\mf K)=\max_{\mc H}\psi(\mf K,\mc H)$ (recall the notation introduced \cref{def:weight-system}), in certain rather specific settings.

Each of the estimates in the following lemma corresponds to a specific quantity that we will need to control when we analyze a generalized high-girth triple process in \cref{sec:nibble}.

\begin{lemma}\label{lem:nibble-moments}
Fix positive real numbers $w,z\ge 1$, fix a descending sequence of subsets $V(K_N)=U_{0}\supseteq\dots\supseteq U_{k}$, and set $n=|U_{k}|$. Let $\mf F_4,\dots,\mf F_g$ be collections of sets of triangles of $K_N$, where each set in $\mf F_j$ contains $j-2$ triangles and each $\mf F_j$ is $z$-well-spread with respect to our sequence of sets. Let $\mc W$ be the set of triangles in $K_N$, and let $\vec{\pi}$ be the weight system defined by $\pi_T = w/|U_{\on{lev}(T)}|$ for all $T\in \mc W$.
\begin{enumerate}
    \item For $4\le j\le g$, $0\le c\le j-5$, and distinct triangles $T,T'$ in $K_N$, let $\mf K^{(1)}_{T,T',j,c}$ be the multiset of sets of triangles constructed as follows. Consider each $\mc S\in \mf F_j$ such that $T,T'\in \mc{S}$, and consider each subset $\mc O\subseteq \mc S\setminus \{T,T'\}$ containing exactly $j-c-4$ triangles, all of which are within $U_k$. For each such $(\mc S,\mc O)$, add a copy of $\mc{S}\setminus(\mc{O}\cup\{T,T'\})$ to $\mf K^{(1)}_{T,T',j,c}$. Then $\kappa(\mf K^{(1)}_{T,T',j,c}) = O_{g,k}(zw^{g}n^{j-c-5})$.
    
    \item For an edge $e$ and triangle $T$ in $K_N$ with $e\not\subseteq T$, let $\mf K^{(2)}_{e,T}$ be the multiset of sets of triangles constructed as follows. Consider each triangle $T'\neq T$, fully within $U_k$ and containing the edge $e$, and consider each $\mc{S}\in\bigcup_{j=4}^g\mf F_j$ such that $T,T'\in \mc{S}$. For each such $(T',\mc S)$, add a copy of $\mc{S}\setminus\{T,T'\}$ to $\mf K^{(2)}_{e,T}$. Then $\kappa(\mf K^{(2)}_{e,T}) = O_{g,k}(zw^{g})$.
    
    \item For (not necessarily distinct) triangles $T,T'$ in $K_N$, let $\mf K^{(3)}_{T,T'}$ be the multiset of sets of triangles constructed as follows. Consider each triangle $T^\ast\notin\{T,T'\}$ fully within $U_k$, and consider each pair of distinct configurations $\mc{S},\mc{S}'\in\bigcup_{j=4}^g\mf{F}_j$ such that $\{T,T',T^\ast\}\cap \mc S = \{T,T^\ast\}$ and $\{T,T',T^\ast\}\cap\mc{S}' = \{T',T^\ast\}$. For each such $(T^\ast,\mc S,\mc S')$, add a copy of $(\mc{S}\cup\mc{S}')\setminus\{T,T',T^\ast\}$ to $\mf K^{(3)}_{T,T'}$. Then $\kappa(\mf K^{(3)}_{T,T'}) = O_{g,k}(z^2w^{2g})$.
    
    \item For a triangle $T$ in $K_N[U_k]$, $4\le j\le g$, and $1\le c\le j-4$, let $\mf K^{(4)}_{T,j,c-1}$ be the multiset of sets of triangles constructed as follows. Consider each $\mc{S}\in\mf F_j$ containing $T$, and consider each subset $\mc{O}\subseteq\mc{S}\setminus\{T\}$ containing exactly $j-c-2$ triangles, all of which are within $U_k$. Then, consider each $\mc{S}'\in\bigcup_{j'=5}^g\mf F_{j'}$ such that $|\mc{S}'\cap(\mc{O}\cup\{T\})| = 2$ and $\mc{S}'\not\subseteq\mc{S}$. For each such $(\mc S,\mc S',\mc O)$, add a copy of $(\mc{S}\cup\mc{S}')\setminus(\mc{O}\cup\{T\})$ to $\mf K^{(4)}_{T,j,c-1}$. Then $\kappa(\mf K^{(4)}_{T,j,c-1}) = O_{g,k}(z^2w^{2g}n^{j-c-3})$.
\end{enumerate}
\end{lemma}
\begin{proof}[Proof of \cref{lem:nibble-moments}(1)]
Fix any set of triangles $\mc H$. We wish to prove an upper bound on the weight $\psi(\mf K^{(1)}_{T,T',j,c},\mc{H})$. Recall the representation of each $\mc K\in \mf K^{(1)}_{T,T',j,c}$ in the form $\mc{K} = \mc S\setminus (\mc O\cup \{T,T'\})$, and note that we can write $\psi(\mf K^{(1)}_{T,T',j,c},\mc{H})=\psi(\mf K')$, where $\mf K'=\{\mc K\setminus \mc H:\mc K\in \mf K^{(1)}_{T,T',j,c},\mc K\supseteq \mc H\}$.

Let $\mc{Q} = \mc{H}\cup\{T,T'\}$ and recall the notation in \cref{lem:general-moments}(1). Observe that we have the multiset inclusion $\mf K'\subseteq \mf{A}_{\mc{Q},j,c-|\mc{H}|}^{(1)}$: we can witness this via the mapping
\[(\mc{S},\mc{O})\mapsto(\mc{E},\mc{F},\mc{Z}) = (\mc{S},\mc{S}\setminus(\mc{Q}\cup\mc{O}),\mc{O}).\]
To be explicit, this is an injective mapping from the data defining $\mf K'$ to the data defining $\mf{A}_{\mc{Q},j,c-|\mc{H}|}^{(1)}$, and $(\mc{S}\setminus(\mc{O}\cup\{T,T'\}))\setminus\mc{H} = \mc{F}$ means that the set corresponding to $(\mc{S},\mc{O})$ in $\mf K'$ and the set corresponding to $(\mc{E},\mc{F},\mc{Z})$ in $\mf{A}_{\mc{Q},j,c-|\mc{H}|}^{(1)}$ are the same. Furthermore, we easily check that the conditions for $(\mc{E},\mc{F},\mc{Z})$ in the definition of $\mf{A}_{\mc{Q},j,c-|\mc{H}|}^{(1)}$ are satisfied.

Let $v=v^j(\mc H\cup\{T,T'\})=(|\mc H|+2)+3$ (here we are using that $\mc{H}\cup\{T,T'\}\subsetneq\mc{S}$, since $|\mc{O}|=j-c-4\ge 1$ and $\mc O$ is disjoint from $\mc H\cup \{T,T'\}$). By our multiset inclusion and \cref{lem:general-moments}(1) we have
\[\psi(\mf K^{(1)}_{T,T',j,c},\mc{H})\le \psi(\mf{A}_{\mc{Q},j,c-|\mc{H}|}^{(1)}) = O_{g,k}(zw^gn^{j-v-(c-|\mc{H}|)}) = O_{g,k}(zw^gn^{j-c-5}).\]
The desired result follows.
\end{proof}
\begin{proof}[Proof of \cref{lem:nibble-moments}(2)]
Fix any $\mc H$, and recall the representation of each $\mc K\in \mf K^{(2)}_{e,T}$ in the form $\mc S\setminus\{T,T'\}$. We write
\[\psi(\mf K^{(2)}_{e,T},\mc{H})= \sum_{j=4}^g\psi(\mf K_j'),~\text{where}~\mf K'_j=\{\mc K\setminus \mc H:\mc K\in \mf K^{(2)}_{e,T},\;|\mc K|=j-4,\;\mc K\supseteq\mc H \}\]
(that is to say, $\psi(\mf{K}_j')$ is the contribution to $\psi(\mf K^{(2)}_{e,T})$ arising from data $(T',\mc S)$ with $\mc{S}\in\mf F_j$). We will prove that $\psi(\mf K_j')=O_{g,k}(zw^g)$ for each $j$; the desired result will follow.

\medskip
\textit{Case 1: $|\mc H|=j-4$. } In this case $\mf K_j'$ consists of copies of the empty set. If $j=4$, the number of such copies is $\psi(\mf K_j')\le z=O_{g,k}(zw^g)$, by \cref{WS3} applied to $T$ and $e$. Otherwise, if $j>4$ then the number of copies is at most $z$, by \cref{WS1} applied with $\mc R=\mc H\cup \{T\}$.

\medskip
\textit{Case 2: $|\mc H|<j-4$. } In this case we will further sum over possibilities for $T'$: let $\mf K'_{j,T'}$ be the submultiset of $\mf K'_j$ arising from a particular choice of $T'$. Let $\mc Q=\mc H\cup \{T,T'\}$, let $v=v^j(\mc H\cup \{T,T'\})\ge (|\mc H|+2)+3$, and note that $\mf K'_{j,T'}\subseteq \mf{A}_{\mc{Q},j,j-4-|\mc H|}^{(1)}$: the inclusion is witnessed by the injective mapping $\mc S\mapsto (\mc{E},\mc F,\mc Z)=(\mc{S},\mc{S}\setminus \mc{Q},\emptyset)$. So, by \cref{lem:general-moments}(1) we have $\psi(\mf K_{j,T'}') = O_{g,k}(zw^gn^{j-v-(j-4-|\mc H|)})=O_{g,k}(zw^g/n)$. The desired bound follows by summing over at most $n$ choices of $T'\supseteq e$.
\end{proof}

\begin{proof}[Proof of \cref{lem:nibble-moments}(3)]
Fix any $\mc H$, and recall the representation of each $\mc K\in \mf K^{(3)}_{T,T'}$ in the form $\mc S\cup \mc S'\setminus \{T,T',T^\ast\}$. Write $\mc S\preceq_\mc H\mc S'$ when $|\mc H\cap \mc S|> |\mc H\cap \mc S'|$ or when $|\mc H\cap \mc S|= |\mc H\cap \mc S'|$ and $|\mc S|\le |\mc S'|$. We slightly modify the definition of $\mf K^{(3)}_{T,T'}$ to only consider choices of $(T^\ast,\mc S,\mc S')$ such that $\mc S\preceq_\mc H\mc S'$ (due to the symmetry between $\mc S$ and $\mc S'$, this affects $\psi(\mf K^{(3)}_{T,T'},\mc{H})$ by a factor of at most $2$). Let $\mf K'=\{\mc K\setminus \mc H:\mc K\in \mf K^{(3)}_{T,T'},\;\mc K\supseteq \mc H\}$, so $\psi(\mf K^{(3)}_{T,T'},\mc H)=\psi(\mf K')$. Keep in mind that for a triple $(T^\ast,\mc S,\mc S')$ contributing to $\mf{K}'$, we always have $\mc{H}\subseteq\mc{S}\cup\mc{S}'\setminus\{T,T',T^\ast\}$.

Unlike the previous parts of \cref{lem:nibble-moments}, it is not quite as convenient to divide the proof into logically independent cases. It is easier to prove two different bounds by applying different parts of \cref{lem:general-moments}, and to patch these together to deduce the desired bound $\psi(\mf K')=O_{g,k}(z^2w^{2g})$.

\medskip
\emph{Claim 1. }
Let $\mf K_{\mc Q,\mc Q',j,j',v',q}'$ be the submultiset of $\mf K'$ arising from data $(T^\ast,\mc S,\mc S')$ such that $\mc S\in \mf F_j$ and $\mc S'\in \mf F_{j'}$, and such that
\[\mc Q=(\mc S\cap \mc H)\cup \{T\},\quad \mc Q'=(\mc S'\cap \mc H)\cup \{T'\},\quad v'=v^{j'}(\mc S'\cap(\mc S\cup \mc Q')),\quad q=|\mc S\cap \mc S'\setminus\{T,T',T^\ast\}|.\]
Let $v=v^j(\mc Q)$. Then $\psi(\mf K_{\mc Q,\mc Q',j,j',v',q}')= O_{g,k}(z^2w^{2g}n^{8+q+|\mc H|-v-v'})$.
\medskip

\emph{Claim 2. }
Suppose that $\mc H=\emptyset$. Let $\mf K_{j}''$ be the submultiset of $\mf K'$ arising from data $(T^\ast,\mc S,\mc S')$ such that $\mc S,\mc S'\in \mf F_{j}$ and $\mc S\setminus \{T\}=\mc S'\setminus \{T'\}$. Then $\psi(\mf K_{j}'')=O_{g,k}(zw^g)$.
\medskip

Before proving these claims, we show how they can be combined to prove that $\psi(\mf K')= O_{g,k}(z^2w^{2g})$. Consider any $\mc K\in\mf{K}'$ arising from data $(T^\ast,\mc S,\mc S')$, where $\mc S\in \mf F_j$ and $\mc S\in \mf F_{j'}$. We will show that either
\begin{enumerate}
    \item[(A)] $v^j((\mc S\cap \mc H)\cup \{T\})+v^{j'}(\mc S'\cap(\mc H\cup \{T'\}\cup \mc S))\ge 8+|\mc S\cap \mc S'\setminus\{T,T',T^\ast\}|+|\mc H|$, or
    \item[(B)] $j=j'$ and $\mc S\setminus \{T\}=\mc S'\setminus \{T'\}$.
\end{enumerate}
It will follow that $\psi(\mf K')$ is upper-bounded by a sum of $O_g(1)$ terms of the form $\psi(\mf K_{\mc Q,\mc Q',j,j',v',q}')$ with $v+v'\ge q+|\mc H|+8$, plus a sum of $O_g(1)$ terms of the form $\psi(\mf K_{j}'')$. The conclusion of the lemma will then follow from Claims 1 and 2.

So, we show that either (A) or (B) must hold. Let $\mc{S}\in\mf{F}_j$, $\mc{S}'\in\mf{F}_{j'}$. We can write $\mc S'\cap(\mc H\cup\{T'\}\cup\mc S)$ as the disjoint union $\{T',T^\ast\}\cup (\mc S'\cap \mc H)\cup ((\mc S'\cap \mc S) \setminus (\mc H\cup \{T',T^\ast\}))$, so
\begin{align*}
|(\mc S\cap \mc H)\cup \{T\}|+|\mc S'\cap(\mc H\cup\{T'\}\cup\mc S)|&=|\mc{S}\cap\mc{H}|+|\mc{S}'\cap\mc{H}|+|(\mc{S}'\cap\mc{S}) \setminus(\mc{H}\cup\{T',T^\ast\})|+3\\
&=|\mc H|+|(\mc S\cap \mc S') \setminus\{T,T',T^\ast\}|+3.
\end{align*}
Recalling the definitions of $v^j(\cdot)$ and $v^{j'}(\cdot)$, and noting that $T^\ast\notin(\mc S\cap \mc H)\cup \{T\}$ and $T',T^\ast\in \mc S'\cap(\mc H\cup\{T'\}\cup \mc S)$, we see that (A) holds unless $|(\mc S\cap \mc H)\cup \{T\}|=1$ and $|\mc S'\cap(\mc H\cup\{T'\}\cup \mc S)|=j'-2$. So, assume that this is the case. It follows from $|(\mc S\cap \mc H)\cup \{T\}|=1$ that $\mc S\cap \mc H=\emptyset$, and since we are assuming that $\mc S\preceq_{\mc H}\mc S'$, we further deduce that $\mc H=\emptyset$ and $j\le j'$. But $|\mc S'\cap(\mc H\cup\{T'\}\cup \mc S)|=j'-2$ means that $\mc S'\setminus \{T'\}\subseteq \mc S\setminus \{T\}$, meaning that $j'\le j$. So, we have $j=j'$, and our assumption $|\mc S'\cap(\mc H\cup\{T'\}\cup \mc S)|=j'-2$ implies that (B) holds.

Now, to complete the proof of the lemma it suffices to prove Claims 1 and 2. For Claim 2, we consider the injective mapping $(T^\ast,\mc S,\mc S')\mapsto (\mc E,\mc E',\mc F,\mc Z)=(\mc S, \mc S', \mc S\setminus\{T,T^\ast\}, \{T^\ast\})$ to see that $\mf K_{j,j'}'\subseteq \mf{A}_{T,T',j,j-4}^{(3)}$. So, the desired result follows from \cref{lem:general-moments}(3).

For Claim 1, we have $(\mc S\setminus(\mc Q\cup \{T^\ast\}))\cup (\mc S'\setminus(\mc Q'\cup \{T^\ast\})) = (\mc S\cup \mc S')\setminus (\mc H\cup \{T,T',T^\ast\})$ (recall that if $T\neq T'$ then $T\notin \mc S'$ and $T'\notin \mc S$, by the definition of $\mf K^{(3)}_{T,T'}$) and
\begin{align*}
|(\mc S\cup \mc S')\setminus (\mc H\cup \{T,T',T^\ast\})|&=|\mc S\setminus \{T,T^\ast\}|+|\mc S'\setminus \{T',T^\ast\}|-|(\mc S\setminus \{T,T^\ast\})\cap (\mc S'\setminus \{T',T^\ast\})|-|\mc H|\\&=(j-4)+(j'-4)-q-|\mc H|.\end{align*}
Let $f=j+j'-8-q-|\mc H|$; we claim that $\mf K_{\mc Q,\mc Q',j,j',v',q}'\subseteq \mf{A}_{\mc{Q},\mc{Q}',j,j',v,v',f}^{(2)}$. The inclusion is given by the injective mapping
\[(T^\ast,\mc S,\mc S')\mapsto(\mc E,\mc E',\mc F,\mc F',\mc Z,\mc Z') = (\mc S,\mc{S}',\mc{S}\setminus(\mc{Q}\cup\{T^\ast\}),\mc{S}'\setminus(\mc{Q}'\cup\{T^\ast\}),\{T^\ast\},\{T^\ast\}).\]
The conditions for $\mf{A}_{\mc{Q},\mc{Q}',j,j',v,v',f}^{(2)}$ are easily checked; note in particular that under this mapping $\mc{F}\cup\mc{F}' = (\mc{S}\cup\mc{S}')\setminus(\mc{H}\cup\{T,T',T^\ast\})$ and $\mc{Q}\cup\mc{Z}\cup\mc{Q}'\cup\mc{Z}' = \mc{H}\cup\{T,T',T^\ast\}$ are disjoint. The desired result then follows from \cref{lem:general-moments}(2).
\end{proof}

\begin{proof}[Proof of \cref{lem:nibble-moments}(4)]
Fix any $\mc H$ and recall the representation of each $\mc K\in \mf K^{(4)}_{T,j,c-1}$ in the form $(\mc S\cup \mc S')\setminus (\mc O\cup \{T\})$. Let $\mf K'=\{\mc K\setminus \mc H: \mc K\in \mf K^{(4)}_{T,j,c-1},\;\mc K\supseteq \mc H\}$, so $\psi(\mf K^{(4)}_{T,j,c-1},\mc H)=\psi(\mf K')$. We may assume $T \notin \mc H$, for otherwise $\psi(\mf K') = 0$. Similarly to the proof of \cref{lem:nibble-moments}(3), we prove three different bounds by applying different parts of \cref{lem:general-moments}, and patch these together to deduce the desired bound $\psi(\mf K')= O_{g,k}(z^2w^{2g}n^{j-c-3})$.

\medskip
\textit{Claim 1. }
Let $\mf K_{j',q,v,v'}'$ be the submultiset of $\mf K'$ arising from data $(\mc S,\mc S',\mc O)$ such that $\mc S'\in \mf F_{j'}$ and such that
\[v=v^j((\mc S\cap \mc H)\cup \{T\}),\quad q=|(\mc S\cap\mc S') \setminus (\mc O\cup\{T\})|,\quad v'=v^{j'}(\mc S'\cap (\mc S\cup \mc H)).\]
Then $\psi(\mf K_{j',q,v,v'}')= O_{g,k}(z^2w^{2g}n^{j-v-v'-((c-1)-4-q-|\mc H|)})$.
\medskip

\textit{Claim 2. }
Suppose that $\mc H\ne \emptyset$. Let $\mf K_{j',v,v'}''$ be the submultiset of $\mf K'$ arising from data $(\mc S,\mc S',\mc O)$ such that $\mc S'\in \mf F_{j'}$, $\mc{S}'\subseteq\mc{S}\cup\mc{H}$, $\mc H\cap \mc S=\emptyset$, and
\[v=v^j(\{T\}\cup(\mc S'\cap \mc S)),\quad v'=v^{j'}(\mc H\cup(\mc{S}'\cap\{T\})).\]
Then $\psi(\mf K_{j',v,v'}'')=O_{g,k}(z^2 w^g n^{j'+j-v'-v-(c-1)})$.
\medskip

\textit{Claim 3. }
Suppose that $|\mc H|=1$. Let $\mf K'''$ be the submultiset of $\mf K'$ arising from data $(\mc S,\mc S',\mc O)$ such that $\mc{S'}\setminus\mc{H}=\mc S\setminus \{T\}$.
Then $\psi(\mf K''')= O_{g,k}(z^2w^{2g}n^{j-c-3})$.
\medskip

Before proving the above three claims, we show how they can be used together to prove that $\psi(\mf K')= O_{g,k}(z^2w^{2g}n^{j-c-3})$. Consider any $\mc K\in \mf{K}'$ arising from data $(\mc S,\mc S',\mc O)$, where $\mc S'\in \mf F_{j'}$. We will show that either
\begin{enumerate}
    \item[(A)] $v^j((\mc S\cap \mc H)\cup \{T\})+v^{j'}(\mc S'\cap (\mc S\cup \mc H))\ge |(\mc S\cap \mc S')\setminus (\mc O\cup\{T\})|+|\mc H|+8$, or
    \item[(B)] $\mc H\ne \emptyset$, $\mc H\cap\mc S=\emptyset$, $\mc{S}'\subseteq\mc{S}\cup\mc{H}$, and $v^j(\{T\}\cup(\mc S'\cap \mc S))+v^{j'}(\mc H\cup(\mc{S}'\cap\{T\}))\ge j'+4$, or
    \item[(C)] $|\mc H| = 1$ and $\mc{S'}\setminus\mc{H}=\mc S\setminus \{T\}$.
\end{enumerate}
It will follow that $\psi(\mf K')$ is upper-bounded by a sum of $O_g(1)$ terms of the form $\psi(\mf K_{j',q,v,v'}')$ with $v+v'\ge q+|\mc H|+8$, plus a sum of $O_g(1)$ terms of the form $\psi(\mf K_{j',v,v'}'')$ with $v+v'\ge j'+4$, plus $\psi(\mf K''')$. The conclusion of the lemma will then follow from Claims 1--3.

So, we show that one of (A)--(C) must hold. Noting that $T\notin \mc H$, we have
\begin{align*}
|(\mc{S}\cap\mc{H})\cup\{T\}|+|\mc{S}'\cap(\mc{S}\cup\mc{H})|
&=1+|\mc{H}|+|\mc S'\cap \mc S|= 1+|\mc{H}|+|(\mc S\cap \mc S')\!\setminus\!(\mc O\cup\{T\})|+|\mc{S}'\cap(\mc{O}\cup\{T\})|.
\end{align*}
Recalling the definitions of $v^j(\cdot)$ and $v^{j'}(\cdot)$, and recalling that $|\mc{S}'\cap(\mc{O}\cup\{T\})|=2$, we see that (A) holds unless $(\mc{S}\cap\mc{H})\cup\{T\}$ has size in $\{1,j-2\}$ and $\mc{S}'\cap(\mc{S}\cup\mc{H})$ has size in $\{1,|\mc S'|\}$. Assume this is the case. Note that $\mc{S}'\cap(\mc{S}\cup\mc{H})$ contains at least the two triangles in $\mc S' \cap (\mc O\cup\{T\})$, so by the second of the two assumptions we have just made, $\mc{S}'\subseteq\mc{S}\cup\mc{H}$. But by the definition of $\mf K^{(4)}_{T,j,c-1}$ we have $\mc{S}'\not\subseteq\mc{S}$, so $\mc{H}\neq\emptyset$. Also, $\mc O$ is disjoint from $\mc H\cup \{T\}$, so $S\not\subseteq \mc H\cup \{T\}$. So, by the first of our two assumptions, $|(\mc{S}\cap\mc{H})\cup\{T\}|=1$, meaning that $\mc S\cap \mc H= \emptyset$.

Now, using the facts that $\mc{S}'\subseteq\mc{S}\cup\mc{H}$ and $\mc S\cap \mc H= \emptyset$, we have
\[|\{T\}\cup(\mc{S}'\cap\mc{S})|+|\mc{H}\cup(\mc{S}'\cap\{T\})| = |\mc{S}'\cap\mc{S}|+|\mc{H}|+1 = |\mc{S}'\setminus\mc{H}|+|\mc{H}|+1 = (|\mc S'|+2)-1.\]
It follows that (B) holds unless $|\mc{H}\cup(\mc{S}'\cap\{T\})|\in \{1,|\mc S'|\}$ and $|\{T\}\cup(\mc{S}'\cap\mc{S})|\in \{1,j-2\}$. Assume this is the case. We have $\mc{H}\subseteq\mc{S}'$, and since $|\mc{S}'\cap(\mc{O}\cup\{T\})| = 2$ it is not possible to have $\mc{H}\cup(\mc{S}'\cap\{T\})=\mc S'$, so $|\mc{H}\cup(\mc{S}'\cap\{T\})|=1$. Since $\mc{H}$ is nonempty it follows that $|\mc{H}| = 1$ and $T\notin\mc{S}'$. Also, $\{T\}\cup(\mc{S}'\cap\mc{S})$ contains two triangles in $\mc{O}\cup\{T\}$, so $\{T\}\cup(\mc{S}'\cap\mc{S}) = \mc{S}$. Combining all the observations and assumptions we have made so far, we see that (C) must hold.

To complete the proof of the lemma it suffices to prove Claims 1--3. For Claim 1, we consider the injective mapping
\begin{align*}
    (\mc S,\mc S',\mc O)&\mapsto(\mc E,\mc E',\mc F,\mc F',\mc Z,\mc Z')=(\mc S,\;\mc S',\;\mc S\setminus(\mc O\cup\mc H\cup \{T\}),\;\mc S'\setminus(\mc O\cup\mc H\cup \{T\}),\;\mc O,\;\mc S'\cap (\mc O\cup \{T\})).
\end{align*}
With $j'=|\mc S'|$ and $q=|(\mc S \cap \mc S') \setminus (\mc O\cup\{T\})|$ we have
\begin{align*}
|\mc{F}\cup\mc{F}'| = |\mc{S}\cup\mc{S}'|-|\mc{O}\cup\mc{H}\cup\{T\}|&= (j-2+j'-2-|\mc{S}\cap\mc{S}'|)-(j-c-2+|\mc{H}|+1)\\
&= j'+c-3-|\mc{S}\cap\mc{S}'\cap(\mc{O}\cup\{T\})|-q-|\mc{H}|\\
&= (c-1)+(j'-4)-q-|\mc{H}|,
\end{align*}
and we can check that $\mf K_{j',q,v,v'}'\subseteq \mf{A}_{\mc{Q},\mc{Q}',j,j',v,v',f}^{(2)}$, where $\mc Q=(\mc S\cap \mc H)\cup\{T\}$, $\mc Q'=\mc S'\cap \mc H$ and $f=(c-1)+(j'-4)-q-|\mc H|$. (Here we use the fact that $T\in K_N[U_k]$ to verify that the triangles of $\mc{Z}'$ are all within $U_k$.) The statement of Claim 1 then follows from \cref{lem:general-moments}(2).

For Claim 2, we consider the ``reverse order'' mapping
\begin{align*}
    (\mc S,\mc S',\mc O)&\mapsto(\mc E,\mc E',\mc F,\mc F',\mc Z,\mc Z')=(\mc S',\;\mc S,\;\mc S'\setminus(\mc O\cup\mc H\cup \{T\}),\;\mc S\setminus(\mc O\cup\mc H\cup \{T\}),\;\mc S'\cap\mc O,\;\mc O).
\end{align*}
In the setting of Claim 2, recalling that $\mc S' \subseteq \mc S \cup \mc H$ (implying $\mc F \subseteq \mc F'$), $\mc H\ne \emptyset$, and $\mc H\cap\mc S=\emptyset$, we have 
\[
|\mc{F}\cup\mc{F}'| = |\mc F'| = |\mc S \setminus (\mc O \cup \mc H \cup \{T\})| = |\mc{S}\setminus(\mc{O}\cup\{T\})| = c-1,
\]
and we can check that $\mf K_{j',v,v'}''\subseteq \mf{A}_{\mc{Q},\{T\},j',j,v',v,c-1}^{(2)}$, where $\mc Q=\mc H\cup(\mc S'\cap \{T\})$. The statement of Claim 2 then follows from \cref{lem:general-moments}(2).

Finally, for Claim 3 we let $T'$ be the single triangle in $\mc H$, consider the injective mapping
\[(\mc S,\mc S',\mc O)\mapsto(\mc E,\mc E',\mc F,\mc Z) = (\mc{S},\mc{S}',\mc{S}\setminus(\{T\}\cup\mc{O}),\mc{O}),\]
check that $\mf K'''\subseteq \mf{A}_{T,T',j,c-1}^{(3)}$, and apply \cref{lem:general-moments}(3).
\end{proof}

Recall from the outline in \cref{subsec:sparsification} that we need an initial ``sparsification'' step which involves a special instance of our high-girth triple process. In order to analyze this process we need a variation on \cref{lem:nibble-moments}(2), in which we prove a sharper weight bound but in a simpler setting without a descending sequence of sets $U_0\supseteq\dots\supseteq  U_k$ (therefore we do not need \cref{lem:general-moments} or the notion of well-spreadness).
\begin{lemma}\label{lem:special-nibble-weight-lemma}
Consider an absorbing structure $H$ (with distinguished set $X$) satisfying \cref{AB1} and \cref{AB2} in \cref{thm:absorbers} for $g\in\mb{N}$, embedded in $K_N$. Let $\mc{B}$ and $\mf F_j^\mc{B}$ be as in \cref{lem:absorber-well-spread}, and suppose $|\mc{B}|^{2g}\le N^\beta$ for some $\beta\le 1$. Let $\mc W$ be the set of triangles in $K_N$ and let $\vec\pi$ be the weight system defined by $\pi_T=1/N$ for all $T\in\mc W$.

For a set of vertices $U$ and an edge $e$, both disjoint from $V(H)\setminus X$, and a triangle $T$, let $\wt{\mf K}_{e,T,U}$ be the multiset of sets of triangles constructed as follows. Consider each triangle $T'\ne T$ containing the edge $e$ with its third vertex in $U$, and consider each $\mc S\in\bigcup_{j=4}^g\mf F_j^{\mf{B}}$ such that $T,T'\in \mc S$.
For each such $(T',\mc S)$, add a copy of $\mc S\setminus \{T,T'\}$ to $\wt{\mf K}_{e,T,U}$. Then
$\kappa(\wt{\mf K}_{e,T,U})=O_g(1+|U|N^{\beta-1})$.
\end{lemma}
\begin{proof}
The proof proceeds in basically the same way as \cref{lem:nibble-moments}(2). Fix $\mc{H}$ and define $\wt{\mf K}_j'$ as in the proof of \cref{lem:nibble-moments}(2), using $\wt{\mf K}_{e,T,U}$ in place of $\mf K^{(2)}_{e,T}$. Explicitly, write
\[\psi(\wt{\mf K}_{e,T,U},\mc{H})= \sum_{j=4}^g\psi(\wt{\mf K}_j'),~\text{where}~\wt{\mf K}_j'=\{\mc K\setminus \mc H:\mc K\in\wt{\mf K}_{e,T,U},\;|\mc K|=j-4,\;\mc K\supseteq\mc H\}.\]

\textit{Case 1: $|\mc H|=j-4$. } In this case $\wt{\mf K}_j'$ consists of copies of the empty set. The number of such copies is at most the number of Erd\H{o}s configurations of the form $\mc E=\mc M\cup \mc H\cup \{T,T'\}$, where $T'$ contains no vertices of $V(H)\setminus X$ and $\mc M \subseteq \mc{B}$. By \cref{AB2} with $\mc{R} = \mc{H}\cup\{T\}$, we must have $\mc M\subseteq \mc L_{\mc R}$, so there are only $O_g(1)$ choices for $\mc M$. But $\mc M\cup \mc R$ determines the vertex set of $\mc E$, by \cref{lem:erdos-minimality}(1). That is to say, $\psi(\wt{\mf K}_j')=O_g(1)$.

\medskip
\textit{Case 2: $|\mc H|<j-4$. } In this case we further sum over possibilities for $T'$: let $\wt{\mf K}'_{j,T'}$ be the submultiset of $\wt{\mf K}'_j$ arising from a particular choice of $T'$. For a particular choice of $T'$, let $v=v^j(\mc H\cup \{T,T'\})\ge (|\mc H|+2)+3$. Using \cref{lem:absorber-well-spread} (with $k = 0$) and \cref{WS1} arising from it applied to $\mc{H}\cup\{T,T'\}$, we compute $\psi(\wt{\mf K}_{j,T'}') = O_{g,k}(N^{\beta+j-v-(j-4-|\mc H|)})=O_{g,k}(N^{\beta-1})$. Summing over at most $|U|$ choices of $T'\supseteq e$ yields the desired bound $\psi(\wt{\mf K}_j') = O_{g,k}(|U|N^{\beta-1})$.
\end{proof}

The estimate in the following lemma will be used to prove a high-probability upper bound on the number of forbidden configurations containing a given triangle at a given stage of the proof of \cref{thm:main}. Although the computations will be less complicated than \cref{lem:nibble-moments}, here we need a more precise estimate, taking into account both the triangles that have been selected at previous stages, and the edges that were not covered by the initial sparsification step. To encode this information, the ground set of our weight system will contain a mixture of triangles and edges.

Henceforth, we will be less detailed in our applications of \cref{lem:general-moments}; as in the proofs of the different parts of \cref{lem:nibble-moments}, whenever applying \cref{lem:general-moments} we should formally define an injective map which witnesses a multiset inclusion, and deduce an inequality of weights. From now on all applications of \cref{lem:general-moments} will be comparatively simple. In particular we will only need \cref{lem:general-moments}(1).

The following lemma will be used in \cref{sub:iter-nib}.

\begin{lemma}\label{lem:nibble-config}
Fix positive real numbers $p\le 1$ and $y,z$, a pair of integers $4\le j\le j'\le g$, a descending sequence of subsets $V(K_N)=U_0\supseteq\dots\supseteq U_k$, and let $n=|U_k|$.
\begin{itemize}
    \item Let $\mf F_{j'}$ be a collection of sets of $j'-2$ triangles in $K_N$, which is $(y,z)$-well-spread with respect to our sequence of sets. \item Let $\mc W_1$ be the set of triangles in $K_N$, and let $\mc{W}_2$ be the set of edges of $K_N[U_k]$.
    \item Let $\vec{\pi}$ be the weight system for $\mc{W} = \mc{W}_1\cup\mc{W}_2$ defined by $\pi_T = 1/|U_{\on{lev}(T)}|$ for each $T\in \mc W_1$ and $\pi_e = p$ for each $e\in\mc{W}_2$.
\end{itemize}
Let $T$ be a triangle in $K_N[U_k]$, and let $\mf{B}_{T,j,j'}$ be the multiset (of sets of edges and triangles) constructed as follows. Consider every $\mc{S}\in\mf F_{j'}$ and every partition $\mc{S} = \{T\}\cup\mc{Z}\cup\mc{F}$, such that $|\mc{Z}| = j-3$ and every triangle in $\mc{Z}$ is within $U_k$. For each such $(\mc{S},\mc{F},\mc{Z})$, add a copy of $E(\mc{Z})\cup\mc{F}$ (i.e., $3(j-3)$ edges and $j'-j$ triangles) to $\mf{B}_{T,j,j'}$.

If $zn^{-1}\le yp^{3(j-3)}$ then $\kappa(\mf{B}_{T,j,j'}) = O_{g,k}(yp^{3(j-3)}n^{j-3})$.
\end{lemma}
\begin{proof}
First note that $\psi(\mf{B}_{T,j,j'},\emptyset) = O_{g,k}(yp^{3(j-3)}n^{j-3})$. Indeed, if we were to ignore the edges in $E(\mc Z)$ in the definition of $\mf{B}_{T,j,j'}$, we would have a bound of the form $\psi(\mf A^{(1)}_{\{T\},j',j'-j})=O_{g,k}(yn^{j-3})$ by \cref{lem:general-moments}(1). Then, the $3(j-3)$ edges in the triangles in $\mc Z$ contribute a factor of $p^{3(j-3)}$.

So it suffices to study weights of the form $\psi(\mf{B}_{T,j,j'},\mc H^2\cup \mc H^3)$, where $\mc H^2\subseteq\mc W_2$ is a set of edges and $\mc H^3\subseteq\mc W_1$ is a set of triangles, and $\mc H^2\cup \mc H^3\ne \emptyset$. Fix such $\mc H^2,\mc H^3$. Actually, in this case where $\mc H^2\cup \mc H^3\ne \emptyset$, we will be able to get away with quite a crude bound that completely ignores the edge-weights $\pi_e$ (only taking the triangle-weights $\pi_T$ into account). Let $\pvec{\pi}$ be the modified weight system where we set $\pi_e' = 1$ for all $e\in\mc W_2$ and $\pi_S' = \pi_S$ for all $S \in\mc W_1$ (i.e., we ignore the influence of the edges in $\mc W_2$), and note that
\[\psi^{(\vec{\pi})}(\mf{B}_{T,j,j'},\mc{H}^2\cup\mc{H}^3)\le \psi^{(\pvec{\pi})}(\mf{B}_{T,j,j'},\mc{H}^2\cup\mc{H}^3).\]
For convenience, from now on we write $\psi'$ instead of $\psi^{(\pvec{\pi})}$, and seek to bound $\psi'(\mf{B}_{T,j,j'},\mc{H}^2\cup\mc{H}^3)$.

\medskip
\textit{Case 1: $\mc H^3 \neq \emptyset$. } In this case, we apply \cref{lem:general-moments}(1) to obtain \[\psi'(\mf{B}_{T,j,j'},\mc H^2\cup\mc{H}^3)\le\psi'(\mf{A}_{\{T\},j',j'-j-|\mc H^3|}^{(1)}) = O_{g,k}(zn^{j-3-|\mc H^3|}) = O_{g,k}(zn^{j-4}).\]
The desired bound then follows from the assumption $z\le yp^{3(j-3)}n$.

\medskip
\textit{Case 2: $\mc H^3=\emptyset$ and $j'>4$. } In this case we fix an edge $e\in\mc{H}^2$ (recall that $\mc H^2\cup \mc H^3\ne \emptyset$) and again apply \cref{lem:general-moments}(1) to obtain
\begin{align*}\psi'(\mf{B}_{T,j,j'},\mc H^2\cup\mc{H}^3)\le\psi'(\mf{B}_{T,j,j'},\{e\})&\le\sum_{\substack{T'\supseteq e\\T'\subseteq U_k}}\psi'(\mf{A}_{\{T,T'\},j',j'-j}^{(1)})= n\cdot O_{g,k}(zn^{j-v^{j'}(\{T,T'\})}).\end{align*}
In the second inequality, we may impose $T'\subseteq U_k$ since each term in $\psi'(\mf{B}_{T,j,j'},\{e\})$ comes from data $(\mc{S},\mc{F},\mc{Z})$ for which $e\in\mc{H}^2\subseteq E(\mc{Z})$ (meaning that $T'\in \mc Z$ is within $U_k$).
Now, since we are assuming $j'>4$, we have $v^{j'}(\{T,T'\}) = 5$, and the desired bound follows (again using that $z\le yp^{3(j-3)}n$).

\medskip
\textit{Case 3: $\mc H^3=\emptyset$ and $j'=4$. } In this case we must have $j=j'=4$, and $\psi'(\mf{B}_{T,j,j'},\mc H^2\cup\mc{H}^3)$ is simply the number of $\mc S\in\mf F_j$ which consist of $T$ and a second triangle $T'\supseteq e$ within $U_k$. This number is at most $z\le yp^{3(j-3)}n$ by \cref{WS3} applied to $T$ and $e$.
\end{proof}

For the remaining lemmas in this section it will be convenient to have some general weight estimates concerning the set of triangles containing a given edge.
\begin{definition}\label{def:fan}
For an edge $e$ and a set of vertices $S$ in a graph $K_N$, let $\mc{N}_e(S)$, be the set of triangles in $K_N$ which contain $e$ and a third vertex $v\in S$. We say $\mc N_e(S)$ is the \emph{fan} of $e$ with respect to $S$.
\end{definition}
\begin{lemma}\label{lem:fan-lemma}
Fix positive real numbers $z,w\ge 1$, a positive integer $4\le j\le g$, a descending sequence of subsets $V(K_N)=U_{0}\supseteq\dots\supseteq U_{k}$, and let $n=|U_{k}|$. Let $\mf F_j$ be a $z$-well-spread collection of sets of $j-2$ triangles in $K_N$, with respect to our sequence of sets. Let $\mc{W}$ be the set of triangles in $K_N$, and let $\vec{\pi}$ be the weight system for $\mc{W}$ defined by $\pi_T = w/|U_{\on{lev}(T)}|$ for each $T\in \mc W$.
\begin{enumerate}
    \item Fix an edge $e\subseteq U_k$ and a nonempty set $\mc{Q}$ of triangles in $K_N$. Let $\mf{C}_{\mc{Q},e,j}^{(1)}$ be the multiset of sets of triangles constructed as follows. Consider every  $\mc{S}\in\mf F_j$ and every partition $\mc{S} = \{T\}\cup\mc{Q}\cup\mc{F}$ such that $|\mc{F}| = j-3-|\mc{Q}|$ and $T\in\mc N_e(U_k)$. For each such $(\mc S,\mc F,T)$, add a copy of $\mc F$ to $\mf{C}_{\mc{Q},e,j}^{(1)}$. Then $\psi(\mf{C}_{\mc{Q},e,j}^{(1)}) = O_{g,k}(zw^{g-1})$.
    
    \item Fix an edge $e\subseteq U_k$ and a nonempty set $\mc{Q}$ of triangles in $K_N$. Let $\mf{C}_{\mc{Q},e,j}^{(2)}$ be the multiset of sets of triangles constructed as follows. As in (1), consider every  $\mc{S}\in\mf F_j$ and every partition $\mc{S} = \{T\}\cup\mc{Q}\cup\mc{F}$ such that $|\mc{F}| = j-3-|\mc{Q}|$ and $T\in\mc N_e(U_k)$. For each such $(\mc S,\mc F,T)$, add a copy of $\mc{F}\cup\{T\}$ (instead of $\mc F$, as in (1)) to $\mf{C}_{\mc{Q},e,j}^{(2)}$. Then $\psi(\mf{C}_{\mc{Q},e,j}^{(2)}) = O_{g,k}(zw^g/n)$.
    
    \item Fix distinct edges $e,e'\subseteq U_k$. Let $\mf{C}_{e,e',j}^{(3)}$ be the multiset of sets of triangles constructed as follows. Consider every $\mc{S}\in\mf F_j$ such that $e'$ appears in one of the triangles in $\mc S$, and consider every partition $\mc{S} = \{T\}\cup\mc{F}$ such that $|\mc{F}| = j-3$, $T\in\mc N_e(U_k)$. For each such $(\mc S,\mc F,T)$, add a copy of $\mc F$ to $\mf{C}_{e,e',j}^{(3)}$. Then $\psi(\mf{C}_{e,e',j}^{(3)}) = O_{g,k}(zw^{g-1})$.
    
    \item Fix distinct edges $e,e'\subseteq U_k$. Let $\mf{C}_{e',e,j}^{(4)}$ be defined in the same way as $\mf{C}_{e,e',j}^{(3)}$, except that for each $(\mc S,\mc F,T)$ as in (3), we add a copy of $\mc{S}$ to $\mf{C}_{e',e,j}^{(4)}$ (instead of $\mc F$, as in (3)). Then $\psi(\mf{C}_{e,e',j}^{(4)}) = O_{g,k}(zw^g/n)$.
\end{enumerate}
\end{lemma}
\begin{proof}[Proof of \cref{lem:fan-lemma}(1)--(2)]
First note that (1) implies (2). Indeed, the terms in $\psi(\mf C_{\mc{Q},e,j}^{(2)})$ exactly correspond to those in $\psi(\mf C_{\mc{Q},e,j}^{(1)})$, except that each has an extra factor of $\pi_T=w/|U_{\on{lev}(T)}| = w/n$ for some $T\in\mc N_e(U_k)$. So, it suffices to prove (1).

\medskip
\textit{Case 1: $1\le |\mc Q|\le j-4$. } Let $\mf C_T'$ be the submultiset of $\mf C_{\mc{Q},e,j}^{(1)}$ corresponding to a particular choice of $T\in \mc N_e(U_k)$, so that
\begin{align*}
\psi(\mf{C}_{\mc{Q},e,j}^{(1)})&=\sum_{T\in\mc N_e(U_k)}\psi(\mf C_T')\le \sum_{T\in\mc N_e(U_k)}\psi(\mf{A}_{\mc{Q}\cup\{T\},j,j-3-|\mc{Q}|}^{(1)})\\
&= n\cdot O_{g,k}(zw^{g-1}n^{j-v^j(\mc{Q}\cup\{T\})-(j-3-|\mc{Q}|)})= n\cdot O_{g,k}(zw^{g-1}n^{-1}) = O_{g,k}(zw^{g-1}),
\end{align*}
by \cref{lem:general-moments}(1).

\medskip
\textit{Case 2: $|\mc Q|=j-3\ge 2$. } In this case, $\psi(\mf C_{\mc{Q},e,j}^{(1)})$ is simply the number of choices of $\mc S\in \mf F_j$ and $T\in \mc N_e(U_k)$ such that $\mc S=\mc Q\cup\{T\}$. This number is at most $zn^{j-v^j(\mc{Q})}\le z$ by \cref{WS1} (here we are using that $e\subseteq U_k$).

\medskip
\textit{Case 3: $j=4$ and $|\mc Q|=1$. } In this case, let $T^\ast$ be the single triangle in $\mc Q$ and note that $\psi(\mf C_{\mc{Q},e,j}^{(1)})$ is simply the number of $\mc S\in \mf F_j$ which consist of $T^\ast$ and a second triangle $T\supseteq e$ which is within $U_k$. This number is at most $z$ by \cref{WS3} applied to $T^\ast$ and $e$.
\end{proof}
\begin{proof}[Proof of \cref{lem:fan-lemma}(3)--(4)]
Similarly to the last proof, we see that (3) implies (4), so it suffices to prove (3). First, let $\mf C'$ be the submultiset of $\mf{C}_{e,e',j}^{(3)}$ arising from data $(\mc S,\mc F,T)$ for which $e,e'$ are both contained in the same triangle of $\mc S$ (then $e$ and $e'$ uniquely determine this triangle, which must be $T$). We have
\[\psi(\mf C')\le \psi(\mf{A}_{\{T\},j,j-3}^{(1)}) =O_{g,k}(zw^{g-1}n^{j-v^j(\{T\})-(j-3)}) = O_{g,k}(zw^{g-1})\]
by \cref{lem:general-moments}(1).

Let $\mf C''=\mf{C}_{e,e',j}^{(3)}\setminus \mf C'$ be the complementary submultiset of $\mf{C}_{e,e',j}^{(3)}$ arising from data $(\mc S,\mc F,T)$ for which $e,e'$ belong to different triangles of $\mc S$. It now suffices to prove that $\psi(\mf C'')=O_{g,k}(zw^{g-1})$.

\medskip
\textit{Case 1: $j\ge 5$. } Let $\mf C_{T,T'}''$ be the submultiset of $\mf C''$ corresponding to a particular choice of $T\in \mc N_e(U_k)$, and for which $e'$ is contained in a particular triangle $T'\in \mc F$. Note that $\psi(\mf C_{T,T'}'')\le \pi_{T'}\psi(\mf{A}_{\{T,T'\},j,j-4}^{(1)})$, and recall that $\pi_{T'}=w/|U_{\on{lev}(T')}|$, so using \cref{lem:general-moments}(1), we have
\begin{align*}
\psi(\mf{C}'')=\sum_{T\in\mc N_e(U_k)}\sum_{T'\supseteq e'}\psi(\mf C_{T,T'}'')&\le \sum_{T\in\mc N_e(U_k)}\sum_{i=0}^k\sum_{\substack{T'\supseteq e'\\\on{lev}(T')=i}}\frac{w}{|U_i|}\psi(\mf{A}_{\{T,T'\},j,j-4}^{(1)})\\
&= n\cdot\sum_{i=0}^k|U_i|\cdot\frac{w}{|U_i|}\cdot O_{g,k}(zn^{j-v^j(\{T,T'\})-(j-4)})= O_{g,k}(zw).
\end{align*}

\medskip
\textit{Case 2: $j=4$. } Let $\mf C_{T'}''$ be the submultiset of $\mf C''$  for which $e'$ is contained in a particular triangle $T'\in \mc F$. Then $\psi(\mf C_{T'}'')$ is $\pi_{T'}$ times the number of choices of $\mc S\in \mf F_j$ which consist of $T'$ and a second triangle $T\supseteq e$ which is within $U_k$. This number is at most $z$ by \cref{WS3} applied to $T'$ and $e$, so
\[
\psi(\mf{C}'')\le\sum_{T'\supseteq e'}\psi(\mf C_{T'}'')\le \sum_{i=0}^k\sum_{\substack{T'\supseteq e'\\\on{lev}(T')=i}}\frac{w}{|U_i|}\cdot z\le\sum_{i=0}^k|U_i|\cdot\frac{w}{|U_i|}\cdot z = O_k(zw),
\]
as desired.
\end{proof}

We now use \cref{lem:fan-lemma} to prove some technical maximum weight estimates. As in \cref{lem:nibble-config}, we will need to consider weight systems whose ground sets consist of mixtures of triangles and edges.

In the following lemma $\mc{W}_1$ and $\mc{W}_2$ will be two copies of the same set of triangles but their weights will be different. The reason for this is that in the proof of \cref{thm:main} we will need to distinguish between triangles chosen during the initial sparsification step, and triangles chosen during the iterative ``vortex'' procedure. This lemma will be used in \cref{sub:iter-left}.

\begin{lemma}\label{lem:moment-left}
Fix positive real numbers $y,z\ge 1$ and $r,p\le 1$. Also fix a positive integer $4\le j\le g$, a descending sequence of subsets $V(K_N)=U_{0}\supseteq\dots\supseteq U_{k}\supseteq U_{k+1}$, and let $n = |U_k|$.
\begin{itemize}
\item Let $U_k^\ast = U_k\setminus U_{k+1}$.
\item Let $\mf F_j$ be a collection of sets of $j-2$ triangles in $K_N$ which is $(y,z)$-well-spread with respect to the sequence $U_{0}\supseteq\dots\supseteq U_{k}$ (not including $U_{k+1}$).
\item Let $\mc W'$ be the set of triangles in $K_N$, and let $\mc W_1, \mc W_2$ be two disjoint ``marked'' copies of $\mc W'$. Specifically, for $T\in \mc W'$, we write $(T,1)$ for its copy in $\mc W_1$, and $(T,2)$ for its copy in $\mc W_2$.
\item Let $\mc W_3$ be the set of edges in $K_N$ between $U_k^\ast$ and $U_{k+1}$.
\item Let $\vec \pi$ be the weight system for $\mc{W} = \mc{W}_1\cup\mc{W}_2\cup\mc{W}_3$ defined by $\pi_{T,1} = 1/N$ and $\pi_{T,2} = p/|U_{\on{lev}(T)}|$ for each triangle $T$, and $\pi_e=rp$ for each $e\in \mc W_3$.
\end{itemize}
(In the last bullet point, and for the rest of the lemma statement and proof, we define $\on{lev}(\cdot)$ with respect to $U_0\supseteq\cdots\supseteq U_k$, without $U_{k+1}$.)

Let $e$ be an edge in $K_N[U_k^\ast]$, and let $\mf{L}_{e,j}$ be the multiset (of sets of edges and ``marked'' triangles) constructed as follows. Consider each $\mc{S}\in\mf{F}_j$ and $T\in \mc{S}\cap\mc N_e(U_{k+1})$, and consider each ``marking'' function $\Phi:\mc{S}\setminus\{T\}\to\{1,2\}$ that gives at least one triangle the mark ``2''. For each such $(\mc S,T,\Phi)$ we add to $\mf{L}_{e,j}$ a copy of \[(E(T)\setminus\{e\})\;\cup\;\{(T',\Phi(T')): T'\in\mc{S}\setminus\{T\}\}.\]
If $z/|U_{k+1}|\le yr^2p^3$ then $\kappa(\mf{L}_{e,j}) = O_{g,k}(yr^2p^3|U_{k+1}|)$.
\end{lemma}
\begin{proof}
The proof will follow a similar strategy to \cref{lem:nibble-config}. Let $\pvec{\pi}$ be the weight system on the ground set $\mc W'\cup \mc W_3$ where we set $\pi_e' = 1$ for all $e\in \mc W_3$ (i.e., we ignore the influence of the edges in $\mc W_3$), and set $\pi_T'=1/|U_{\on{lev}(T)}|$. So, for each triangle $T\in \mc W'$ we have $\pi_{T,1}\le\pi_T'$ and $\pi_{T,2}=p\cdot \pi_T'$. For convenience we write $\psi'$ instead of $\psi^{(\pvec{\pi})}$ (so $\psi$ with no superscripts will be reserved for weights with respect to $\vec \pi$).

First we study $\psi(\mf{L}_{e,j},\emptyset)$. Let $\mf L_T'$ be the submultiset of $\mf{L}_{e,j}$ corresponding to a particular choice of $T$, and note that $\psi(\mf{L}_T')\le 2^{j-3}p(rp)^2\psi'(\mf{A}^{(1)}_{\{T\},j,j-3})$. Indeed, first note that for any $\mc S,T$ there are $2^{j-3}-1\le 2^{j-3}$ choices of $\Phi$. For each such $(\mc S,T,\Phi)$ contributing to $\mf{L}_{e,j}$ we have $\pi_{T',\Phi(T')}\le \pi_{T'}'$ for all $T'\in \mc S\setminus \{T\}$, and $\pi_{T',\Phi(T')}=\pi_{T',2}=p\cdot \pi_{T'}'$ for at least one $T'\in \mc S\setminus \{T\}$. The factor of $(rp)^2$ is due to the weights $\pi_{e'}=rp$ for the two edges $e'\in E(T)\setminus \{e\}$. We then have 
\[\psi(\mf{L}_{e,j},\emptyset)\le \sum_{T\in \mc N_e(U_{k+1})}\psi(\mf L_T')\le |U_{k+1}|2^{j-3}p(rp)^2\cdot O_{g,k}(yn^{j-3-(j-3)})=O_{g,k}(|U_{k+1}|r^2p^3y),\]
by \cref{lem:general-moments}(1).

So, it now suffices to study weights of the form $\psi(\mf{L}_{e,j},\mc{H}^2\cup\mc{H}^3)$, where $\mc{H}^2\subseteq\mc{W}_3\setminus \{e\}$ is a set of edges and $\mc{H}^3\subseteq\mc{W}_1\cup\mc{W}_2$ is a set of ``marked'' triangles, and $\mc H^2\cup\mc H^3\ne\emptyset$. Actually, in this case where $\mc H^2\cup \mc H^3\ne\emptyset$, we will be able to get away with a crude bound that completely ignores the edge-weights $\pi_e$ and ignores the distinction between $\mc W_1$ and $\mc W_2$. To be precise, let $\mc H^{3\prime}\subseteq \mc W'$ be the set of triangles obtained by removing the marks from the triangles in $\mc H^3$, and let $\mf{L}_{\mc H^2}''$ be the multiset of sets of triangles obtained by including a copy of $\mc S\setminus \{T\}$ for each $\mc S\in \mf F_j$ and $T\in \mc S\cap \mc N_e(U_{k+1})$ such that $\mc H^2\subseteq E(\mc S)$. Note that $\psi(\mf L_{e,j},\mc H^2\cup \mc H^3)\le 2^{j-3}\psi'(\mf L_{\mc H^2}'',\mc H^{3\prime})$. So, recalling our assumption $z/|U_{k+1}|\le yr^2p^3$, it suffices to prove that $\psi'(\mf{L}_{\mc H^2}'',\mc H^{3\prime})=O_{g,k}(z)$.

\medskip
\textit{Case 1: $\mc H^3\ne \emptyset$. } In this case, we observe that $\mc N_e(U_{k+1})\subseteq \mc N_e(U_{k})$ and apply \cref{lem:fan-lemma}(1) to see that
\[\psi'(\mf{L}_{\mc H^2}'',\mc H^{3\prime})\le \psi'(\mf{C}_{\mc H^{3\prime},e,j}^{(1)}) = O_{g,k}(z).\]

\medskip
\textit{Case 2: $\mc H^2\ne \emptyset$ and $\mc H^3=\emptyset$. } In this case we fix $e'\in \mc H^2$ and apply \cref{lem:fan-lemma}(3) to see that
\[\psi'(\mf{L}_{\mc H^2}'',\mc H^{3\prime})\le \psi'(\mf{C}_{e',e,j}^{(3)}) = O_{g,k}(z).\qedhere\]
\end{proof}

In the following lemma we have \emph{three} distinguished sets of triangles: one corresponds to the initial sparsification step, one corresponds to the triangles chosen in previous steps of the ``vortex'' procedure, and the other corresponds to triangles that survive an additional sparsification in preparation for the ``covering crossing edges'' stage of the proof (recall the outline in \cref{sec:overview}). This lemma will be used in \cref{sub:iter-link}.

\begin{lemma}\label{lem:moment-link}
Fix positive real numbers $y,z,r,p,\gamma$ with $y,z\ge 1$ and $r,p\le 1$. Also fix a positive integer $4\le j\le g$, a descending sequence of subsets $V(K_N)=U_{0}\supseteq\dots\supseteq U_{k}\supseteq U_{k+1}$, and let $n = |U_k|$.
\begin{itemize}
\item Let $U_k^\ast = U_k\setminus U_{k+1}$.
\item Let $\mf F_j$ be a collection of sets of $j-2$ triangles in $K_N$ which is $(y,z)$-well-spread with respect to the sequence $U_{0}\supseteq\dots\supseteq U_{k}$ (not including $U_{k+1}$).
\item Let $\mc W'$ be the set of triangles in $K_N$. For $i\in\{1,2\}$ let $\mc W_i=\{(T,i):T\in \mc W'\}$, and let
\[\mc W_3=\{(T,3):T\in \mc W',\;T\emph{ has one vertex in }U_k^\ast\emph{ and two vertices in }U_{k+1}\}.\]
\item Let $\mc W_4$ be the set of edges in $K_N[U_k]$ which are not between two vertices in $U_k^\ast$.
\item Let $\vec \pi$ be the weight system for $\mc{W} = \mc{W}_1\cup\mc{W}_2\cup\mc{W}_3\cup \mc W_4$ defined as follows. Let $\pi_{T,1} = 1/N$ and $\pi_{T,2} = p/|U_{\on{lev}(T)}|$ for each $T\in \mc W'$, and $\pi_{T,3}=n^\gamma/(rp^2|U_{k+1}|)$ for each $(T,3)\in \mc W_3$. Then, let $\pi_{e'}=rp$ for each edge $e'$ between $U_k^\ast$ and $U_{k+1}$, and let $\pi_{e'}=p$ for each edge $e'$ between two vertices of $U_{k+1}$.
\end{itemize}
(In the last bullet point, and for the rest of the lemma statement and proof, we define $\on{lev}(\cdot)$ with respect to $U_0\supseteq\cdots\supseteq U_k$, without $U_{k+1}$.)

Let $e$ be an edge of $K_N$ between $U_k^\ast$ and $U_{k+1}$, and let $\mf{M}_{e,j}$ be the multiset (of sets of edges and ``marked'' triangles) constructed as follows. Consider each $\mc{S}\in\mf{F}_j$ and $T\in \mc{S}\cap\mc N_e(U_{k+1})$, such that at least one of the triangles in $\mc S\setminus\{T\}$ lies within $U_k$, and consider each marking function $\Phi:\mc{S}\setminus\{T\}\to\{1,2,3\}$ that gives at least one triangle a mark different from ``1'', and for which $\Phi(T')=3$ implies $(T',3)\in \mc W_3$. For each such $(\mc S,T,\Phi)$ we add to $\mf{M}_{e,j}$ a copy of \[(E(T)\setminus\{e\})\;\cup\bigcup_{\substack{T'\in\mc{S}\setminus T\\\Phi(T')=3}}E(T')\;\cup\;\{(T',\Phi(T')): T'\in\mc{S}\setminus\{T\}\}\;\cup\;\{(T,3)\}.\]
If $z(n^\gamma|U_k|/(rp^2|U_{k+1}|))^{g+1}/n\le y$ and $rn^\gamma\le p|U_{k+1}|/|U_k|$ and $pn^\gamma\le 1$, then $\kappa(\mf{M}_{e,j}) = O_{g,k}(y)$.
\end{lemma}
\begin{proof}
We proceed in a similar way to the proofs of \cref{lem:nibble-config,lem:moment-left}, but here there are more cases to consider. Let $\pvec{\pi}$ be the weight system on the ground set $\mc W'$ where we let $\pi_{T'}'=1/|U_{\on{lev}(T')}|$ for each $T'\in \mc W'$. We write $\psi'$ instead of $\psi^{(\pvec \pi)}$. Note that for each $T'\in \mc W'$ we have $\pi_{T',1}\le \pi'_{T'}$ and $\pi_{T',2}=p\cdot \pi'_{T'}$, and for each $(T',3)\in \mc W_3$ we have
\[\pi_{T',3}\prod_{e'\in E(T')}\pi_{e'} = \frac{n^\gamma}{rp^2|U_{k+1}|}\cdot p(rp)^2 = \frac{rpn^\gamma}{|U_{k+1}|}\le\frac{p^2}{|U_k|}\le \frac{p}{|U_{\on{lev}(T')}|} = p\cdot \pi_{T'}'.\]
Let $\mf M_T'$ be the subcollection of $\mf M$ corresponding to a particular choice of $T$, and note that
\[\psi(\mf M_T')\le 3^{j-3}\cdot\pi_{T,3}\cdot p(rp)\cdot p\cdot\psi'(\mf A^{(1)}_{\{T\},j,j-3})=\frac{3^{j-3}pn^\gamma}{|U_{k+1}|}\psi'(\mf A^{(1)}_{\{T\},j,j-3}).\]
Indeed, here we are using that for each $\mc S,T$ there are at most $3^{j-3}$ choices for $\Phi$, we are using that $\prod_{e'\in E(T)\setminus \{e\}}\pi_{e'}=p(rp)$, and we are using that at least one of the triangles in $\mc S\setminus\{T\}$ has mark different from ``1'' (from which we gain a factor of $p$). We deduce using \cref{lem:general-moments}(1) that
\begin{align*}
\psi(\mf{M}_{e,j},\emptyset)=\sum_{T\in \mc N_e(U_{k+1})}\psi(\mf M_T')&\le\frac{3^{j-3}pn^\gamma}{|U_{k+1}|}\sum_{T\in\mc N_e(U_{k+1})}\psi'(\mf{A}_{\{T\},j,j-3}^{(1)})\\
&= \frac{pn^\gamma}{|U_{k+1}|}\cdot|U_{k+1}|\cdot O_{g,k}(yn^{j-(1+2)-(j-3)}) = O_{g,k}(ypn^\gamma)=O_{g,k}(y).
\end{align*}

At this point, similarly to the proofs of \cref{lem:moment-left,lem:nibble-config}, we will break into cases and use much cruder estimates to handle $\psi(\mf M_{e,j},\mc H)$ for $\mc H\ne\emptyset$. However, in the upcoming \textit{Case 2}, we will need to assume a certain technical condition, so first we explicitly consider a special situation which we have to handle quite precisely.

In particular, suppose that $\mc H^3\subseteq\mc W_1\cup\mc W_2\cup\mc W_3$ consists of a single (marked) triangle and that $\{e\}\cup\mc H^2\subseteq E(\mc H^3)$. We call this the ``exceptional case''. In this case, we need only consider data $(\mc S,T,\Phi)$ for which $\mc H^3 = \{T\}$ and we see
\[\psi(\mf M_{e,j},\mc H^2\cup \mc H^3)\le 3^{j-2}\psi'(\mf A^{(1)}_{\{T\},j,j-3})=O_{g,k}(y),\]
again using \cref{lem:general-moments}(1).

Now, fix $\mc H^2\subseteq \mc W_4\setminus \{e\}$ and $\mc H^3\subseteq \mc W_1\cup \mc W_2\cup \mc W_3$ such that $\mc H^2\cup \mc H^3\ne\emptyset$ and such that the exceptional case does not hold; it suffices to bound $\psi(\mf M_{e,j},\mc H^2\cup\mc H^3)$. We can afford to be rather crude with our estimates. Let $w=n^\gamma|U_k|/(rp^2|U_{k+1}|)\ge 1$, let $\mc H^{3\prime}\subseteq \mc W$ be the set of unmarked triangles underlying $\mc H^3$, and let $\mf M_{\mc H^2}''$ be the collection of all $\mc S\in \mf F_j$ so that there is some $T\in \mc S\cap \mc N_e(U_{k+1})$ (if such a $T$ exists, it is unique) for which at least one of the triangles in $\mc S\setminus \{T\}$ lies within $U_k$ and for which $\mc H^2\subseteq E(\mc S)$. Note that for $(T',3)\in\mc{W}_3$ we have $\pi_{T',3} = w\cdot\pi_{T'}'$, which implies the crude bound
\begin{equation}
\psi(\mf M_{e,j},\mc H^2\cup \mc H^3)\le 3^{j-3}\cdot w^{j-2-|\mc H^3|}\cdot \psi'(\mf M_{\mc H^2}'',\mc H^{3\prime}).\label{eq:crude-weight-link-bound}
\end{equation}

We now consider five different cases for $\mc H$ and $j$. In most cases we will prove $\psi'(\mf M_{\mc H^2}'',\mc H^{3\prime}) = O_{g,k}(z/n)$, from which the desired bound $\psi(\mf M_{e,j},\mc H^2\cup \mc H^3)=O_{g,k}(y)$ follows, using \cref{eq:crude-weight-link-bound} and our assumption $zw^{g+1}/n\le y$. (\textit{Case 5} is slightly different.)

\medskip
\textit{Case 1: $\mc H^3=\emptyset$ and $\mc H^2\ne \emptyset$. } Fix any $e'\in \mc H^2$. By \cref{lem:fan-lemma}(4),
\[\psi'(\mf{M}_{\mc H^2}'',\mc H^{3\prime})\le \psi'(\mf{C}_{e,e',j}^{(4)}) = O_{g,k}(z/n).\]

\medskip
\textit{Case 2: $|\mc H^{3}|=1$ and $\mc H^{3\prime}\cap \mc N_e(U_k)\ne\emptyset$. }  In this case, we need only consider data $(\mc S,T)$ for which $T$ is the single triangle in $\mc H^{3\prime}\cap \mc N_e(U_k)$. We are assuming that the exceptional case does not hold, meaning $\{e\}\cup \mc H^2\not\subseteq E(\mc H^3)$, so $\mc H^2\not\subseteq E(T)$, i.e., $\mc H^2$ contains an edge $e'$ not in $T$. Now, we proceed similarly to the proof of \cref{lem:fan-lemma}(1). We consider two subcases. First, if $j=4$ then each $\mc S\in \mf M_{\mc H^2}''$ consists of $T$ and a second triangle $T'\supseteq e'$ satisfying $\on{lev}(T') = k$ (recall that at least one of the triangles in $\mc S\setminus \{T\}$ must lie within $U_k$). So, by \cref{WS3} applied to $T$ and $e'$ we have
\[\psi'(\mf M_{\mc H^2}'',\mc H^{3\prime})\le\frac{1}{|U_k|}\cdot z=O(z/n).\]
On the other hand, if $j>4$ then by \cref{lem:general-moments}(1) we have
\begin{align*}
\psi'(\mf{M}_{\mc H^2}'',\mc{H}^{3\prime})&\le\sum_{i=0}^k\sum_{\substack{T'\supseteq e':\\\on{lev}(T')=i}}\frac{1}{|U_i|}\cdot \psi'(\mf{A}_{\{T,T'\},j,j-4}^{(1)})= O_{g,k}(zn^{j-v^j(\{T,T'\})-(j-4)}) = O_{g,k}(z/n).
\end{align*}

\medskip
\textit{Case 3: $\mc H^3\ne\emptyset$ but $\mc H^{3\prime}\cap \mc N_e(U_k)=\emptyset$. } By \cref{lem:fan-lemma}(2) we have
\[\psi'(\mf{M}_{\mc H^2}'',\mc{H}^{3\prime})\le \psi'(\mf{C}_{\mc{H}^{3\prime},e,j}^{(2)}) = O_{g,k}(z/n).\]

\medskip
\textit{Case 4: $\mc H^{3\prime}\cap \mc N_e(U_k)\ne \emptyset$ and $2\le |\mc H^3|\le j-3$.  } Any $\mc S\in \mf F_j$ contains at most one triangle in $\mc N_e(U_k)$, so we may assume $\mc{H}^{3\prime}\cap\mc N_e(U_k)$ contains a single triangle (and we need only consider data $(\mc S,T)$ for which that single triangle is $T$). Then by \cref{lem:general-moments}(1) we have
\[
\psi'(\mf{M}_{\mc H^2}'',\mc{H}^{3\prime})\le \psi'(\mf{A}_{\mc{H}^{3\prime},j,j-2-|\mc{H}^3|}^{(1)}) = O_{g,k}(zn^{j-(|\mc{H}^3|+3)-(j-2-|\mc{H}^3|)}) = O_{g,k}(z/n).\]
\medskip
\textit{Case 5: $|\mc H^3|=j-2$. } Finally, this case is trivial: we have $\psi'(\mf M_{\mc H^2}'',\mc H^{3\prime})\in\{0,1\}$ since the only possible element is $\mc{S} = \mc{H}^3$. Then \cref{eq:crude-weight-link-bound} and $|\mc{H}^3| = j-2$ demonstrate
\[\psi(\mf{M}_{e,j},\mc{H}^2\cup\mc{H}^3)\le 3^{j-3} = O_{g,k}(y).\qedhere\]
\end{proof}

\newcommand{\letter}{i}

Finally, the following lemma will be used in \cref{sub:iter-final}, applied at the end of stage $k$ to all $\letter\ge k+1$.

\begin{lemma}\label{lem:moment-quasi}
Fix positive real numbers $z,w\ge 1$ and $p\le 1$, integers $h\ge 1$ and $4\le j\le g$, and a descending sequence of subsets $V(K_N)=U_{0}\supseteq\dots\supseteq U_{\letter}\supseteq U_{\letter+1}$.
\begin{itemize}
\item Let $\mf F_j$ be a collection of sets of $j-2$ triangles in $K_N$ which is $(y,z)$-well-spread with respect to the sequence $U_{0}\supseteq\dots\supseteq U_{\letter}$ (not including $U_{\letter+1}$).
\item Let $\mc W'$ be the set of triangles in $K_N$, and for $a\in\{1,2\}$ let $\mc W_a=\{(T,a):T\in \mc W'\}$. Let $\mc W_3$ be the set of edges in $K_N[U_\letter]$.
\item Let $\vec \pi$ be the weight system for $\mc{W} = \mc{W}_1\cup\mc{W}_2\cup\mc{W}_3$ defined by $\pi_{T,1} = 1/N$ and $\pi_{T,2} = p/|U_{\on{lev}(T)}|$ for each triangle $T$, and $\pi_e=p$ for each edge $e$.
\end{itemize}
(In the last bullet point, and for the rest of the lemma statement and proof, we define $\on{lev}(\cdot)$ with respect to $U_0\supseteq\cdots\supseteq U_\letter$, without $U_{\letter+1}$.)

Let $Q$ be a set of edges in $K_N[U_\letter]$ spanning $|V(Q)|\le h$ vertices, and fix $\letter^\ast\in \{\letter,\letter+1\}$. Let $\mf R_{Q,\letter^\ast,j}$ be the multiset (of sets of edges and marked triangles) constructed as follows. Consider each $e\in Q$, each $\mc S\in \mf F_j$, and $T\in \mc S\cap \mc N_e(U_{\letter^\ast})$. Also consider each marking function $\Phi:\mc S\setminus \{T\}\to \{1,2\}$ that gives at least one triangle the mark ``2''. For each such $(e,\mc S,T,\Phi)$, let $u$ be the single vertex in $T\setminus e$ and add to $\mf R_{Q,\letter^\ast,j}$ a copy of 
\[\{uv\in \mc W_3:v\in V(Q)\}\;\cup \;\{(T',\Phi(T')):T'\in \mc S\setminus \{T\}\}.\]
If $zp^{-h-1}/|U_{\letter^\ast}|\le y$ then $\kappa(\mf{R}_{Q,\letter^\ast,j}) = O_{g,\letter,h}(yp^{|V(Q)|+1}|U_{\letter^\ast}|)$.
\end{lemma}
\begin{proof}
Let $\mf R_e'$ be the submultiset of $\mf R_{Q,i^\ast,j}$ corresponding to a particular choice of $e$. There are fewer than $h^2=O_h(1)$ choices of $e$, so it suffices to show that $\kappa(\mf R_e')=O_{g,\letter}(yp^{|V(Q)|+1}|U_{\letter^\ast}|)$. We proceed in a similar way to the proofs of \cref{lem:nibble-config,lem:moment-left,lem:moment-link}. Let $\pvec \pi$ be the weight system on the ground set $\mc W'$ where we let $\pi_{T'}'=1/|U_{\on{lev}(T')}|$ for each $T'\in \mc W'$, and write $\psi'$ instead of $\psi^{(\pvec \pi)}$.

First, by \cref{lem:general-moments}(1), we have
\[\psi(\mf{R}_e',\emptyset)\le 2^{j-3}\cdot p\cdot p^{|V(Q)|}\cdot\sum_{T\in \mc N_e(U_{\letter^\ast})} \psi'(\mf{A}_{\{T\},j,j-3}) = O_{g,\letter}(|U_{\letter^\ast}|p^{|V(Q)|+1}y).\]
(The factor of $p$ comes from the fact that we are only considering $(\mc S,T,\Phi)$ for which at least one triangle in $\mc S\setminus\{T\}$ has the mark ``2''.)

Now, fix $\mc H^2\subseteq \mc W_3$ and $\mc H^3\subseteq \mc W_1\cup \mc W_2$ such that $\mc H^2\cup \mc H^3\ne \emptyset$. Let $\mc H^{3\prime}$ be the set of unmarked triangles underlying $\mc H^3$, and let $\mf R_{e,\mc H^2}''$ be obtained by including a copy of $\mc S\setminus\{T\}$ for every $\mc S$ and $T\in \mc S\cap \mc N_e(U_{\letter^\ast})$ such that all edges in $\mc H^2$ contain the single vertex in $T\setminus \{e\}$. We have $\psi(\mf R_e',\mc H^2\cup \mc H^3)\le 2^{j-3}\psi'(\mf R_{e,\mc H^2}'',\mc H^{3\prime})$. Recalling our assumption $zp^{-h-1}/|U_{\letter^\ast}|\le y$, it suffices to prove that $\psi'(\mf R_{e,\mc H^2}'', \mc H^{3\prime})=O_{g,\letter}(z)$.

\medskip
\textit{Case 1: $\mc H^3\ne\emptyset$. } In this case, \cref{lem:fan-lemma}(1) implies
\[\psi'(\mf R_{e,\mc H^2}'',\mc H^{3\prime})\le \psi'(\mf{C}_{\mc{H}^{3\prime},e,j}^{(1)}) = O_{g,\letter}(z).\]

\medskip
\textit{Case 2: $\mc H^3=\emptyset$ and $\mc H^2\ne\emptyset$. } Fix an edge $e''\in \mc H^2$. Each set in $\mf R_{e,\mc H^2}''$ is specified by a pair $(\mc S,T)$, where $T$ is a triangle containing $e$ and another vertex $u\in e''$. That is to say, there are at most two possibilities for $T$. For each such $T$, let $\mf R_{T}'''$ be the submultiset of $\mf R_{e,\mc H^2}''$ corresponding to that choice of $T$, let $e'$ be one of the two edges of $E(T)\setminus e$, and observe that
\[\psi'(\mf{R}_{T}''',\mc{H}^{3\prime})\le \psi'(\mf{C}_{e,e',j}^{(3)}) = O_{g,\letter}(z)\]
by \cref{lem:fan-lemma}(3).
\end{proof}

\section{A generalized high-girth triple process}\label{sec:nibble}
In this section we study a generalization of the high-girth triple process studied in \cite{GKLO20,BW19}. We show that in quite a general setting it is possible to find a set of edge-disjoint triangles avoiding a set of ``forbidden configurations''. Where possible, we make some effort to use similar notation as in \cite{GKLO20}.

As sketched in \cref{sec:overview}, it is important that at every stage of the main iteration, we ``remember the randomness'' of prior stages. So, for several of the lemmas in this paper, one of the inputs will be a random set of triangles coming from the results of the iteration so far. The following definition captures the properties we need from this random set of triangles: they are consistent with having arisen from the iteration so far, and in particular the probability that any set of triangles has been chosen can be approximated by a product of weights. This will be crucial in order to apply the estimates in \cref{sec:weight}. Later, in the proof of \cref{thm:main} we will need a more refined version of this definition (\cref{def:consistent}).

\newcommand{\wderror}{b}

\begin{definition}[Well-distributedness]\label{def:well-distributed-1}
Fix a descending sequence of subsets $V(K_N)=U_0\supseteq\dots\supseteq U_k$, and an ``error factor'' $w\ge 1$. Say a random set of triangles $\mc{D}$ in $K_{N}$ is \emph{$(w,\wderror)$-well-distributed} (with respect to our descending sequence of subsets) if for any $s\in \mb N$ and distinct triangles $D_{1},\dots,D_{s}$ we have \[\Pr[D_{1},\dots,D_{s}\in\mc{D}]\le w^s\bigg(\prod_{i=1}^s\frac{1}{|U_{\on{lev}(D_i)}|} + \wderror\bigg).\]
\end{definition}
We will always take $\wderror=n^{-\omega(1)}$ to be super-polynomially small in $n = |U_k|$, so the condition in \cref{def:well-distributed-1} will be meaningful for all bounded $s=O(1)$.

We are nearly ready to state the main theorem in this section. It is convenient to separately define the assumptions of this theorem, as we will want to refer to them multiple times.

\begin{definition}[Goodness]\label{def:good}
Consider a descending sequence of subsets $V(K_N)=U_0\supseteq \dots\supseteq U_k$, and consider collections $\mf F^{\sup}_4,\dots,\mf F^{\sup}_{g}$ of triangles in $K_N$, where each $\mc S\in \mf F_j^{\sup}$ consists of $j-2$ edge-disjoint triangles. We also consider some initial data (which are themselves allowed to be random):
\begin{itemize}
    \item let $\mc D$ be a random ``already chosen'' set of edge-disjoint triangles in $K_N$;
    \item let $G\subseteq K_N[U_k]\cong K_n$ be a random ``remainder'' graph (the graph which we wish to decompose into triangles), with edge set $E$ disjoint from $E(\mc{D})$;
    \item let $\mc A$ be a (random) subset of the triangles in $G$ (the ``available'' triangles which we may use in our triangle-decomposition);
    \item let $\mf J_4,\dots,\mf J_{g}$ be random collections of sets of triangles in $\mc A$ (the ``forbidden configurations'' we would like to avoid), such that for each $\mc S\in \bigcup_{j=4}^g \mf J_j$ there is $\mc S'\in \bigcup_{j=4}^g \mf F_j^{\sup}$ with $\mc{S}\subseteq\mc{S}'$ and $\mc{S}'\setminus\mc{S}\subseteq \mc D$. (That is to say, each of our forbidden configurations is a set of triangles which together with some triangles in $\mc D$ forms one of the configurations in $\bigcup_{j=4}^g\mf F_j^{\sup}$.)
\end{itemize}
We say that all the above data (namely, the sets $U_i$, the collections of triangles $\mf F_j^{\sup}$, and the random objects $\mc D,G,\mc A,\mf J_4,\dots,\mf J_{g}$) are \emph{$(C,\beta,\wderror)$-good} if:
\begin{itemize}
    \item each $\mf F^{\sup}_j$ is $n^\beta$-well-spread (\cref{def:well-spread});
    \item $\mc D$ is $(n^\beta,\wderror)$-well-distributed (\cref{def:well-distributed-1});
    \item with probability $1$:
    \begin{itemize}
        \item every $e\in E$ is in $(1\pm n^{-1/C})3|\mc A|/|E|$ triangles of $\mc A$;
        \item every $T\in \mc A$ is contained in $(j-2)|\mf J_j|/|\mc A|\pm n^{-1/C}|\mc A|^{j-3}/|E|^{j-3}$ of the configurations in $\mf J_j$;
        \item $|\mc A|\ge n^{1-\beta}|E|$;
        \item $|E|\ge n^{2-\beta}$;
        \item $|\mf J_j|\le C|\mc A|^{j-2}/|E|^{j-3}$;
        \item There is no $\mc S\in \bigcup_{j=4}^g \mf J_j$ which is a subset of another $\mc S'\in \bigcup_{j=4}^g \mf J_j$.
    \end{itemize}
\end{itemize}
\end{definition}
We remark that the inequality $|E|\ge n^{2-\beta}$ can actually be derived from the inequality $|\mc A|\ge n^{1-\beta}|E|$, so the former inequality is not strictly speaking necessary (though we will not actually need this deduction anywhere in the paper).

The role of the collections $\mf F_j^{\sup}$ is that, while the collections of forbidden configurations $\mf J_j$ are allowed to depend on the random data, they are always ``induced'' from the collections $\mf F_j^{\sup}$, which are well-spread and do not depend on the random data. In our proof of \cref{thm:main} we will take $\mf{F}_j^{\sup}$ to contain all Erd\H{o}s configurations, all forbidden configurations induced by absorber triangles (see \cref{lem:absorber-well-spread}), and in addition some random sets of triangles that will be used for ``regularization''.

Now, we present the main theorem in this section. In addition to guaranteeing the existence of an almost-triangle-decomposition, we require the existence of a probability distribution over such triangle-decompositions, satisfying certain properties; we will ``remember the randomness'' of this distribution for future stages of the main iterations.
\begin{theorem}\label{thm:random-high-girth-nibble}
Fix a constant $g\in \mb N$. For $C>0$ there are $\beta = \beta_{\ref{thm:random-high-girth-nibble}}(g,C)>0$ and $\alpha = \alpha_{\ref{thm:random-high-girth-nibble}}(g,C)>0$ such that the following holds. Consider $(C,\beta,\wderror)$-good data, as defined in \cref{def:good}. Define $h(\wderror)$ by setting $h(\wderror)=\sqrt{\wderror^\alpha+\exp(-n^\alpha)}$ if $\wderror\le N^{-1/\beta}$, and $h(\wderror)=1$ otherwise. There is a random set of edge-disjoint triangles $\mc M\subseteq\mc A$ (depending on our data) such that:
\begin{itemize}
    \item no $\mc S\in \bigcup_j \mf J_j$ is fully included in $\mc M$, and
    \item with probability at least $1-h(\wderror)$ over the randomness of the initial data $\mc D,G,\mc A,\mf J_4,\dots,\mf J_g$: for any $s_1,s_2\ge 0$
and any triangles $D_1,\ldots,D_{s_1}\in \mc A$ and edges $e_1,\ldots,e_{s_2}\in E$ we have
\begin{align*}
&\Pr[D_1,\ldots,D_{s_1}\in\mc M\emph{ and }e_1,\dots,e_{s_2}\notin E(\mc M)\mid\mc D,G,\mc A,\mf J_4,\dots,\mf J_g]\le\big(O_{g,C}(|E|/|\mc{A}|)\big)^{s_1}\big(O(n^{-\beta})\big)^{s_2}\!+\!2h(\wderror).
\end{align*}
\end{itemize}
\end{theorem}
The particular form of the error term $h(\wderror)$ is not very important; all that is important for the proof of \cref{thm:main} is that if $\wderror=n^{-\omega_{g,C}(1)}$ then $h(\wderror)=n^{-\omega_{g,C}(1)}$ as well. We prove \cref{thm:random-high-girth-nibble} via analysis of a random process, as follows.

\begin{definition}[Generalized high-girth triple process]\label{def:high-girth-nibble}
Let $G$ be an $n$-vertex graph with edge set $E=E(0)$, let $\mc A=\mc A(0)$ be a set of triangles in $G$, and for $4\le j\le g$ let $\mf{J}_j$ be a collection of sets of $j-2$ edge-disjoint triangles which are forbidden from being used together\footnote{The reason for our perhaps unnatural-seeming indexing is that, in practice, the configurations in $\mf J_j$ will ``behave like'' Erd\H{o}s $j$-configurations, in the sense of \cref{def:well-spread}. This indexing convention also makes our analysis more consistent with the analysis in \cite{GKLO20}. However, one important difference between our analysis and the analysis in \cite{GKLO20} is that in \cite{GKLO20}, the set $\mf J_4$ has a very different role: it is the set of \emph{diamonds} (pairs of triangles that share an edge). Recall that our forbidden configurations always consist of edge-disjoint triangles, and we separately enforce an edge-disjointness condition in \cref{def:high-girth-nibble}.}. All these objects (``initial data'') may be random.

Now, consider the following random process (using fresh independent randomness).
Let $\mc C(0)=\emptyset$ and let $t=0$. While $\mc A(t)\ne \emptyset$:
\begin{enumerate}
    \item choose a uniformly random triangle $T^\ast(t)\in \mc A(t)$;
    \item let $\mc C(t+1)=\mc C(t)\cup\{T^\ast(t)\}$ (add $T^\ast(t)$ to the set of ``chosen triangles'');
    \item let $\mc A(t+1)$ be the set of triangles in $\mc A(t)$ which do not share an edge with $T^\ast(t)$ and would not complete any $\mc{E}\in\bigcup_{j=4}^g\mf J_j$ when combined with the triangles in $\mc C(t+1)$;
    \item increment $t$ to $t+1$.
\end{enumerate}
We also define some variables that track the evolution of the process:
\begin{itemize}
    \item Let $E(t)\subseteq E(0)$ be the set of edges not appearing in any triangle in $\mc C(t)$ (i.e., the set of ``uncovered'' edges). Note that $|E(t)|=|E(0)|-3t$.
    \item For some step $t$ and an edge $e\in E(t)$, let $\mc X_e(t)$ be the set of triangles in $\mc{A}(t)$ which contain the edge $e$.
    \item For a triangle $T\in\mc A(t)$ let $\mc T_T(t)\subseteq\mc{A}(t)$ be the set of triangles $T^\ast\ne T$ that ``threaten'' $T$, in the sense that choosing $T^\ast$ as the next triangle in the process would cause $T\notin \mc A(t+1)$ (i.e., $E(T)\cap E(T^\ast)\neq\emptyset$ or $\{T,T^\ast\}\cup\mc{C}(t)$ contains some element of $\bigcup_{j=4}^g\mf{J}_j$).
\end{itemize}
\end{definition}

We also define some functions that describe the typical behavior of the process. (They are defined in terms of the initial data, so if those data are random, then the following functions are random as well.)
\begin{definition}[Trajectories]\label{def:traj}
Given $|E(0)|,|\mc A(0)|,|\mf J_4|,\dots,|\mf J_{g}|$, we define
\[p(t)=\frac{|E(0)|-3t}{|E(0)|},\qquad\rho(t) = \sum_{j=4}^g\frac{(j-2)t^{j-3}|\mf{J}_j|}{|\mc{A}(0)|^{j-2}},\]
\[
f_{j,c}(t) = \binom{j-3}{c}\bigg(\frac{t}{|\mc A(0)|}\bigg)^c\big(p(t)^3e^{-\rho(t)}\big)^{j-3-c}\bigg(\frac{(j-2)|\mf{J}_j|}{|\mc A(0)|}\bigg)~\text{ for } 4\le j\le g,\; 0\le c\le j-4,\]
\[
f_{\mr{edge}}(t) = p(t)^2e^{-\rho(t)}\bigg(\frac{3|\mc A(0)|}{|E(0)|}\bigg),\qquad f_{\mr{threat}}(t) = 3f_{\mr{edge}}(t)+\sum_{j=4}^gf_{j,j-4}(t).
\]
\end{definition}
\begin{remark}
The reader can interpret $p(t)$ as the fraction of the graph $G$ which has not yet been covered at time $t$. As we will see in the next result, $f_{\mr{edge}}(t)$ and $f_{\mr{threat}}(t)$ describe the trajectories of the random processes $|\mc{X}_e(t)|$ and $|\mc{T}_T(t)|$, respectively. We can also interpret $f_{j,c}(t)$ as describing the trajectory of a quantity that evolves with the process: namely, it describes the approximate number of forbidden configurations $\mc S\in \mf{J}_j$, containing a fixed triangle $T$, such that $c$ of the other $j-3$ triangles in $\mc S$ have already been chosen, and the other $j-3-c$ triangles are still available. Finally, $\rho(t)$ is a slightly less intuitive parameter which captures the Poissonian rate at which triangles become unavailable due to forbidden configurations.

Roughly speaking, for a triangle $T\in \mc A(0)$, the probability that its three edges are in $E(t)$ is about $p(t)^3$, and given this, the probability $T$ is in $\mc A(t)$ is about $e^{-\rho(t)}$.
\end{remark}

We will prove the following theorem. 
\begin{theorem}\label{thm:high-girth-nibble}Fix a constant $g$. For $C > 0$ there are $\beta = \beta_{\ref{thm:high-girth-nibble}}(g,C)>0$ and $\alpha = \alpha_{\ref{thm:high-girth-nibble}}(g,C)>0$ such that the following holds. Fix $(C,\beta,\wderror)$-good data as in \cref{def:good} with $\wderror\le N^{-1/\beta}$. Then with probability at least $1-\wderror^{\alpha}-\exp(-n^{\alpha})$ (over all randomness, including that of the initial data): the process in \cref{def:high-girth-nibble} runs for $\tau_{\mr{cut}}:= \lceil (1-n^{-\beta})|E|/3 \rceil$ steps without terminating, and moreover for each $0\le t\le\tau_{\mr{cut}}$, in the notation of \cref{def:traj},
\begin{itemize}
    \item $\big||\mc X_{e}(t)|-f_{\mr{edge}}(t)\big|\le n^{-\beta}f_{\mr{edge}}(t)$ for each $e\in E(t)$;
    \item $\big||\mc T_T(t)|-f_{\mr{threat}}(t)\big|\le n^{-\beta}f_{\mr{threat}}(t)$ for each $T\in \mc A(t)$;
    \item $|\mc T_T(t)\cap \mc T_{T'}(t)|\le n^{1/2}$ for each pair of edge-disjoint $T,T'\in \mc A(t)$;
    \item $|\mc T_T(t)\cap \mc X_e(t)|\le n^{1/2}$ for each $T\in \mc A(t)$ and $e\in E(t)$ with $e\not\subseteq T$.
\end{itemize}
\end{theorem}

Before proving \cref{thm:high-girth-nibble}, we deduce \cref{thm:random-high-girth-nibble} from it.

\begin{proof}[Proof of \cref{thm:random-high-girth-nibble}]
We use the same $\alpha,\beta$ as in \cref{thm:high-girth-nibble}, except that we additionally enforce $\alpha < \beta/3$. That is, we take $\beta_{\ref{thm:random-high-girth-nibble}}(g,C)=\beta_{\ref{thm:high-girth-nibble}}(g,C)$ and $\alpha_{\ref{thm:random-high-girth-nibble}}(g,C)=\min\{\alpha_{\ref{thm:high-girth-nibble}}(g,C),\beta_{\ref{thm:high-girth-nibble}}(g,C)/4\}$.

Let $\mbf E$ be the event that the conclusion of \cref{thm:high-girth-nibble} holds, so $\Pr[\mbf E]\ge 1-h(\wderror)^2$. Fix an outcome of the initial data $\mc D,G,\mc A,\mf J_4,\dots,\mf J_g$ such that
$\Pr[\mbf E\,|\,\mc D,G,\mc A,\mf J_4,\dots,\mf J_g]\ge 1-h(\wderror)$ (the initial data satisfies this property with probability $1-h(\wderror)$ by Markov's inequality). For all probabilities in the rest of this proof, we implicitly condition on our fixed outcome of $\mc D,G,\mc A,\mf J_4,\dots,\mf J_g$.

For this proof, we need to consider the random set of triangles obtained by running the process in \cref{def:high-girth-nibble} for $\tau_{\mr{cut}}$ steps. However, due to an extension to \cref{thm:random-high-girth-nibble} we will need to prove later, it will be helpful to consider the more general case where we fix $\tau\le \tau_{\mr{cut}}$ and consider the result of running the process for $\tau$ steps. Let $\mc M$ be the corresponding random set of triangles. Fix $s_1,s_2$, distinct triangles $D_1,\dots,D_{s_1}$ and distinct edges $e_1,\dots,e_{s_2}$. Define \[P(D_1,\dots,D_{s_1};e_1,\dots,e_{s_2})=\Pr[D_1,\ldots,D_{s_1}\in\mc M\text{ and }e_1,\dots,e_{s_2}\notin E(\mc M)\mid\mc D,G,\mc A,\mf J_4,\dots,\mf J_g].\]
We wish to prove that $P(D_1,\dots,D_{s_1};e_1,\dots,e_{s_2})\le \big(O_{g,C}(|E|/|\mc{A}|)\big)^{s_1}\big(O(p(\tau))\big)^{s_2}+2h(\wderror)$ (note that $p(\tau_\mr{cut}) \approx n^{-\beta}$).

Before proceeding, we observe that it essentially suffices to consider the case where $s_1,s_2$ are quite small. Indeed, let $r_1=\min\{s_1,\lfloor n^{\beta/3}\rfloor\}$ and $r_2=\min\{s_2,\lfloor n^{\beta/3}\rfloor\}$. If we can prove that
\begin{equation}
    P(D_1,\dots,D_{r_1};e_1,\dots,e_{r_2})\le \big(O_{g,C}(|E|/|\mc{A}|)\big)^{r_1}\big(O(p(\tau))\big)^{r_2}+h(\wderror)\label{eq:truncate-nibble-prob-guarantee}\end{equation}
it would follow that
\begin{align*}
P(D_1,\dots,D_{s_1};e_1,\dots,e_{s_2})&\le P(D_1,\dots,D_{r_1};e_1,\dots,e_{r_2})\\
&= \on{min}(P(D_1,\dots,D_{r_1};e_1,\dots,e_{r_2}),P(D_1,\dots,D_{r_1}),P(e_1,\dots,e_{r_2}),1)\\
&\le \big(\min\{O_{g,C}(|E|/|\mc{A}|),1\}\big)^{r_1}\big(\min\{O(p(\tau)),1\}\big)^{r_2}+h(\wderror)\\
&\le\Big(\big(\min\{2\cdot O_{g,C}(|E|/|\mc{A}|),1\}\big)^{s_1}+\frac{h(\wderror)}{3}\Big)\Big(\big(\min\{2\cdot O(p(\tau)),1\}\big)^{s_2}+\frac{h(\wderror)}{3}\Big)+h(\wderror)\\
&\le 2^{s_1+s_2}\big(O_{g,C}(|E|/|\mc{A}|)\big)^{s_1}\big(O(p(\tau))\big)^{s_2}+2h(\wderror),
\end{align*}
as desired (since $\alpha<\beta/3$). The second line comes from dropping extra terms in the probability, the third line comes from applying \cref{eq:truncate-nibble-prob-guarantee} (in one of four ways), and the last line comes from expansion. For the fourth line, we separately bound the two terms in the product with the same argument: if $r_2\neq s_2$, then $r_2=\lfloor n^{\beta/3}\rfloor$ in which case either $(O(p(\tau)))^{r_2}<h(\wderror)/3$ or $O(p(\tau))\ge 1/2$. In the latter situation, we see that $(\min\{O(p(\tau)),1\})^{r_2}\le1=(\min\{2\cdot O(p(\tau)),1\})^{s_2}$.

So, it suffices to prove \cref{eq:truncate-nibble-prob-guarantee}. Fix steps $t_1,\dots,t_{r_1}\le \tau_\mr{cut}$. We will estimate the probability that $T^\ast(t_1)=D_1,\dots,T^\ast(t_{r_1})=D_{r_1}$ and $e_1,\dots,e_{r_2}\notin E(\mc M)$; we will then sum this expression over choices of $t_1,\dots,t_{r_1}$. Note that we may assume that no $e_i$ appears in $D_1\cup\dots\cup D_{r_1}$, and in fact the $D_j$ are edge-disjoint (as otherwise the relevant probability is zero).

First, note that (as long as the process does not terminate before time $t$) we always have $|E(t)|=p(t)|E(0)|$, so if the conclusion of \cref{thm:high-girth-nibble} holds then for each $t\le \tau$ we have $|\mc A(t)|=(1\pm n^{-\beta})A(t)$, where $A(t)=p(t)|E(0)|f_{\mr{edge}}(t)/3 = p(t)^3e^{-\rho(t)}|\mc{A}(0)|$.

Next, we observe that for any $e\ne e'$ we always have $|\mc X_e\cap\mc X_{e'}|\le 1\le \sqrt n$, so if the conclusion of \cref{thm:high-girth-nibble} holds, then for each $t\le \tau$ we have
\begin{align*}
\left|\left(\bigcup_{i:t<t_i}\mc T_{D_i}(t)\right)\cup\left(\bigcup_{i=1}^{r_2}\mc X_{e_i}(t)\right)\right|&\ge \sum_{i:t<t_i}|\mc T_{T_i}(t)|+\sum_{i=1}^{r_2}|\mc X_{e_i}|-\binom{r_1+r_2}2\sqrt n\\
&\ge |\{i:t<t_i\}|(1 \pm n^{-\beta})f_\mr{threat}(t)+r_2(1\pm n^{-\beta})f_\mr{edge}(t)-n^{1-5\beta}\\
&\ge (1-2n^{-\beta})(r_2+3|\{i:t<t_i\}|)f_\mr{edge}(t),
\end{align*}
provided $\beta$ is sufficiently small (here we write $\{i:t<t_i\}$ for the set of all $i\in \{1,\dots,r_1\}$ such that $t<t_i$). In the last inequality, we implicitly used $f_\mr{threat}(t)\ge 3 f_\mr{edge}(t)=\Omega_{g,C}(n^{1-3\beta})$ which follows from the given bounds on the initial data.

Now, to have $T^\ast(t_i)=D_i$, at each step $t<t_i$ we must choose a triangle not threatening $D_i$ (i.e., not in $\mc T_{D_i}(t)$), out of a total of $|\mc A(t)|$ possibilities. To have $e_i\notin E(\mc M)$, at each step $t$ we must choose a triangle not in $\mc X_{e_i}(t)$. It follows that
\begin{align}
    &\Pr[T^\ast(t_1)=D_1,\dots,T^\ast(t_{r_1})=D_{r_1}\;\text{and}\;e_1,\dots,e_{r_2}\notin E(\mc M)\text{ and }\mbf E\text{ holds}]\notag\\
    &\qquad \le \!\!\prod_{t\in \{0,\dots,\tau\}\setminus\{t_1,\dots,t_{r_1}\}}\left(1-\frac{(1-2n^{-\beta})(r_2+3|\{i:t<t_i\}|)f_\mr{edge}(t)}{(1+n^{-\beta})A(t)}\right)\prod_{i=1}^{r_1} \frac{1}{(1- n^{-\beta})A(t_i)}.\label{eq:process-wd}
\end{align}
We then compute
\[\sum_{t=0}^{\tau}\frac{f_{\mr{edge}}(t)}{A(t)}=\sum_{t=0}^{\tau}\frac{1}{\left|E\left(0\right)\right|/3-t}=\log\left(\frac{\left|E\left(0\right)\right|/3}{\left|E\left(0\right)\right|/3-\tau}\right)+O\left(1\right)=\log (1/p(\tau))+O(1)\]
and similarly
\[\sum_{t=0}^{\tau}\frac{|\{i:t<t_i\}|f_\mr{edge}(t)}{A(t)}=\sum_{i=1}^{r_1} \left(\log (1/p(t_i))+O(1)\right).\]
Also, we have $f_{\mr{edge}}(t_{1})/A(t_1)+\dots+f_{\mr{edge}}(t_{r_1})/A(t_{r_1})\le r_{1}n^{\beta}/|E(0)|\le 1/\sqrt n$ for small $\beta$. Using the inequality $1-x\le e^{-x}$, we deduce that
\[\Pr[T^\ast(t_1)=D_1,\dots,T^\ast(t_{r_1})=D_{r_1}\;\text{and}\;e_1,\dots,e_{r_2}\notin E(\mc M)]\le\big(O(p(\tau))\big)^{r_2}\prod_{i=1}^{r_1}O\Big(\frac{e^{\rho(t_i)}}{|\mc A|}\Big)+h(\wderror).\]
Noting that $\rho(t)=O_{g,C}(1)$ for all $t$, and summing over all $O(|E|^{r_1})$ choices of $t_1,\dots,t_{r_1}$, we obtain the desired inequality \cref{eq:truncate-nibble-prob-guarantee}.
\end{proof}

\newcommand{\XX}{\mf X}

Now we prove \cref{thm:high-girth-nibble}. An important part of this is to carefully track the evolution of the sets $\mc X_e(t)$. In order to study how these random sets evolve, there is one other class of random variables we will need to track.
\begin{definition}\label{def:XT}Recall the process and notation defined in \cref{def:high-girth-nibble}. For some step $t$ and a triangle $T\in \mc A(t)$, and some $4\le j\le g$ and $0\le c\le j-4$, let $\XX_{T,j,c}(t)$ be the set of $\mc{E}\in\mf J_j$ such that $T\in\mc{E}$, such that $c$ triangles of $\mc{E}$ are already in $\mc C(t)$, and such that the remaining $j-3-c$ triangles are in $\mc A(t)$.
\end{definition}

The configurations in $\XX_{T,j,j-4}(t)$ are especially important because they specify a pair of triangles $\{T,T^\ast\}$ such that choosing $T^\ast$ next would make $T$ unavailable ($T^\ast$ threatens $T$).
The idea is that each $|\mc X_e(t)|$ tends to be close to $f_{\mr{edge}}(t)$, and each $|\XX_{T,j,c}(t)|$ tends to be close to $f_{j,c}(t)$. We define some hitting times that measure the first time we leave these trajectories.

\begin{definition}
In the setting of \cref{thm:high-girth-nibble}, let $B$ be a large constant (depending on $g$ and $C$; large enough to satisfy a certain inequality that will arise in the proof of \cref{thm:high-girth-nibble}). Recalling the notation $p(t)=(|E(0)|-3t)/|E(0)|$ in \cref{def:traj}, we define ``error thresholds''\footnote{These error thresholds are completely different from those in \cite{GKLO20}, and similar to those in \cite{BW19}.}
\begin{itemize}
    \item $e_{\mr{edge}}(t)=p(t)^{-B}n^{1-1/(2C)}$;
    \item $e_{j,c}(t)=e_{\mr{edge}}(t)(p(t)^2|\mc A(0)|/|E(0)|)^{j-4-c}$.
\end{itemize}
Now, recall the functions $f_{\mr{edge}},f_{j,c}$ defined in \cref{def:traj}:
\[p(t)=\frac{|E(0)|-3t}{|E(0)|},\qquad\rho(t) = \sum_{j=4}^g\frac{(j-2)t^{j-3}|\mf{J}_j|}{|\mc{A}(0)|^{j-2}},\]
\[
f_{j,c}(t) = \binom{j-3}{c}\bigg(\frac{t}{|\mc A(0)|}\bigg)^c\big(p(t)^3e^{-\rho(t)}\big)^{j-3-c}\bigg(\frac{(j-2)|\mf{J}_j|}{|\mc A(0)|}\bigg)~\text{ for } 4\le j\le g,\; 0\le c\le j-4,\]
\[
f_{\mr{edge}}(t) = p(t)^2e^{-\rho(t)}\bigg(\frac{3|\mc A(0)|}{|E(0)|}\bigg),\qquad f_{\mr{threat}}(t) = 3f_{\mr{edge}}(t)+\sum_{j=4}^gf_{j,j-4}(t).
\]
and also recall the random sets $\mc X_e(t),\XX_{T,j,c}(t)$ defined in \cref{def:high-girth-nibble,def:XT}. We then define $\tau_{\mr{traj}}$ to be the first $t$ for which
\[\big||\mc X_e(t)|-f_{\mr{edge}}(t)\big|> e_{\mr{edge}}(t)\quad \text{or}\quad \big||\XX_{T,j,c}(t)|-f_{j,c}(t)\big| > e_{j,c}(t)\]
for some $e\in E(t)$ or some $j,c$ and some $T\in \mc A(t)$ (i.e., the point when one of our statistics leaves our predicted trajectory)\footnote{This deviates slightly from the notation in \cite{GKLO20} (the authors of that paper write $\tau_{\mr{violated}}$ for something similar but not exactly the same).}. If there is no $t$ for which this happens, we write $\tau_{\mr{traj}}=\infty$.
\end{definition}

Our main goal is to prove that $\tau_{\mr{traj}}\ge \tau_{\mr{cut}}= \lceil (1-n^{-\beta})|E|/3 \rceil$ (i.e., that our statistics stay on our predicted trajectory until our desired cutoff point). We will prove this with a martingale concentration inequality and the fact that $|\mc X_{e}|-f_{\mr{edge}}$ and $|\XX_{T,j,c}|-f_{j,c}$ are each approximately martingales, with bounded differences. To make this rigorous, we will need crude bounds on certain auxiliary statistics.
\begin{definition}\label{def:crude-statistics}
Using the process and notation defined in \cref{def:high-girth-nibble}, we make the following additional definitions.
\begin{itemize}
    \item For distinct $T, T'\in \mc A(t)$, and any $4\le j\le g$ and $0\le c\le j-5$, let $Z_{T,T',j,c}(t)$ be the number of $\mc E\in \XX_{T,j,c}(t)$ with $T'\in\mc E$.
    \item For $e\in E(t)$ and $T\in \mc A(t)$ with $e\not\subseteq T$, let $Z_{e,T}(t)$ be the number of $\mc E\in \bigcup_{j=4}^g \mf J_j$ such that $\mc E\cap\mc A(t)=\{T,T'\}$ for some triangle $T'$ containing $e$, and $\mc E\setminus\mc{A}(t)\subseteq\mc C(t)$.
    \item For (not necessarily distinct) $T,T'\in \mc A(t)$, let $Z_{T,T'}(t)$ be the number of distinct $\mc{E},\mc{E}'\in \bigcup_{j=4}^g \mf J_j$ such that $\mc{E}\cap \mc A(t)=\{T,T^\ast\}$ and $\mc{E}'\cap \mc A(t)=\{T',T^\ast\}$ for some $T^\ast\in\mc{A}(t)\setminus\{T,T'\}$, and $(\mc E\cup \mc E')\setminus\mc{A}(t)\subseteq \mc C(t)$ (that is, all other triangles in $\mc E,\mc E'$ have been chosen during the process). This implies that $T^\ast$ threatens both $T$ and $T'$ (or if $T = T'$, that it threatens in more than one way).
    \item For $T\in \mc A(t)$, some $4\le j\le g$ and $1\le c\le j-4$, let $Z_{T,j,c-1}(t)$ be the number of distinct pairs $\mc E,\mc E'$ with $\mc E\in \XX_{T,j,c-1}(t)$ and $\mc E'\in \bigcup_{j'=5}^g\mf J_{j'}$, such that $|\mc E'\cap \mc A(t)|=2$, $\mc E'\cap \mc A(t)\subseteq \mc E\cap\mc A(t)$, and $(\mc E\cup \mc E')\setminus\mc{A}(t)\subseteq \mc C(t)$ (in such a case $\mc E$ would actually be redundant; avoiding $\mc E'$ is technically a stronger constraint than avoiding $\mc E$ now).
\end{itemize}
\end{definition}

\begin{lemma}\label{lem:crude-bounds-high-girth-nibble}
In the setting of \cref{thm:high-girth-nibble}, if $\beta,\alpha$ are sufficiently small then with probability $1-\wderror^\alpha-\exp(-n^\alpha)$, for each $t< \tau_{\mr{traj}}\land \tau_{\mr{cut}}$, with the definitions in \cref{def:crude-statistics} we have
\begin{enumerate}
    \item For $4\le j\le g$ and $0\le c\le j-5$, $Z_{T,T',j,c}(t)\le n^{j-c-5+100g\beta}$.
    \item $Z_{e,T}(t)\le n^{100g\beta}$ for every $e\in E(t)$ and $T\in \mc A(t)$.
    \item $Z_{T,T'}(t)\le n^{100g\beta}$ for (not necessarily distinct) $T,T'\in \mc A(t)$.
    \item $Z_{T,j,c-1}(t)\le n^{j-c-3+100g\beta}$ for each $4\le j\le g$, $1\le c\le j-4$, and each $T\in\mc A(t)$.
\end{enumerate}
\end{lemma}

\begin{proof}
If $\beta$ is sufficiently small in terms of $B,C$ then it follows $e_\mr{edge}(t)/f_\mr{edge}(t)\le 1/2$ for all $t < \tau_\mr{traj}\wedge\tau_\mr{cut}$. Therefore for such $t$,
\[|\mc{A}(t)| = \frac{1}{3}\sum_{e\in E(t)}|\mc{X}_e(t)|=\Omega\left((n^{-\beta}|E|)\cdot f_\mr{edge}(\tau_\mr{cut})\right)=\Omega(n^{-4\beta}|\mc{A}|).\]
We immediately deduce that $\mc C(\tau_\mr{traj}\land\tau_\mr{cut})$ contains $s$ specific triangles with probability $(O(n^{4\beta}|E|/|\mc{A}|))^s$. Using $|\mc{A}|\ge n^{1-\beta}|E|$, we see this set of triangles is $(O(n^{5\beta}),0)$-well-distributed.

Recall that the randomness of our generalized high-girth triple process is conditional on the randomness of the initial data (including $\mc D$), and recall that we are assuming $\mc D$ is $(n^\beta,\wderror)$-well-distributed. It follows that $\mc D\cup\mc C(\tau_\mr{traj}\land\tau_\mr{cut})$ is $(n^{6\beta},\wderror)$-well-distributed (for any particular set of $s$ triangles, we sum over all $2^s$ ways to designate those triangles as appearing in $\mc D$ or in $\mc C(\tau_\mr{traj})$).

Given a multiset $\mf{K}$ of sets of triangles, write $X(\mf{K})$ for the random number of sets in $\mf{K}$ that are fully included in $\mc{D}\cup\mc{C}(\tau_\mr{traj}\wedge\tau_\mr{cut})$. Let $\mf K^{(1)}_{T,T',j+r,c+r}$, etc., be defined as in \cref{lem:nibble-moments} with subset sequence $U_0\supseteq\cdots\supseteq U_k$ and with $\mf{F}_j = \mf{F}_j^{\sup}$ for $4\le j\le g$. The desired bounds now follow from \cref{lem:moments,lem:nibble-moments} (with say $s=\min(\frac1{10}\log_n (1/\wderror),n^\beta)$ in \cref{lem:moments}), provided $\alpha$ is sufficiently small relative to $\beta$:
\begin{enumerate}
    \item Note that $Z_{T,T',j,c}(t)\le\sum_{r=0}^{g-j}X(\mf K^{(1)}_{T,T',j+r,c+r})$ for $T,T'\in\mc{A}(t)$ (since $\mf{J}_j$ is ``induced'' from $\bigcup_{j'=4}^g\mf{F}_{j'}^{\sup}$). (1) then follows from \cref{lem:moments} and \cref{lem:nibble-moments}(1), and a union bound over $O_g(n^6)$ choices of $T,T',j,c$.
    \item Note that $Z_{e,T}(t)\le X(\mf K^{(2)}_{e,T})$ for $e\in E(t)$ and $T\in\mc{A}(t)$. (2) then follows from \cref{lem:moments} and \cref{lem:nibble-moments}(2), and a union bound over $O(n^5)$ choices of $e,T$.
    \item Note that $Z_{T,T'}(t)\le X(\mf K^{(3)}_{T,T'})$ for $T,T'\in\mc{A}(t)$, due to the relation between $\mf{J}_j$ and $\mf{F}_j^{\sup}$. (3) then follows from \cref{lem:moments} and 
    \cref{lem:nibble-moments}(3), and a union bound over $O(n^6)$ choices of $T,T'$.
    \item Note that $Z_{T,j,c-1}(t)\le \sum_{r=0}^{g-j}X(\mf K^{(4)}_{T,j+r,c+r-1})$ for $T\in\mc{A}(t)$ (recall that we assume no configuration in $\bigcup_{j=4}^g\mf J_j$ is contained in another). (4) then follows from \cref{lem:moments} and \cref{lem:nibble-moments}(4), and a union bound over $O_g(n^3)$ choices of $T,j,c$.\qedhere
\end{enumerate}
\end{proof}

\begin{proof}[Proof of \cref{thm:high-girth-nibble}]
In this proof we write ``w.s.h.p.'' (short for ``with sufficiently high probability'') to mean that an event holds with probability at least $1-\wderror^\alpha-\exp(-n^\alpha)$, for $\alpha$ that will be chosen sufficiently small to satisfy certain inequalities. Let $\tau_{\mr{crude}}$ be the first time $t$ for which any of the events in \cref{lem:crude-bounds-high-girth-nibble} fail, and $\tau_{\mr{crude}} = \infty$ if there is no such $t$. Note that for $t < \tau_{\mr{crude}}$, $T,T' \in \mc A(t)$ edge-disjoint, and $e \in E(t)$ with $e\nsubseteq T$ we have
\begin{equation}\label{eq:X_T bounds}
    \begin{split}
    X_{T,T'}(t) := |\mc T_T(t)\cap \mc T_{T'}(t)| &\le Z_{T,T'}(t) + \sum_{e'\subseteq T}Z_{e',T'}(t)+\sum_{e''\subseteq T'}Z_{e'',T}(t)\le n^{200g\beta},\\
    X_{e,T}(t) := |\mc X_{e}(t)\cap \mc T_T(t)| &\le Z_{e,T}(t)\le n^{100g\beta},
    \end{split}
\end{equation}
for sufficiently small $\beta$. This verifies the last two of the claims in \cref{thm:high-girth-nibble} for such $t$. Let $\tau_\mr{stop} \coloneqq \tau_{\mr{traj}}\land \tau_{\mr{crude}}\land \tau_{\mr{cut}}$. Our main goal is to show that w.s.h.p.\ 
we have $\tau_{\mr{stop}}=\tau_{\mr{cut}}$, meaning that the process does not terminate and stays on-trajectory until the cutoff time. This also implies the claims about the sets $\mc X_e(t)$ and $\mc T_T$. Indeed, for $t < \tau_{\mr{stop}}$ we have
\begin{equation}\label{eq:X_e bound}
|\mc X_e(t)| = f_{\mr{edge}}(t) \pm e_{\mr{edge}}(t)
\end{equation}
and
\begin{equation}\label{eq:T_T bound}
|\mc T_T(t)| = \sum_{e\in T}|\mc X_{e}(t)|-3+\sum_{j=4}^{g} |\XX_{T,j,j-4}(t)| \pm O_g \left( Z_{T,T}(t) \right)
= f_\mr{threat}(t) \pm O_g \left(e_{\mr{edge}}(t) \right).
\end{equation}
Similarly:
\begin{equation}\label{eq:A_T bound}
    |\mc A(t)| = \frac{1}{3} \sum_{e \in E(t)} |\mc X_e(t)| \stackrel{\cref{eq:X_e bound}}{=} \frac{1}{3}|E(t)| \left( f_{\mr{edge}}(t) \pm e_{\mr{edge}}(t) \right) = \frac{1}{3} |E(0)| p(t) \left( f_{\mr{edge}}(t) \pm e_{\mr{edge}}(t) \right).
\end{equation}
For $\beta$ sufficiently small in terms of $C$ it holds that $e_{\mr{edge}}(t) = o \left( n^{-\beta}f_{\mr{edge}}(t) \right)$ and $e_{\mr{edge}}(t) = o \left( n^{-\beta}f_{\mr{threat}}(t) \right)$ with $t\le\tau_\mr{cut}$. Thus it remains to show that  w.s.h.p.\ $\tau_\mr{stop} = \tau_\mr{cut}$.

For an edge $e\in E(0)$, let $\tau_e$ be the first $t$ for which $e\notin E(t)$ (i.e., the ``point when it gets covered''), and let $\tau^{\mr{freeze}}_{e}=\tau_\mr{stop}\land(\tau_e-1)$. For a triangle $T\in \mc A(0)$, let $\tau_T$ be the first $t$ for which $T\notin \mc A(t)$ (i.e., the ``point when it becomes unavailable'') and let $\tau^{\mr{freeze}}_{T}=\tau_\mr{stop}\land(\tau_T-1)$. Define
\begin{align*}
X_e^\pm(t)&=\pm |\mc X_e(t\land \tau^{\mr{freeze}}_{e})|\mp f_{\mr{edge}}(t\land \tau^{\mr{freeze}}_{e})-e_{\mr{edge}}(t\land \tau^{\mr{freeze}}_{e}),\\
X_{T,j,c}^\pm(t)&=\pm |\XX_{T,j,c}(t\land \tau^{\mr{freeze}}_{T})|\mp f_{j,c}(t\land \tau^{\mr{freeze}}_{T})-e_{j,c}(t\land \tau^{\mr{freeze}}_{T}).
\end{align*}
Here, each of these definitions is really shorthand for two separate definitions, one with a superscript ``$+$'' and one with a superscript ``$-$''. By \cref{lem:crude-bounds-high-girth-nibble} (the conclusion of which occurs with sufficiently high probability) and the definition of $\tau_{\mr{traj}}$, in order to show that $\tau_{\mr{stop}}=\tau_{\mr{cut}}$ w.s.h.p., it suffices to show that $X^s_e(t)\le 0$ and $X^s_{T,j,c}(t)\le 0$ for all $e,T,j,c,t$ and all $s\in \{-,+\}$ with probability at least $1-\exp(-n^\alpha)$.

For $\beta$ sufficiently small in terms of $C$, we have $X^s_e(0)\le -n^{1-2/C}$ and $X^s_{T,j,c}(0)\le -n^{j-3-c-2/C}$ for all $e,T,j,c,s$ by the initial regularity conditions in the definition of $(C,\beta,\wderror)$-goodness. For $0 \leq t \leq \tau_{\mr{stop}}$ we define $\mc C^+(t) = \left( \mc C(t), \mc D, G, \mc A, \mc J_4, \ldots, \mc J_g \right)$. That is, $\mc C^+(t)$ consists of all data contributing to the process up to step $t$. Recalling the notation $\Delta X(t)=X(t+1)-X(t)$, we will show for all such $t,e,T,j,c,s$ that:
\begin{enumerate}
    \item[(A)] $\mb E[\Delta X^s_e(t) | \mc C^+(t)],\mb E[\Delta X^s_{T,j,c}(t)|\mc C^+(t)]\le 0$. That is to say, $X^s_e$ and $X^s_{T,j,c}$ are supermartingales.
    
    \item[(B)] $\mb E\left[\vphantom{|\Delta X^s_{T,j,c}(t)|}|\Delta X^s_e(t)|\middle|\mc C^+(t)\right]\le n^{-1+1/3}$ and $\mb E\left[|\Delta X^s_{T,j,c}(t)|\middle|\mc C^+(t)\right]\le n^{j-5-c+1/3}$.
    
    \item[(C)] $|\Delta X^s_e(t)|\le n^{1/3}$ and $|\Delta X^s_{T,j,c}(t)|\le n^{j-c-4+1/3}$.
\end{enumerate}
The desired result will then follow from Freedman's martingale concentration inequality (\cref{cor:freedman}) as long as $C$ is larger than some absolute constant (which we may assume without loss of generality).

We note that $\Delta X_e^s(t) = \Delta X_{T,j,c}^s(t) = 0$ trivially for $t \geq \tau_{\mr{stop}}$. Similarly $\Delta X_e^s(t) = 0$ if $t \geq \tau_e - 1$ and $\Delta X_{T,j,c}(t) = 0$ if $t \geq \tau_T - 1$. Thus, we may fix an outcome $\mc C^+(t)$ of the process for which $t < \tau_{\mr{stop}}$ and when $\mc X_e(t)$ is involved we condition such that $e \in E(t)$ and $t+1 < \tau_e$, while when $\XX_{T,j,c}(t)$ is involved we condition such that $T \in \mc A(t)$ and $t+1 < \tau_T$.

\medskip
\textit{Step 1: Expected changes. } Let $e \in E(t)$. We note that $\tau_e > t+1$ if and only if $T^*(t) \notin \mc X_e(t)$. Conditioning on this event, and in light of the fact that $|\mc X_e(t)| \leq 4|\mc A(t)| / |E(t)|$, we see that $T^*(t)$ is chosen uniformly at random from $|\mc A(t)| - |\mc X_e(t)| = \left( 1 \pm 4 / |E(t)| \right) |\mc A(t)|$ triangles. Additionally, each $T \in \mc X_e(t)$ will not be in $\mc A(t+1)$ if and only if $T^*(t+1) \in \mc T_T(t) \setminus \mc X_e(t)$. Therefore:
\begin{align*}
    \E \left[ \Delta|\mc X_e(t)| \,\middle|\, \mc{C}^+(t), t < \tau^{\mr{freeze}}_{e}\right] & = - \frac{1}{|\mc A(t)| - |\mc X_e(t)|} \sum_{T \in \mc X_e(t)} \left( |\mc T_T(t)| - |\mc X_e(t)| \vphantom\sum\right)\\
    & \stackrel{\mathclap{\cref{eq:X_e bound}, \cref{eq:T_T bound}}}{=} - \frac{1}{\left( 1 \pm 4 / |E(t)| \right) |\mc A(t)|} \sum_{T \in \mc X_e(t)} \left( f_\mr{threat}(t) - f_{\mr{edge}}(t) \pm O_g \left( e_{\mr{edge}}(t) \right)\vphantom\sum \right)\\
    & = - \frac{| \mc X_e(t)| \left( f_\mr{threat}(t) - f_{\mr{edge}}(t) \pm O_g \left( e_{\mr{edge}}(t) \right) \right)}{\left( 1 \pm 4 / |E(t)| \right) |\mc A(t)|}\\
    & \stackrel{\mathclap{\cref{eq:X_e bound}}}{=} - \frac{\left( f_{\mr{edge}}(t) \pm e_{\mr{edge}}(t) \right) \left(f_\mr{threat}(t) - f_{\mr{edge}}(t) \pm O_g \left( e_{\mr{edge}}(t) \right) \right)}{\left( 1 \pm 4 / |E(t)| \right) |\mc A(t)| }\\
    & = - \frac{f_{\mr{edge}}(t) \left( f_{\mr{threat}}(t) - f_{\mr{edge}}(t) \right)}{|\mc A(t)|} \pm O_g \left( \frac{e_{\mr{edge}}(t)\left( f_{\mr{edge}}(t) + f_{\mr{threat}}(t) \right)}{|\mc A(t)|} \right)\\
    & \stackrel{\mathclap{\cref{eq:A_T bound}}}{=} - \frac{f_{\mr{edge}}(t) \left( f_{\mr{threat}}(t) - f_{\mr{edge}}(t) \right)}{|E(0)| p(t) \left( f_{\mr{edge}}(t) \pm e_{\mr{edge}}(t) \right) / 3} \pm O_g \left( \frac{e_{\mr{edge}}(t)\left( f_{\mr{edge}}(t) + f_{\mr{threat}}(t) \right)}{|\mc A(t)|} \right)\\
    & = - \frac{3 \left( f_{\mr{threat}}(t) - f_{\mr{edge}}(t) \right)}{|E(0)| p(t)} \pm O_g \left( \frac{e_{\mr{edge}}(t)\left( f_{\mr{edge}}(t) + f_{\mr{threat}}(t) \right)}{|\mc A(t)|} \right).
\end{align*}
We observe that for $t < \tau_{\mr{cut}}$ we have $f_{\mr{threat}}(t) = O_{g,C} \left( f_{\mr{edge}}(t) \right) = O_{g,C} \left( |\mc A(t)| / |E(t)| \right)$. Thus:
\begin{equation}\label{eq:E Delta X_e}
\E \left[ \Delta|\mc X_e(t)| \,\middle|\, \mc C^+(t), t < \tau^{\mr{freeze}}_{e} \right] = - \frac{3\left( f_{\mr{threat}}(t) - f_{\mr{edge}}(t) \right)}{|E(0)| p(t)} \pm O_{g,C}\left( \frac{e_{\mr{edge}}(t) }{|E(t)|} \right).
\end{equation}

Next, we calculate $\E \Delta |\XX_{T,j,c}(t)|$ for $T \in \mc A(t)$, $4 \leq j \leq g$ and $0 \leq c \leq j-4$. In contrast with $\mc X_e(t)$, which always decreases with $t$, sets of triangles can be added as well as removed from $\XX_{T,j,c}$ in any given time step. Specifically, a configuration $\mc S \in \XX_{T,j,c-1}(t)$ becomes a member of $\XX_{T,j,c}(t+1)$ if $T^*(t+1) \in \mc S \setminus \{T\}$ --- but only if $T^*(t+1)$ does not threaten an additional triangle in $\mc S$. On the other hand, a configuration $\mc S \in \XX_{T,j,c}(t)$ is not in $\XX_{T,j,c}(t+1)$ if $T^*(t+1) \in \mc S$ or if any of the triangles in $\mc S$ becomes unavailable (i.e., for some $T' \in \mc S$ we have $T^*(t+1) \in \mc T_{T'}(t) \cup \mc S$). With this in mind (and accounting for double counting) we write
\[
\E \left[ \Delta |\XX_{T,j,c}(t)| \,\middle|\, \mc C^+(t), t < \tau^{\mr{freeze}}_{T}\right] = \frac{N^{\mr{loss}}+N^{\mr{gain}}}{|\mc A(t)| - |\mc T_T(t)| - 1},
\]
where
\[
N^{\mr{loss}} = -\sum_{\mc S\in \XX_{T,j,c}(t)}\left(\sum_{T'\in \mc S\cap \mc A(t)-T}|\mc T_{T'}(t)| \pm O\left(1+\sum_{T',T''\in \mc S\cap \mc A(t)}X_{T',T''}\right)\right),
\]
and
\[
N^{\mr{gain}} = (|\XX_{T,j,c-1}(t)|-O_g(Z_{T,j,c-1}(t)))(j-2-c).
\]
(Here, and in the following calculations, some of the random variables, such as $N^{\mr{gain}}$ and $f_{j,c-1}(t)$, are undefined when $c=0$. In these cases we define them as $0$, without further comment).

We observe that since $t < \tau_{\mr{traj}}$ we have $|\mc T_T(t)| = O_{g,C}(f_{\mr{edge}}(t))$ and so
\[
|\mc A(t)| - |\mc T_T(t)| = \left( 1 \pm O_{g,C} \left( \frac{1}{|E(t)|} \right) \right) |\mc A(t)| = \left( 1 \pm O \left( \frac{e_{\mr{edge}}(t)}{f_{\mr{edge}}(t)} \right) \right) |E(0)|p(t) f_{\mr{edge}}(t) / 3.
\]
Similarly, applying \cref{eq:X_T bounds},
\begin{align*}
N^{\mr{loss}} & = - \left(j-3-c\right) f_{j,c}(t) f_{\mr{threat}}(t) \pm O_g \left( e_{j,c}(t)f_{\mr{threat}}(t) + n^{200g\beta} f_{j,c}(t) \right)\\
& = - \left(j-3-c\right) f_{j,c}(t) f_{\mr{threat}}(t) \pm O_{g,C}\left( \frac{e_{j,c}(t) |\mc A(t)|}{p(t) |E(0)|} \right)
\end{align*}
and
\begin{align*}
N^{\mr{gain}} & = (j-2-c) f_{j,c-1}(t) \pm O_g \left( e_{j,c-1}(t)+Z_{T,j,c-1}(t) \right) = (j-2-c) f_{j,c-1}(t) \pm O_{g,C}\left( \frac{e_{j,c}(t) |\mc A(t)|}{p(t) |E(0)|} \right).
\end{align*}
The last inequality comes from $Z_{T,j,c-1}(t)\le n^{j-c-3+100g\beta}$ by \cref{lem:crude-bounds-high-girth-nibble}.

Therefore:
\[
\E \left[ \Delta |\XX_{T,j,c}(t)| \,\middle|\, \mc C^+(t), t < \tau^{\mr{freeze}}_{T}\right] = \frac{(j-2-c)f_{j,c-1}(t) - (j-3-c)f_{j,c}(t) f_{\mr{threat}}(t)}{|E(0)| p(t) f_{\mr{edge}}(t) / 3} \pm O_{g,C}\left( \frac{e_{j,c}(t)}{p(t)|E(0)|} \right).
\]

Since $\Delta |\mc X_e(t)|$ is always negative, \cref{eq:E Delta X_e} implies that
\begin{equation}\label{eq:max change X_e}
\mb E \left[ \vphantom\int\left|\vphantom\sum\Delta |\mc X_e(t)|\right| \,\middle|\, \mc C^+(t),\; t < \tau^{\mr{freeze}}_e \right] \le n^{-1+1/6}.
\end{equation}
Considering separately the contributions from $N^{\mr{loss}}$ and $N^{\mr{gain}}$, we see that
\begin{equation}\label{eq:max-change-XTjc}
\mb E \left[ \vphantom\int\left|\vphantom\sum\Delta |\XX_{T,j,c}(t)|\right| \,\middle|\, \mc C^+(t),\; t < \tau^{\mr{freeze}}_T \right] \le n^{j-5-c+1/6}.
\end{equation}

\medskip
\textit{Step 2: Functional differences.} Direct computation reveals that
\[
f_\mr{edge}'(t) = e^{-\rho(t)} \left( \frac{3|A(0)|}{|E(0)|} \right) \left( 2p(t)p'(t) - \rho'(t)p(t)^2 \right) = - \frac{3}{|E(0)|p(t)} \left( f_\mr{threat}(t) - f_\mr{edge}(t) \right).
\]
Furthermore
\[
f_\mr{edge}''(t) = e^{-\rho(t)} \left( \frac{3|A(0)|}{|E(0)|} \right) \left( 2(p'(t))^2 - 4p(t)p'(t)\rho'(t) + \left( - \rho''(t) + (\rho'(t))^2 \right) p(t)^2 \right).
\]
We note that $p'(t), \rho'(t) = O_{g,C}(1/|E(0)|)$ and $\rho''(t) = O_{g,C}(1 / |E(0)|^2)$. Therefore, for every $0 \leq t \leq \tau_{\mr{cut}}$ we have $|f_\mr{edge}''(t)| = O_{g,C}\left( |A(0)| / |E(0)|^3 \right) = O_{g,C}(n^3\!/n^{3(2-\beta)}) = O \left( e_{\mr{edge}}(t) / (|E(0)|p(t)) \right)$. Thus, by applying Taylor's theorem:
\[
\Delta f_{\mr{edge}}(t) = f_{\mr{edge}}'(t) \pm O\left( \frac{e_{\mr{edge}}(t)}{|E(0)|p(t)} \right).
\]
Similarly
\[
f_{j,c}'(t) = \frac{(j-2-c)f_{j,c-1}(t) - (j-3-c)f_{j,c}(t) f_{\mr{threat}}(t)}{|E(0)| p(t) f_{\mr{edge}}(t) / 3}
\]
and
\[
|f_{j,c}''(t)| = O_{g,C} \left( f_{j,c}(t) \left( \frac{\mbm{1}_{c\ge 2}}{t^2} + \frac{1}{(|E(0)|p(t))^2} + \rho''(t) + (\rho'(t))^2 + \frac{\mbm{1}_{c\ge 1}}{t|E(0)|p(t)} + \frac{\rho'(t)\mbm{1}_{c\ge 1}}{t} + \frac{\rho'(t)}{|E(0)|p(t)} \right) \right).
\]
Therefore $|f_{j,c}''(t)| = O_{g,C}\left( \left( |A(0)| / |E(0)| \right)^{j-3-c} / (|E(0)|p(t)^2) \right) = O\left( e_{j,c}(t) / (|E(0)|p(t)) \right)$. Applying Taylor's theorem:
\[
\Delta f_{j,c}(t) = f_{j,c}'(t) \pm O\left( \frac{e_{j,c}(t)}{|E(0)|p(t)} \right).
\]

We also compute that if $B\ge 2g$ then
\[
\Delta e_{\mr{edge}}(t)\ge \Omega\bigg(\frac{B e_{\mr{edge}(t)}}{p(t)|E(0)|}\bigg),\qquad \Delta e_{j,c}(t)\ge \Omega\bigg(\frac{B e_{j,c}(t)}{p(t)|E(0)|}\bigg).
\]
Hence, if $B$ is sufficiently large in terms of $g,C$ then we have $\mb E \left[ \Delta X_e^\pm(t) \,\middle|\, \mc C^+(t), t+1<\tau^{\mr{freeze}}_e \right] \leq 0$ and $\mb E \left[ \Delta X_{T,j,c}^\pm(t) \middle| \mc C^+(t), t+1<\tau^{\mr{freeze}}_T\right] \le 0$. Next, recall as shown earlier that if $t\ge\tau^{\mr{freeze}}_e$ or $t\ge\tau^{\mr{freeze}}_T$, respectively, then $\Delta X_e^s(t) = 0$ or $\Delta X_{T,j,c}^s(t) = 0$. This justifies (A).

To verify (B) we note that:
\begin{align*}
\E \left[ | \Delta X_e^s(t)| \,\middle|\, \mc C^+(t), t < \tau^{\mr{freeze}}_e \right] & \leq \E \left[ \vphantom{\int_{X_x}} \left|\vphantom\sum \Delta |\mc X_e(t)|\right| \,\middle|\, \mc C^+(t), t < \tau^{\mr{freeze}}_{e} \right] + |\Delta f_{\mr{edge}}(t)| + |\Delta e_{\mr{edge}}(t)|\\
& \stackrel{\cref{eq:max change X_e}}{=} O_{g,C} \left( \frac{n^{1/6}}{n} \right) = o\left( \frac{n^{1/3}}{n} \right).
\end{align*}
Similarly,
\begin{align*}
\E \left[ |\Delta X_{T,j,c}^s(t)| \,\middle|\, \mc C^+(t), t < \tau^{\mr{freeze}}_T \right] & \leq \E \left[\vphantom{\int_{X_x}} \left|\vphantom\sum\Delta |\XX_{T,j,c}(t)|\right| \,\middle|\, \mc C^+(t), t < \tau^{\mr{freeze}}_{T} \right] + |\Delta f_{j,c}(t)| + |\Delta e_{j,c}(t)|\\
& \stackrel{\mathclap{\cref{eq:max-change-XTjc} }}{=} O_{g,C} \left( n^{j-5-c+1/6} \right) = o\left( n^{j-5-c+1/3} \right).
\end{align*}
Finally, again use that if $t\ge\tau^{\mr{freeze}}_e$ or $t\ge\tau^{\mr{freeze}}_T$, respectively, then $\Delta X_e^s(t) = 0$ or $\Delta X_{T,j,c}^s(t) = 0$. This justifies (B).

\medskip
\textit{Step 3: Boundedness. } It remains to verify (C). We have $\Delta |\mc X_e(t)|\le 0$ and for $t < \tau^{\mr{freeze}}_e\le \tau_\mr{crude}$ we have 
\[
-\Delta |\mc X_e(t)|\le 1 + \max_{T^*\in \mc A(t)\setminus \mc X_e(t)} Z_{e,T^*}(t) \le 2 n^{100g\beta},
\]
using the definition of $\tau_\mr{crude}$ as the first time the outcomes of \cref{lem:crude-bounds-high-girth-nibble} are violated. Additionally, for $t < \tau^{\mr{freeze}}_T$:
\[
\Delta |\XX_{T,j,c}(t)|\le \max_{T^*\in \mc A(t)\setminus \{T\}} Z_{T,T^*,j,c-1}(t)\le n^{j-c-4+100g\beta}.
\]
Finally:
\begin{align*}
    -\Delta |\XX_{T,j,c}(t)| &\le \begin{cases}
\max_{T^*\in \mc A(t)\setminus \{T\}} \sum_{T'\in \mc T_{T^*}(t)\cup\{T^*\}\setminus \{T\}}Z_{T,T',j,c}(t) & \text{if }c<j-4\\
\max_{T^*\in \mc A(t)\setminus \{T\}} X_{T,T^*}(t) & \text{if }c=j-4.
\end{cases}\\
&\le n^{j-c-4+100g\beta}.
\end{align*}
We can then deduce (C) (provided $\beta$ is sufficiently small in terms of $g$).
\end{proof}

\subsection{Tracking extension statistics, starting from an almost-complete graph}\label{sec:extension-statistics}
Recall from the outline in \cref{sec:overview} that we need our generalized high-girth triple process in two different settings for the proof of \cref{thm:main}. First, right after planting our absorber we run this process on the remaining edges of the complete graph $K_N$, to ``sparsify'' to some density $p$. Second, we run the process as the main part of the iterative step of the proof of \cref{thm:main}; in stage $k$ of the iteration we use our process to cover almost all edges of $K_N[U_k]$ which do not lie in $U_{k+1}$.

While \cref{thm:random-high-girth-nibble} is sufficient in the latter setting, for the initial sparsification step we need to track a bit more information about the outcome of our generalized high-girth triple process; namely, we need to understand certain ``extension statistics'' between the ``vortex'' sets $U_k$. The following lemma tracks such statistics, but only in the relatively simple setting of the initial sparsification step.

\begin{proposition}\label{prop:initial-nibble}
For constant $g\in \mb N$, there is $\beta=\beta_{\ref{prop:initial-nibble}}(g)>0$ such that the following holds.
Consider an absorbing structure $H$ (with distinguished set $X$) satisfying \cref{AB1} and \cref{AB2} in \cref{thm:absorbers}, embedded in $K_N$. Let $\mc{B}$ and $\mf F_j^\mc{B}$ be as in \cref{lem:absorber-well-spread}, with $|\mc{B}|^{2g}\le N^\beta$. We can then derive data suitable for running the generalized high-girth triple process defined in \cref{def:high-girth-nibble}:

\begin{itemize}
    \item Let $n=N$ and let $G\subseteq K_n$ be the graph of edges which do not appear in $H$.
    \item Let $\mc A$ be the set of triangles in $G$ which do not create an Erd\H os configuration when added to $\mc{B}$.
    \item Let $\mf J_4,\dots,\mf J_g$ be defined as follows. First, let $\mf J_4',\dots,\mf J_g'$ be obtained from $\mf F_4^\mc{B},\dots,\mf F_g^\mc{B}$ by removing all configurations which contain a triangle not in $\mc A$. Then, obtain $\mf J_4,\dots,\mf J_g$ by ``removing redundancies'' from $\mf J_4',\dots,\mf J_g'$ (to be precise, we remove from $\mf J_j'$ all configurations $\mc S$ which fully include some configuration $\mc S'\in\bigcup_{j'<j}\mf J_{j'}'$).
\end{itemize}
Then, for some $\nu = \nu_{\ref{prop:initial-nibble}}(g) > 0$, with probability $1-n^{-\omega(1)}$: the process in \cref{def:high-girth-nibble} runs for $\tau_{\mr{cut}}:= \lceil (1-n^{-\nu})|E|/3 \rceil$ steps without terminating, and for $0\le t\le\tau_{\mr{cut}}$,
\begin{itemize}
    \item $\big||\mc X_{e}(t)|-f_{\mr{edge}}(t)\big|\le n^{-\beta}f_{\mr{edge}}(t)$ for each $e\in E(t)$;
    \item $\big||\mc T_T(t)|-f_{\mr{threat}}(t)\big|\le n^{-\beta}f_{\mr{threat}}(t)$ for each $T\in \mc A(t)$.
    \item $|\mc T_T(t)\cap \mc T_{T'}(t)|\le n^{1/2}$ for each pair of edge-disjoint $T,T'\in \mc A(t)$.
    \item $|\mc T_T(t)\cap \mc X_e(t)|\le n^{1/2}$ for each $T\in \mc A(t)$ and $e\in E(t)$ with $e\not\subseteq T$.
\end{itemize}
Moreover, there is a constant $\tilde B = \tilde B_{\ref{prop:initial-nibble}}(g) > 0$ such that the following hold for any $\zeta\in (0,1/2)$.

\begin{enumerate}
    \item Fix a vertex $v$ and a vertex subset $U\subseteq V(G)$ such that either
    \begin{enumerate}
        \item[(i)] $|U|\ge n^{1/2}$, or
        \item[(ii)] $U$ is disjoint from $V(H)\setminus X$, and $v\notin V(H)$, and $|U|\ge n^{\zeta}$.
    \end{enumerate}
    Let $\tilde \tau_\mr{cut}=(1-|U|^{-1/\tilde B})|E|/3$. Then with probability $1-n^{-\omega_\zeta(1)}$, for each $0\le t\le \tilde \tau_\mr{cut}$: there are $(1+o(1))p(t)|U|$ edges in $E(t)$ between $v$ and $U$.
    \item Fix a set of at most six edges $Q\subseteq E(G)$ and a vertex subset $U\subseteq V(G)$ such that either
    \begin{enumerate}
        \item[(i)] $|U|\ge n^{1/2}$, or
        \item[(ii)] $U$ is disjoint from $V(H)\setminus X$, no edge in $Q$ contains a vertex of $V(H)$, and $|U|\ge n^{\zeta}$.
    \end{enumerate}
    Let $\tilde \tau_\mr{cut}=(1-|U|^{-1/\tilde B})|E|/3$. Then with probability $1-n^{-\omega_\zeta(1)}$, for each $0\le t\le \tilde \tau_\mr{cut}$: there are $(1+o(1))p(t)^{|V(Q)|}e^{-|Q|\rho(t)}|U|$ vertices $u\in U$ such that $xyu\in \mc A(t)$ for each edge $xy\in Q$.
\end{enumerate}
\end{proposition}
Note that $\tau_\mr{cut}$ and $\tilde \tau_\mr{cut}$ are different: if $U$ is very small then we can't run the process for very long.

To prove \cref{prop:initial-nibble} we will extend the analysis in the proof of \cref{thm:high-girth-nibble}.

\begin{proof}
We will take $\beta<\beta_{\ref{thm:high-girth-nibble}}(g,O_g(1))$ (and also assume that $\beta$ is small enough to satisfy certain inequalities throughout the proof).

Let $\mc D=\emptyset$, $k=0$, $U_0=V(K_n)$, $\wderror=0$ and $\mf F_j^\mr{sup}=\mf J_j$ for each $4\le j\le g$. We first claim that these data, together with $G,\mc A,\mf J_4,\dots,\mf J_g$, are $(O_g(1),\beta,n^{-\omega(1)})$-good, as defined in \cref{def:good}. Given this claim, all parts of \cref{prop:initial-nibble} except (1) and (2) will directly follow from \cref{thm:high-girth-nibble}.

Intuitively, this goodness claim follows from the fact that the absorber $H$ is so small that $G$ is very nearly just $K_N$, and $\mc A$ is very nearly just the set of triangles in $K_N$, and $\mf J_j$ is very nearly just the set of $j$-vertex Erd\H{o}s configurations (by symmetry, every triangle is included in the same number of Erd\H{o}s configurations). To be precise, $(O_g(1),\beta,n^{-\omega(1)})$-goodness follows from the following facts.

\begin{itemize}
    \item For every edge $e\in G$, the number of triangles in $\mc{A}$ including $e$ is at least $n-v(H)=(1-n^{-1/3})|\mc A|/|E(G)|$. (\cref{lem:erdos-minimality}(1) implies every vertex in an Erd\H os configuration is contained in at least two triangles of the configuration, so every triangle of $G$ not in $\mc A$ must have all its vertices in $V(H)$.)
    \item For every triangle $T$ in $G$, the number of configurations $\mc S\in \mf J_j$ including $T$ which are \emph{not} Erd\H{o}s configurations is at most \[\sum_{i=1}^{g-j}|\mc{B}|^i n^{j+i-(i+1+3)}=O_g(|\mc{B}|^g n^{j-4})\le n^{-2/3}|\mc A|^{j-3}/|E(G)|^{j-3}\]
    (here we are summing over possibilities for a configuration $\mc S\in \mf J_j$ which arises from an Erd\H{o}s $(i+j)$-configuration including a set $\mc Q$ of $i\ge 1$ triangles of $\mc{B}$, and observing that in such a situation we have $v(\mc Q\cup \{T\})\geq i+1+3$ by \cref{lem:erdos-minimality}(1)).
    \item For every triangle $T$ in $G$, the number of Erd\H{o}s $j$-configurations containing $T$ which do not appear in $\mf J_j$ is at most
    \[\sum_{j'=4}^{g}\sum_{i=1}^{g-j'}|\mc{B}|^i n^{j'+i-(i+1+3)}\cdot n^{j-((j'-2)+3)}=O_g(|\mc{B}|^gn^{j-5})\le n^{-2/3}|\mc A|^{j-3}/|E(G)|^{j-3}\]
    (here we are summing over Erd\H{o}s $j$-configurations which fully include some $\mc S\in \mf J_{j'}$ which arises from an Erd\H{o}s $(i+j')$-configuration, for $i\ge 1$).
    \item For every triangle $T$ in $G$, the number of Erd\H{o}s $j$-configurations containing $T$ which also contain a triangle not in $\mc A$ is at most $(|E(H)|n+|V(H)|^3)\cdot n^{j-5}\le n^{-2/3}|\mc A|^{j-3}/|E(G)|^{j-3}$ (since two edge-disjoint triangles span at least five vertices, and from earlier a triangle of $G$ not in $\mc{A}$ must have all vertices in $V(H)$).
\end{itemize}

Now, it suffices to consider (1) and (2). The proofs of (1) and (2) are very similar (both involve defining an auxiliary supermartingale and using similar analysis as in \cref{thm:high-girth-nibble}), so we handle them together.

First, for (1), we note that $v$ has at least $(1-n^{-1/3})|U|$ neighbors in $U$, with respect to $G$. This is true for different reasons in case (i) and (ii): in case (i), we use that $|V(H)|$ is tiny compared to $|U|$, and in case (ii), we observe that $v$ is adjacent to all vertices of $U$.

For (2), we similarly note that there are at least $(1-n^{-1/3})|U|$ vertices $u\in U$ such that $e\cup\{u\}\in \mc A$ for each $e\in Q$. Again, this is true for different reasons in cases (i) and (ii). In case (ii), we recall from \cref{lem:erdos-minimality}(1) that every vertex in an Erd\H os configuration is contained in at least two triangles of the configuration, so every triangle not in $\mc A$ must have all its vertices in $V(H)$. That is to say, every vertex in $U$ forms a valid triangle with every edge in $Q$. For (i), we again use that $|V(H)|$ is tiny compared to $|U|$.

Recall the definitions of $\tau_\mr{traj}$ and $\tau_\mr{crude}$ from the proof of \cref{thm:high-girth-nibble} (taking $C=O_g(1)$), and let $\tilde Z_{e,T}(t)\le Z_{e,T}(t)$ be the number of forbidden configurations $\mc S\in \bigcup_{j=4}^g\mf J_j$ such that $\mc S\cap \mc A(t)=\{T,T'\}$ for some triangle $T'$ consisting of $e$ and a vertex of $U$. We claim that $\tilde Z_{e,T}(t)\le |U|^{1/3}$ for all $t\le \tau_\mr{traj}\land\tilde \tau_\mr{cut}\land \tau_{\mr{crude}}$, with probability $1-n^{-\omega_\zeta(1)}$. Indeed, in case (i) this is true with probability 1 (we have $\tilde Z_{e,T}(t)\le Z_{e,T}(t)=n^{100g\beta}\le |U|^{1/3}$ for all $t\le\tau_{\mr{crude}}$, provided $\beta$ is sufficiently small). In case (ii), we use basically the same proof as for \cref{lem:crude-bounds-high-girth-nibble}(3), applying \cref{lem:special-nibble-weight-lemma} in place of \cref{lem:nibble-moments}(2). Specifically, for each $t\le \tau_\mr{traj}\land\tilde \tau_\mr{cut}$, we have $|\mc A(t)| = \Omega(|U|^{-4/\tilde B}|\mc A|)$ (where $\tilde B$ is sufficiently large), and therefore any $s$ given triangles appear in $\mc C(\tau_\mr{traj}\land\tilde \tau_\mr{cut})$ with probability $(O(|U|^{4/\tilde B}|E|/|\mc A|))^s$. Noting that $\mc D=\emptyset$, we can then deduce the desired fact from \cref{lem:special-nibble-weight-lemma} and \cref{lem:moments}.

Let $\tilde \tau_\mr{crude}$ be the first time $t$ for which $\tilde Z_{e,T}(t)\ge |U|^{1/3}$. Define the error threshold $\tilde e(t)=p(t)^{-\tilde B/8}|U|n^{-1/6}$.

For (1), let $f(t)=p(t)|U|$ and let $Y(t)$ be the set of vertices $u\in U$ for which $uv\in E(t)$. For (2), let $f(t)=p(t)^{|V(Q)|}e^{-|Q|\rho(t)}|U|$ and let $Y(t)$ be the set of vertices $u\in U$ such that $xyu\in \mc A(t)$ for each $xy\in Q$. For both (1) and (2), we let $\tilde \tau_{\mr{traj}}$ be the first $t$ for which $\big||Y(t)|-f(t)\big|> \tilde e(t)$, and let $\tilde \tau_{\mr{stop}}=\tilde \tau_{\mr{cut}}\land\tau_{\mr{stop}}\land \tilde\tau_{\mr{crude}}\land \tilde\tau_{\mr{traj}}$ (recalling the definition of $\tau_\mr{stop}$ from the proof of \cref{thm:high-girth-nibble}, with $C=O_g(1)$). Then, define
\begin{align*}
X^\pm(t)&=\pm |Y(t\land \tilde\tau_{\mr{stop}})|\mp f(t\land \tilde\tau_{\mr{stop}})-\tilde e(t\land \tilde \tau_{\mr{stop}}).
\end{align*}
It now suffices to show that, for $s\in \{+,-\}$,
\begin{enumerate}
    \item[(A)] $\mb E[\Delta X^s(t)|\mc C(t)]\le 0$,
    \item[(B)] $\mb E\left[\vphantom{|\Delta X^s_{T,j,c}(t)|}|\Delta X^s(t)|\middle|\mc C(t)\right]\le |U|^{1+1/3}/n^2$,
    \item[(C)] $|\Delta X^s(t)|\le |U|^{1/3}$.
\end{enumerate}
Indeed, the desired result will then follow from \cref{cor:freedman}.

\medskip
\textit{Step 1: expected changes. }
Fix $t$ and an outcome of $\mc C^+(t)$ such that $t\le \tilde \tau_\mr{stop}\le \tau_\mr{stop}$. Recall that $\tau_\mr{stop}$ is defined in terms of error functions $e_\mr{edge},e_{j,c}$, and that $e_\mr{edge}(t),e_{j,j-4}(t)=p(t)^{-B}n^{1-1/4}$ for some constant $B$. We may assume that $\tilde B$ is large relative to $B$, so $(e_\mr{edge}(t)+e_{j,j-4}(t))/f_\mr{edge}(t)$ is much smaller than $\tilde e(t)/f(t)$.

Let $A(t) = f_{\mr{edge}}(t)|E(t)|/3$. For (1), we compute
\[\E \left[ \Delta|Y(t)| \,\middle|\, \mc C^+(t)\right]= -\frac{\sum_{u\in Y(t)}|\mc X_{vu}|}{|\mc A(t)|}=\frac{-f(t)f_\mr{edge}(t)}{A(t)} \pm O\left(\frac{\tilde e(t)}{p(t)|E(0)|}\right).\]
For (2), for each $u\in U$, let $\mc Z_{u,Q}$ be the set of $|Q|$ triangles of the form $xyu$ for $xy\in Q$. We have
\begin{align*}\E \left[\Delta |Y(t)| \,\middle|\, \mc C^+(t)\right]&= -\frac{\sum_{u\in Y(t)}\left(\sum_{x\in V(Q)} |\mc X_{xu}(t)|+ \sum_{T\in \mc Z_{u,Q}}\sum_{j=4}^{g} |\XX_{T,j,j-4}(t)|-O(\varepsilon_{V(Q),\mc Z_{u,Q}}^u(t))\right)}{|\mc A(t)|}\\
&= \frac{-|Y(t)|(|V(Q)|f_\mr{edge}(t)+|Q| \sum_{j=4}^gf_{j,j-4}(t))}{A(t)} \pm O_g\bigg(\frac{\tilde e(t)}{p(t)|E(0)|}\bigg),\end{align*}
where 
\[\varepsilon_{V,\mc Z}^u(t)=\sum_{T,T'\in \mc Z}X_{T,T'}(t)+\sum_{T\in \mc Z,x\in V}Z_{ux,T}(t)\le O(n^{200g\beta})\le e_\mr{edge}(t)\]
(recalling the definition of $X_{T,T'}(t)$ from \cref{thm:high-girth-nibble}).

\medskip
\textit{Step 2: computing functional differences. }
For (1), we compute
\[\Delta f(t)=\frac{-f(t)f_\mr{edge}(t)}{A(t)}.\]
For (2), we compute
\[\Delta f(t)=\frac{-|Y(t)|(|V(Q)|f_\mr{edge}(t)+|Q| \sum_{j=4}^gf_{j,j-4}(t))}{A(t)}+O\left(\frac{|U|}{|E(0)|^2}\right)\]
We also compute $\Delta \tilde e(t))\ge \Omega\big(\tilde B \tilde e(t)/(p(t)|E(0)|)\big)$,
so if $\tilde B$ is sufficiently large then $\mb E \Delta X_e^\pm(t),\mb E \Delta X_{T,j,c}^\pm(t)\le 0$, verifying (A). We can also verify (B) with the calculations so far.

\medskip
\textit{Step 3: boundedness. }
Let $t\le \tilde \tau_\mr{stop}$. Note that $\Delta|Y(t)|\le 0$. For (1) we have
$-\Delta|Y(t)|\le 2$ and for (2) we have 
\[-\Delta|Y(t)|\le |Q|\max_{T}\tilde Z_{e,T}(t)\le 6|U|^{1/3},\]
which implies (C).
\end{proof}

\section{The master iteration lemma}\label{sec:iter}

As sketched in \cref{sec:overview}, the key ``cover down'' lemma driving iterative absorption takes as input an appropriate ``vortex'' of sets $K_N=U_0\supseteq U_1\supseteq\dots\supseteq U_\ell$, and a set of edge-disjoint triangles satisfying various properties (in particular, leaving uncovered a ``typical'' subgraph of edges in $K_N[U_k]$, for some $k<\ell$). Then, it extends this set of triangles in such a way that all the edges in $K_N[U_{k}]$ are covered except some of those in $K_N[U_{k+1}]$ (and the remaining uncovered edges in $K_N[U_{k+1}]$ form a typical subgraph). So, if certain assumptions are satisfied at the beginning (e.g., the uncovered edges form a typical subgraph of $K_N=K_N[U_0]$), then it is possible to ``feed this lemma into itself'' $\ell$ times, to obtain a set of edge-disjoint triangles covering all edges not in $K_N[U_\ell]$. In this section we state and prove a master iteration lemma in our setting (using the definitions and results in the previous sections).

The statement of our master iteration lemma will be a little technical, as it must take quite a lot of data as input. First, we require that our graph ``looks like a random set of triangles in a random graph'' in terms of degrees and rooted subgraph statistics, and with respect to the descending sequence of subsets defining our vortex.

\begin{definition}[Iteration-typicality]\label{def:iteration-typical}
Fix a descending sequence of subsets $V(K_n)=U_k\supseteq\dots\supseteq U_\ell$. Consider a graph $G\subseteq K_n$ and a set of triangles $\mc{A}$ in $G$. We say that $(G,\mc{A})$ is \emph{$(p,q,\xi,h)$-iteration-typical} (with respect to our sequence of subsets) if:
\begin{itemize}
    \item for every $k\le \letter<\ell$, every vertex in $U_\letter$ is adjacent to a $(1\pm \xi)p$ fraction both of the vertices in $U_{\letter}$, and of the vertices in $U_{\letter+1}$ (with respect to $G$), and
    \item for any $\letter,\letter^*$ with $k\le \letter<\ell$ and $\letter^*\in \{\letter,\letter+1\}$, and any edge subset $Q\subseteq G[U_\letter]$ spanning $|V(Q)|\le h$ vertices, a $(1\pm\xi)p^{|V(Q)|}q^{|Q|}$-fraction of the vertices $u\in U_{\letter^*}$ are such that $uvw\in\mc{A}$ for all $vw\in Q$.
\end{itemize} 
\end{definition}

In our application, $p\ge n^{-\nu}$ will be a sufficiently decaying function of $n$, and $q=\Omega_g(1)$ will exceed some positive constant (i.e., a constant fraction of the triangles in our graph are available). Also, for our application $\xi$ just needs to be sufficiently small in terms of various constants (but we will take $\xi=o(1)$ for convenience), and $h$ just needs to be at least 4 (this will be necessary to apply \cref{lem:reg-deg}).

We will also need a stronger notion of well-distributedness than that in \cref{def:well-distributed-1}, concerning a random \emph{pair} of sets of triangles $\mc I,\mc D$ (one of which will arise from our initial ``sparsification'' process, and one of which will arise from the master iteration lemma itself). We consider not only the probability that a given set of triangles are present in $\mc I$ and $\mc D$, but also the probability that a given set of edges is uncovered by the triangles in $\mc I$. It is necessary to record all this information so that it can be used as input for the weight estimate lemmas in \cref{sec:weight}.

\begin{definition}[Strong well-distributedness]\label{def:consistent}
Fix a descending sequence of subsets $V(K_N)=U_0\supseteq\dots\supseteq U_k$, and $p\in\mb{R}$. Say that a random pair of sets of triangles $(\mc{I},\mc{D})$ is \emph{strongly $(p,C,\wderror)$-well-distributed} with respect to our sequence of subsets if for any $s,t,r\ge 0$, any distinct triangles $I_1,\ldots,I_s$, $D_1,\ldots,D_{t}$ in $K_N$, and any distinct edges $e_1,\ldots,e_r$ in $K_N$, we have
\[\Pr[I_1,\ldots,I_s\in\mc{I}\text{ and } D_1,\dots,D_{t}\in\mc{D}\text{ and } e_1,\dots,e_r\notin E(\mc{I})]\le C^{s+t+r}\left(p^rN^{-s}\prod_{j=1}^t\frac{p}{|U_{\on{lev}(T_j)}|}+\wderror\right),\]
where $\on{lev}(\cdot)$ is defined with respect to $U_0\supseteq\cdots\supseteq U_k$.
\end{definition}

The reader should think of the ``$p$'' in \cref{def:consistent} as being the same as the ``$p$'' in \cref{def:iteration-typical}, of $C$ as being a large constant ``error factor'', and of $\wderror$ as being a very small additive error term (of the form $n^{-\omega(1)}$). To explain (informally) the meaning of the term ``$p/|U_{\on{lev}(T_j)}|$'': at stage $k$, we select about $p|U_k|^2$ triangles, out of about $p^3|U_k|^3$ available triangles (up to a constant factor depending only on $g$). So, each available triangle is selected with probability about $1/(p^2|U_k|)$. Each triangle in $K_N$ is available with probability about $p^3$, so the overall probability that a triangle is selected at stage $k$ is at most about $p/|U_k|$.

Next, the following definitions describe the initial data (the vortex and the set of forbidden configurations), the data that is recorded at each stage (including the random set of triangles constructed so far), and the assumptions that we need to make about these data.

\begin{definition}[Initial data]\label{def:initial-data}
Consider a descending sequence of subsets $V(K_N)=U_0\supseteq\cdots\supseteq U_\ell$, and collections of sets of triangles $\mf{F}_4,\ldots,\mf{F}_g$ in $K_N$. We say that the data $(U_0,\ldots,U_\ell,\mf{F}_4,\ldots,\mf{F}_g)$ are \emph{$(\rho,\beta,\ell)$-structured} if the following properties hold.
\begin{itemize}
    \item $|U_k|=\lfloor|U_{k-1}|^{1-\rho}\rfloor$ for each $1\le k\le\ell$.
    \item Each $\mf F_j$ is a collection of sets of $j-2$ edge-disjoint triangles, which is $(O_{g,\ell}(1),|U_k|^\beta)$-well-spread with respect to the truncated sequence $U_0\supseteq\dots\supseteq U_k$ for each $0\le k\le \ell$.
\end{itemize}
\end{definition}

\begin{definition}[Data for stage $k$]\label{def:stage-k}
Fix a descending sequence of subsets $V(K_N)=U_0\supseteq\dots\supseteq U_\ell$, and fix collections of sets of triangles $\mf F_4,\dots,\mf F_g$. Now, for some $0\le k<\ell$ consider a random graph $G\subseteq K_N[U_k]$, a random set of triangles $\mc A$ in $G$, and a random pair of sets of triangles $\mc{I},\mc{D}$. We say that the data $(G,\mc A,\mc I,\mc D)$ is \emph{$(p,q,\xi,C,\wderror)$-iteration-good} for stage $k$, with respect to the initial data $(U_0,\dots,U_\ell,\mf F_4,\dots,\mf F_g)$, if:
\begin{enumerate}[{\bfseries{IG\arabic{enumi}}}]
\setcounter{enumi}{-1}
    \item\label{IG0} All degrees of $G$ are even,
    \item\label{IG1} $(\mc{I},\mc{D})$ is strongly $(p,C,\wderror)$-well-distributed with respect to the truncated sequence $U_0\supseteq\cdots\supseteq U_k$,
\end{enumerate}
and if the following three properties are satisfied with probability $1-\xi$.
\begin{enumerate}[{\bfseries{IG\arabic{enumi}}}]
\setcounter{enumi}{1}
    \item\label{IG2} $\mc{I}$ and $\mc{D}$ are disjoint, and $\mc{I}\cup\mc{D}$ is a partial Steiner triple system containing none of the configurations in $\bigcup_{j=4}^g\mf{F}_j$. 
    \item\label{IG3} $(G,\mc{A})$ is $(p,q,\xi,4)$-iteration-typical with respect to the truncated sequence $U_k\supseteq\cdots\supseteq U_\ell$.
    \item\label{IG4} None of the edges of $G$ are covered by the triangles in $\mc{I}\cup\mc{D}$, and for every $T\in\mc{A}$, the set of triangles $\mc{I}\cup\mc{D}\cup\{T\}$ contains no configuration in $\bigcup_{j=4}^g\mf{F}_j$ (that is to say, the triangles in $\mc A$ are still ``available'' for use, in that they would not complete a forbidden configuration on their own).
\end{enumerate}
\end{definition}

We also introduce some notation for how our data are updated as we add new triangles to our partial Steiner triple system.

\begin{definition}[Updating data]\label{def:updated-data}
Fix collections of sets of triangles $\mf{F}_4,\ldots,\mf{F}_g$, let $G\subseteq K_N$ be any graph and $\mc{D}$ any set of triangles in $K_N$, let $\mc A$ be any set of triangles in $G$, and let $U\subseteq V(G)$ be any subset of vertices in $G$. We define $G_U(\mc D)\subseteq G[U]$ to be the graph of edges of $G[U]$ which do not appear in any of the triangles in $\mc D$ (i.e., the edges that are uncovered by $\mc D$). We also define $\mc A_U(\mc D)\subseteq\mc A$ to be the set of all triangles $T\in\mc A$ with edges in $G_U(\mc D)$ for which $\mc D\cup\{T\}$ contains no configuration in $\bigcup_{j=4}^g\mf F_j$ that includes $T$ (i.e., the set of triangles that are still available in combination with $\mc D$).
\end{definition}

Finally, we state our master iteration lemma. Here and in the rest of the paper, $g$ will be treated as a constant for the purpose of all asymptotic notation. We will encounter several other parameters (e.g.\ $\beta,\theta,\rho,\nu$), depending on $g$, which will also be treated as constants for the purpose of asymptotic notation.

\begin{proposition}\label{prop:iter}
There are positive real-valued functions
\[\beta_{\ref{prop:iter}}:\mb N\to \mb R_{>0},\quad\rho_{\ref{prop:iter}}:\mb N\times \mb R_{>0}\to \mb R_{>0},\quad\nu_{\ref{prop:iter}}:\mb N\times \mb R_{>0}\to \mb R_{>0}\]
such that the following holds. Fix a constant $g\in\mathbb{N}$, let $0<\beta\le \beta_{\ref{prop:iter}}(g)$, let $0<\rho\le \rho_{\ref{prop:iter}}(g,\beta)$, and let $0<\nu\le\nu_{\ref{prop:iter}}(g,\rho)$. Fix $(\rho,\beta,\ell)$-structured initial data $(U_0,\dots,U_\ell,\mf F_4,\dots,\mf F_g)$, where $\ell = O(1)$, consider some $k\in \{0,\ldots,\ell-1\}$, and consider data $(G,\mc A,\mc I,\mc D)$ that is $(p,q,\xi,C,\wderror)$-iteration-good for stage $k$, where
\[|U_k|^{-\nu}\le p\le|U_k|^{-\Omega(1)},\quad q = \Omega(1),\quad \xi=o(1),\quad C=O(1),\quad \wderror=n^{-\omega(1)}.\]
Then, there is a random set of triangles $\mc M\subseteq \mc A$, such that with high probability all edges of $G$ are covered by $\mc M$ except those in $G[U_{k+1}]$, and such that the updated data \[(G_{U_{k+1}}(\mc M),\;\mc A_{U_{k+1}}(\mc{I}\cup\mc{D}\cup\mc M),\;\mc I,\;\mc D\cup \mc M)\] is $(p,q,o(1),O(1),n^{-\omega(1)})$-iteration-good for stage $k+1$.
\end{proposition}

We emphasize that the asymptotic notation in the conclusion of \cref{prop:iter} depends on the asymptotic notation in the assumptions (i.e., the $o(1)$ term in the conclusion may for example decay like the square root of the $o(1)$ term in the assumptions). We will only apply \cref{prop:iter} $\ell=O(1)$ times, so we will not actually have to worry about the dependence.

The assumption that $(G,\mc A,\mc I,\mc D)$ are $(p,q,\xi,C,\wderror)$-iteration-good for stage $k$ tells us that \cref{IG2}--\cref{IG4} hold with probability at least $1-\xi=1-o(1)$. In proving \cref{prop:iter}, we may (and do) assume that \cref{IG2}--\cref{IG4} actually hold with probability 1. Indeed, we can simply condition on the event that \cref{IG2}--\cref{IG4} hold; this does not have much of an impact on the strong well-distributedness assumption in \cref{IG1}, because conditioning on an event that holds with probability $1-o(1)$ increases the probability of any other event by a factor of at most $1+o(1)$.

The proof of \cref{prop:iter} will be rather long, so we split it into subsections. As sketched in \cref{sec:overview}, we find our desired set of triangles $\mc M$ via a sequence of several steps. First we ``regularize'' our data using the tools in \cref{sec:regularization}, then we find a set of triangles covering almost all the edges not in $G[U_{k+1}]$, using the generalized high-girth process described in \cref{thm:random-high-girth-nibble} (both these steps are in \cref{sub:iter-nib}). Then, we cover the remaining edges not in $G[U_{k+1}]$ with two further steps: a simple greedy algorithm to cover the remaining ``internal'' edges in $G[U_k\setminus U_{k+1}]$, (in \cref{sub:iter-left}), and a matching argument to cover the remaining ``crossing'' edges between $U_k$ and $U_{k+1}$ (in \cref{sub:iter-link}). Each subsection will begin with a detailed summary of what will be proved in that subsection; the reader may wish to read all these summaries before starting on the details.

\subsection{Regularization and approximate covering}\label{sub:iter-nib}

We now start the proof of \cref{prop:iter}. So, consider initial data $(U_0,\dots,U_\ell,\mf F_4,\dots,\mf F_g)$, and random data $(G,\mc A,\mc I,\mc D)$ for stage $k$, as in the lemma statement. Let $n=|U_k|$. Our goal is to find an appropriate set of triangles $\mc M$ covering all the edges of $G$ that are not in $G[U_{k+1}]$. In this subsection we use our regularization and high-girth process lemmas to find a random set of triangles $\mc M^*\subseteq \mc A$ covering \emph{almost} all the desired edges.

The majority of the uncovered edges outside $G[U_{k+1}]$ will belong to a random set of edges $R$, which we set aside before applying our regularization and high-girth process lemmas. This random graph $R$ will come in handy in the later steps of the proof of \cref{prop:iter}, when we need to cover the leftover edges. To be precise, for any outcome of the random data $(G,\mc A,\mc I,\mc D)$, let $R$ be a random subgraph of $G$ obtained by including each edge between $U_{k+1}$ and $U_k\setminus U_{k+1}$ with probability $n^{-\theta}$ independently. Here $\theta>0$ is a constant that will be chosen sufficiently small with respect to $g,\beta$ (later, we will require $\rho$ to be sufficiently small with respect to $\theta$). Let $G^*=G\setminus (R\cup G[U_{k+1}])$ be the graph consisting of the edges we would like to cover to prove \cref{prop:iter}, other than the edges in $R$.

To summarize what we will prove in this subsection, our random set of triangles $\mc M^*\subseteq \mc A$ will satisfy the following properties.
\begin{enumerate}[{\bfseries{A\arabic{enumi}}}]
    \item \label{item:nibble-inside-graph}Each triangle in $\mc M^*$ lies in the graph $G^*$, and the triangles in $\mc M^*$ are edge-disjoint.
    \item \label{item:nibble-no-dangerous}$\mc I\cup \mc D\cup \mc M^*$ contains no forbidden configuration $\mc S\in \bigcup_{j=4}^g\mf F_j$.
    \item\label{item:nibble-prob-guarantee} With high probability (over $R$ and the random data $\mc I, \mc D,G,\mc A$): for any $s,t\in \mb N$, any distinct edges $e_1,\ldots,e_s\in G^*$ and triangles $F_1^*,\dots,F_t^*$, we have
\[
\Pr\left[F_1^*,\dots,F_t^*\in \mc M^*\text{ and }e_1,\dots,e_s\notin E(\mc M^*)\;\middle|\; R,\mc I, \mc D,G,\mc A\right]\le (C^*n^{-\beta})^s(C^*/(p^2 n))^t+n^{-\omega(1)},
\]
for some $C^*=O(1)$.
\item\label{item:link-overlaps}With high probability, for every pair of distinct vertices $u,w\in U_{k+1}$ there are at most $n^{1-2\theta+o(1)}$ vertices $v\in U_k\setminus U_{k+1}$ for which $uv,vw\in G\setminus E(\mc M^*)$.
\end{enumerate}

\subsubsection{Regularizing the available triangles}
Let $\mc A_0^*\subseteq \mc A$ be the subset of available triangles that are in the graph $G^*$. By \cref{IG3} (iteration-typicality),
\begin{itemize}
    \item every edge $e\in G$ is in $(1\pm o(1))p^2q n$ triangles of $\mc A$, and
    \item for every clique $K\subseteq G$ with $2\le s\le 4$ vertices, there are $(1\pm o(1))p^s q^{\binom s 2}n$ different vertices $u$ which form a triangle in $\mc A$ with every edge of $K$.
\end{itemize}
Recall that $p\ge n^{-\nu}$, that each edge of $R$ is present with probability $n^{-\theta}$, that $|U_{k+1}|=n^{1-\rho}$ and that $G^*=G\setminus (R\cup G[U_{k+1}])$. If $\nu$ is sufficiently small compared to $\rho$ and $\theta$, then using a Chernoff bound over the randomness of $R$, we can see that w.h.p.\ $G$ is sufficiently similar to $G^*$ (and $\mc A$ is sufficiently similar to $\mc A_0^*$) that the above two properties also hold with $G^*$ in place of $G$ and $\mc A^*_0$ in place of $\mc A$. So, by \cref{lem:reg-deg}, we can find $\mc A^*\subseteq \mc A^*_0$ for which every edge $e\in G^*$ is in $(1\pm n^{-1/4})p^2qn/4$ triangles of $\mc A^*$.

\subsubsection{Regularizing the forbidden configurations} We introduce notation for the (random) configurations which are ``locally'' dangerous at stage $k$ (i.e., given that we have already chosen the triangles in $\mc I\cup \mc D$, what subconfigurations must we now avoid?). Specifically, let $\mf J_j$ contain every set $\mc S\subseteq \mc A^*$ of $j-2$ triangles which comprises a forbidden configuration (in $\bigcup_{j'=4}^g \mf F_{j'}$) together with some (possibly empty) set of triangles in $\mc I\cup\mc D$. That is to say, if we can choose $\mc M^*$ in such a way as to avoid each configuration in $\bigcup_{j=4}^g \mf J_j$, then \cref{item:nibble-no-dangerous} is satisfied.

Note that each $\mf J_j$ can be interpreted as a $(j-2)$-uniform hypergraph on the vertex set $\mc A^*$. We would like to apply \cref{lem:reg-threat} to ``regularize'' these hypergraphs. Unfortunately, we cannot directly apply \cref{lem:reg-threat} to each $\mf J_j$. The reason is that these collections of configurations may have some redundancies: a configuration $\mc S\in \mf J_{j}$ may be a proper subset of a configuration $\mc S'\in \mf J_{j'}$, meaning that avoiding $\mc S$ is a stronger property than avoiding $\mc S'$. Our generalized high-girth process analysis cannot tolerate such redundancies (cf.\ \cref{def:good}).

Our application of \cref{lem:reg-threat} is therefore a little delicate. We introduce some notation to remove redundancies: for collections of sets of triangles $\mf S,\mf S_1,\dots,\mf S_h$, we define $\mf S(\mf S_1,\dots,\mf S_h)\subseteq \mf S$ to be the subcollection of $\mf S$ obtained by removing all supersets of configurations in $\bigcup_{j=1}^h\mf S_{j}$. Now, we iteratively construct regularized sets of triangles $\mf J^*_4,\dots,\mf J_g^*$, with no redundancies, such that avoidance of all configurations in $\bigcup_{j=4}^g\mf J_j^*$ implies avoidance of all configurations in $\bigcup_{j=4}^g\mf J_j$. We will do this in such a way that the following conditions are satisfied (viewing each $\mf J_j^*$ as a $(j-2)$-uniform hypergraph on the vertex set $\mc A^*$):
\begin{enumerate}
    \item each $\mf J_j^*$ has maximum degree $o((|\mc A^*|/|E|)^{j-3})$,
    \item the maximum and minimum degree of each $\mf J_j^*$ differ by $n^{o(1)}$,
    \item There is a $O(n^\beta)$-well-spread collection of sets of $j-2$ triangles $\mf F_j^\mr{sup}$ which includes $\mf F_j\cup (\mf J^*_j\setminus \mf J_j)$.
\end{enumerate}

So, suppose for some $4\le j\le g$ we have already constructed $\mf J^*_4,\dots,\mf J^*_{j-1}$ in such a way that (1)--(3) hold; we describe how to construct $\mf J^*_j$. First, we deduce from \cref{lem:nibble-config,lem:moments-asymptotic} that w.h.p.\ $\mf J_j$ has maximum degree $d_\mr{max}\le p^{3(j-3)} n^{j-3+o(1)}$. Here we are applying \cref{lem:nibble-config} with $y=O(1)$ and $z=O(n^\beta)$, where the weights of the form $\pi_{T'}$ describe the probability that a triangle $T'$ is in $\mc I\cup \mc D$ (which is $O(1/N)+O(p/|U_{\on{lev}(T')}|)=O(1/|U_{\on{lev}(T')}|)$), and the weights of the form $\pi_e$ describe the probability that an edge $e$ is uncovered by $\mc I$. Then, with notation as in \cref{lem:moments-asymptotic,lem:nibble-config}, $X(\mf B_{T,j,j'})$ is the number of configurations in $\mf F_{j'}$ which yield a configuration in $\mf J_j$ containing a particular triangle $T$. The joint probability guarantees that we need to apply \cref{lem:moments-asymptotic} are provided by the strong well-distributedness of $(\mc I,\mc D)$.

Next, by \cref{IG3} (iteration-typicality) and a Chernoff bound over the randomness of $R$, w.h.p.\ $|\mc A^*|=\Theta_g(p^3 n^3)$ and $|E(G^*)|=\Theta(pn^2)$, so in fact $d_\mr{max}\le p^{j-3-o(1)}(|\mc A^*|/|E(G^*)|)^{j-3}=o((|\mc A^*|/|E(G^*)|)^{j-3})$. Let $\mf G_j=\mf J_j(\mf J_4^*,\dots,\mf J_{j-1}^*)$ be the subset of $\mf J_j$ obtained by removing redundancies with previously defined $\mf J_{j-1}^*$.

Now, let $\mf H^{(1)}$ be the $(j-2)$-uniform hypergraph with vertex set $\mc A^*$, whose hyperedges are those sets of $j-2$ triangles which are not vertex-disjoint. Let $\mf K$ be the complete $(j-2)$-uniform hypergraph with vertex set $\mc A^*$, and let $\mf H^{(2)}$ be the complement of $\mf K(\mf J_4^*,\dots,\mf J_{j-1}^*)$ (i.e., the hypergraph of supersets of configurations in $\mf J_4^*,\dots,\mf J_{j-1}^*$). The maximum degree of $\mf H^{(1)}$ is $O(n^{3(j-3)-1})$, and the maximum degree of $\mf H^{(2)}$ is at most $\sum_{4\le j'<j}n^{3(j-j')} o((|\mc A^*|/|E|)^{j'-3})=O(n^{3(j-3)-2})$. So, \cref{lem:reg-threat} (applied with $\mf G=\mf G_j$ and $\mf H=\mf H^{(1)}\cup \mf H^{(2)}\cup \mf G_j$, both of which have $|\mc A|$ vertices) gives us a collection of sets of triangles $\mf G_j'$, such that $\mf J_j^*:=\mf G_j'\cup \mf G_j$ w.h.p.\ satisfies (1)--(2) above.
Also, $\mf G_j'$ is stochastically dominated by a binomial random hypergraph $\mf G_j^{\mr{rand}}\sim\mb G^{(j-2)}(|\mc A|,p_j)$, for $p_j=n^\beta n^{-2j+6}$ (assuming $\nu$ is sufficiently small compared to $\beta$). Let $\mf F^\mr{sup}_j=\mf G_j^{\mr{rand}}\cup \mf F_j$, and note that by \cref{lem:random-well-spread}, w.h.p.\ this collection of sets of triangles is $O(n^\beta)$-well-spread. So, there is an outcome of $\mf J_j^*$ satisfying (1)--(3), as desired.

\subsubsection{The generalized high-girth triple process}

We now apply \cref{thm:random-high-girth-nibble} with the random objects $\mc D\cup \mc I, G^*,\mc A^*,\mf J_4^*,\dots,\mf J_g^*$. (We assume $\nu<\beta/2$, so $p\ge n^{-\beta/2}$). We thereby obtain a random set of triangles $\mc M^*\subseteq \mc A^*$ satisfying \cref{item:nibble-inside-graph}--\cref{item:nibble-prob-guarantee}. More explicitly, $\mc M^*$ satisfies \cref{item:nibble-inside-graph} as we have applied \cref{thm:random-high-girth-nibble} with the graph $G^{\ast}$. That $\mc M^*$ satisfies \cref{item:nibble-no-dangerous} follows from the first item of \cref{thm:random-high-girth-nibble} as we are only using triangles from $\mc A^*$ and we have constructed $\mf J^*_4,\dots,\mf J_g^*$ such that avoidance of all configurations in $\bigcup_{j=4}^g\mf J_j^*$ implies avoidance of all configurations in $\bigcup_{j=4}^g\mf J_j$. Finally, the second item of \cref{thm:random-high-girth-nibble} corresponds exactly to \cref{item:nibble-prob-guarantee}. 

\subsubsection{Counting extensions of edges}

Finally, we verify \cref{item:link-overlaps}. For any $u,w\in U_{k+1}$ and $v\in U_k\setminus U_{k+1}$, the probability that $vu,vw\in G\setminus E(\mc M^*)$ is at most $n^{-2\theta}+O(n^{-\theta-\beta})+O(n^{-2\beta})\le 2n^{-2\theta}$ by \cref{item:nibble-prob-guarantee} ($vu$ and $vw$ must be in $R$ or uncovered by $\mc M^*$, and we are assuming $\theta$ is small relative to $\beta$). In fact, for any $s\in \mb N$ and any $s$ choices of $v$, the probability that $vu,vw\in G\setminus E(\mc M^*)$ for all these $v$ is at most $(2n^{-2\theta})^s+n^{-\omega(1)}$, so by \cref{cor:moments-asymptotic-simple}, with probability $1-n^{-\omega(1)}$, there are at most $n^{1-2\theta+o(1)}$ different $v$ for which this is the case. The desired result then follows from a union bound over choices of $u,w$.

\subsection{Covering leftover internal edges}\label{sub:iter-left}
So far, we have a random set of triangles $\mc M^*$ covering almost all of the edges of $G^*$. We distinguish two types of edges we still need to cover to prove \cref{prop:iter}: first, we need to handle the remaining uncovered edges in $G[U_k\setminus U_{k+1}]$ (i.e., those fully outside $U_{k+1}$), and second, we need to handle the remaining uncovered edges between $U_k\setminus U_{k+1}$ and $U_{k+1}$ (some of these are in $G^*$, but most are in $R$).

In this subsection, we handle all the edges of the first type. Namely, using a simple random greedy procedure, we augment $\mc M^*$ with a further random set of triangles $\mc M^\dagger\subseteq \mc A$, covering all the remaining edges in $G[U_k\setminus U_{k+1}]$. In the process, we will also cover a few edges of our reserve graph $R$. To summarize what we will prove in this subsection, our random set of triangles $\mc M^\dagger\subseteq \mc A$ will satisfy the following properties.
\begin{enumerate}[{\bfseries{B\arabic{enumi}}}]
    \item Each triangle in $\mc M^\dagger$ consists of an edge in $G[U_k\setminus U_{k+1}]\setminus E(\mc M^*)$ and two edges in $R$, and the triangles in $\mc M^\dagger$ are edge-disjoint.
    \item With high probability, no edge in $G[U_{k}\setminus U_{k+1}]$ is left uncovered by $\mc M^*\cup \mc M^\dagger$.
    \item $\mc I\cup \mc D\cup \mc M^* \cup \mc M^\dagger$ contains no forbidden configuration $\mc S\in \bigcup_{j=4}^g\mf F_j$.
    \item\label{item:CD1-prob-guarantee} Conditioning on any outcome of $R,\mc M^*$ and the random data $\mc I, \mc D,G,\mc A$: for any $s\in \mb N$ and any distinct triangles $F_1^\dagger,\ldots,F_s^\dagger\in G$ we have
\[
\Pr\left[F_1^\dagger,\dots,F_s^\dagger\in \mc M^\dagger\;\middle|\; R,\mc M^*,\mc I,\mc D,G,\mc A\right]\le (C^\dagger/(n^{1-\rho-2\theta} p^2))^s,
\]
for some $C^\dagger=O(1)$.
\end{enumerate}

\subsubsection{Defining a random greedy procedure}
Arbitrarily enumerate the uncovered edges in $G[U_{k}\setminus U_{k+1}]$ as $f_1,\dots,f_m$. By iteration-typicality and a Chernoff bound using the randomness of $R$, for each $i\le m$ there are $(1+o(1))n^{-2\theta} p^2 q |U_{k+1}|=(1+o(1))n^{1-\rho-2\theta} p^2 q$ choices for a triangle $T_i\in \mc A$ consisting of $f_i$ and two edges of $R$ (call these ``candidates for $T_i$''). However, we cannot arbitrarily choose such a triangle for each $i$, as these triangles may edge-intersect with each other, and they may form forbidden configurations (with each other and with $\mc M^*$, in combination with the triangles in $\mc I\cup \mc D$).

We will choose our triangles $T_1,\dots,T_m$ via the following random greedy procedure. For each $i\le m$, in order, we will consider all possible triangles $T\in \mc A$ consisting of $f_i$ and two edges of $R\setminus E(T_1\cup\dots\cup T_{i-1})$, such that $\mc I\cup \mc D\cup \mc M^*\cup \{T_1,\dots,T_{i-1},T\}$ contains no forbidden configuration. We will choose such a triangle uniformly at random to take as $T_i$.

Of course, it is possible that there are no such triangles, in which case we do not define $T_i$. In fact, we leave $T_i$ undefined (and say that our entire random greedy procedure fails) whenever there are fewer than $n^{1-\rho-2\theta} p^2 q/2$ choices for $T_i$ at step $i$ (this ensures that at every step there is ``a lot of randomness'' in the triangles chosen throughout the process, and in particular that \cref{item:CD1-prob-guarantee} holds). Let $\mc M^\dagger$ be the set of $T_i$ which actually get defined in our process.

\subsubsection{A well-distributedness-type property}\label{sub:CD1-consistency} Recall that we are assuming $(\mc I,\mc D)$ is strongly $(p,C,n^{-\omega(1)})$-well-distributed (\cref{IG1}), for some $C=O(1)$, and recall the probabilistic guarantee on $\mc M^*$ from \cref{item:nibble-prob-guarantee}. We now claim that $(\mc I,\mc D\cup \mc M^*\cup\mc M^\dagger)$ is strongly $(p,O(1),n^{-\omega(1)})$-well-distributed. In fact we need a slightly stronger property incorporating $R$: for any $s,t,r,u\in \mb N$, any distinct triangles $I_1,\ldots,I_s$, $D_1,\ldots,D_{t}$ in $K_N$, and any distinct edges $e_1,\dots,e_r,e_1',\dots,e_u'\in K_N$, with $e_1',\dots,e_u'$ being between $U_k\setminus U_{k+1}$ and $U_{k+1}$, let $\mbf E$ be the event that $I_1,\dots,I_s\in\mc I$, and $e_1,\dots,e_r$ are uncovered by $\mc I$, and $D_1,\dots,D_t\in \mc D\cup \mc M^*\cup \mc M^\dagger$, and $e_1',\dots,e_u'\in R$. We claim that there is a constant $C_{\ref{sub:CD1-consistency}}=O(1)$ such that 
\begin{equation}\Pr[\mbf E]\le C_{\ref{sub:CD1-consistency}}^{s+t+r+u}\left(p^{r}(pn^{-\theta})^uN^{-s}\prod_{j=1}^t\frac{p}{|U_{\on{lev}(D_j)}|}+n^{-\omega(1)}\right).\label{eq:CD1-consistency}\end{equation}

This will allow us to apply \cref{lem:moment-left} to analyze our random greedy procedure. To verify our claim, it is convenient to partition $\mbf E$ into sub-events, which we will later sum over. Let $\mc T=\{D_1,\dots,D_t\}$, and consider a partition $\mc T=\mc T^0\cup \mc T^*\cup \mc T^\dagger$. We consider the sub-event $\mbf E'\subseteq \mbf E$ that $\mbf E$ occurs, and moreover that $\mc T^0\subseteq \mc D,\mc T^*\subseteq \mc M^*,\mc T^\dagger\subseteq \mc M^\dagger$. We have
\begin{align*}&\Pr[\mbf E']\le C^{s+t+r}\left(p^{r+u+3|\mc T^*|}N^{-s}\prod_{D\in \mc T^0}\frac{p}{|U_{\on{lev}(D)}|}+n^{-\omega(1)}\right)\cdot(n^{-\theta})^u\\
&\qquad\qquad\qquad\qquad\qquad\cdot\left(\left(\frac{C^*}{p^2 n}\right)^{|\mc T^*|}\left(\frac{C^*}{n^{\beta}}\right)^{|\mc T^\dagger|}+n^{-\omega(1)}\right)\cdot\left(\frac{C^\dagger}{n^{-2\theta} p^2 n^{1-\rho}}\right)^{|\mc T^\dagger|}.
\end{align*}
Indeed, for $\mbf E'$ to occur, first we must have $I_1,\dots,I_s\in \mc I$, then $e_1,\dots,e_r,e_1',\dots,e_u'$ and the edges in the triangles in $\mc T^*$ must be uncovered by $\mc I$, then we must have $\mc T^0\subseteq \mc D$. The probability of these events (over the randomness of $\mc I,\mc D$) may be bounded using \cref{IG1}. Then, given an outcome of $(\mc I,\mc D)$, the edges $e_1',\dots,e_u'$ must be in $R$. The probability of this event is $(n^{-\theta})^u$. Then, we must have $\mc T^*\subseteq \mc M^*$, and the edges in the triangles in $\mc T^\dagger$ must be uncovered by $\mc M^*$. The probability of these events can be bounded using \cref{item:nibble-prob-guarantee}. Finally, we must have $\mc T^\dagger\subseteq \mc M^\dagger$, and the probability of this can be bounded by \cref{item:CD1-prob-guarantee} (which we have already proved). 

If $\theta,\nu,\rho$ are sufficiently small with respect to $\beta$, then summing this bound for $\Pr[\mbf E']$ over all $3^t$ choices of the partition $\mc T^0\cup \mc T^*\cup \mc T^\dagger$ yields the desired bound \cref{eq:CD1-consistency}, with say $C_{\ref{sub:CD1-consistency}}=3CC^*C^\dagger$.

\subsubsection{Bounding the impact of forbidden configurations.} We now show that w.h.p.\ for each $i\le m$, only a $o(1)$-fraction of candidates for $T_i$ are forbidden due to the fact that their addition to $\mc M^\dagger$ would create a forbidden configuration (together with other triangles in $\mc I\cup \mc D\cup \mc M^*\cup \mc M^\dagger$).

Fix $i\le m$ and let $\mc M^\dagger_i$ be the set of triangles we have chosen before the $i$-th step of our random greedy procedure (so, $\mc M^\dagger_i$ contains the triangles $T_1,\dots,T_{i-1}$, unless the procedure has already failed). Since $\mc M^\dagger_i$ is a subset of $\mc M^\dagger$ for each $i$, the well-distributedness-type property described in \cref{eq:CD1-consistency} holds with $\mc M^\dagger_i$ in place of $\mc M^\dagger$. We now deduce from \cref{lem:moment-left,lem:moments} that with probability $1-n^{-\omega(1)}$ at most $n^{-2\theta}p^3n^{1-\rho+o(1)}$ of the candidates for $T_i$ (that is to say, an $o(1)$ fraction of these candidates) cannot be chosen due to forbidden configurations.

Indeed, here we are applying \cref{lem:moment-left} with $y=O(1)$, $z=O(n^\beta)$, $r=n^{-\theta}$ and $e=f_i$, where the weights of the form $\pi_{T,1}$ and $\pi_{T,2}$ describe the probability that a triangle $T$ is in $\mc I$ or $\mc D\cup\mc M^*\cup \mc M^\dagger$, respectively, and the weights of the form $\pi_e$ describe the probability that an edge $e$ is uncovered by $\mc I$ and appears in $R$. In the notation of \cref{lem:moments-asymptotic,lem:moment-left}, $X(\mf L_{f_i,j})$ is an upper bound on the number of configurations $\mc S\in \mf F_{j}$ containing a candidate for $T_i$ for which $\mc S\setminus \{T\}\subseteq\mc I\cup\mc D\cup \mc M^*\cup\mc M^\dagger$. The joint probability guarantees that we need to apply \cref{lem:moments-asymptotic} are provided by \cref{eq:CD1-consistency}.

\subsubsection{Bounding the impact of edge-intersections} \label{subsub:edge-intersections-bound} We now show that for each $i\le m$, with probability $1-n^{-\omega(1)}$, only $n^{o(1)-\beta}p|U_{k+1}|$ of the candidates for $T_i$ (that is, a $o(1)$-fraction, provided $\theta,\nu$ are small compared to $\beta$) are forbidden due to the fact that they edge-intersect previously chosen triangles in $\mc M^\dagger$. Note that a previously chosen triangle $T_j\in \mc M^\dagger$ can only edge-intersect a candidate for $T_i$ if $f_j$ and $f_i$ share a vertex. So, it suffices to show that with probability $1-n^{-\omega(1)}$ every vertex $v\in U_k\setminus U_{k+1}$ is incident to at most $n^{o(1)-\beta}p|U_{k+1}|$ of the $(1+o(1))p|U_{k+1}|$ edges of $G$ which were uncovered by $\mc M^*$. By \cref{item:nibble-prob-guarantee}, the probability a given edge $e\in E(G)$ incident to $v$ is uncovered by $\mc M^*$ is at most $O(n^{-\beta})$, and in fact we can use \cref{item:nibble-prob-guarantee} and \cref{cor:moments-asymptotic-simple} to deduce the desired fact.

\subsection{Covering leftover crossing edges}\label{sub:iter-link}

At this point, we have found a random set of triangles $\mc M^*$ covering almost all the edges of $G^*$, as well as a random set of triangles $\mc M^\dagger$ covering all the remaining edges in $G[U_k\setminus U_{k+1}]$. In this subsection we will find a final random set of triangles $\mc M^\ddagger\subseteq \mc A$ covering all the remaining edges between $U_k\setminus U_{k+1}$ and $U_{k+1}$ (most of which are in the reserve graph $R$). The problem of covering the remaining edges between $U_{k+1}$ and a given vertex $v\in U_{k}\setminus U_{k+1}$ can be viewed as a matching problem for a certain auxiliary graph (a ``link graph'' associated with $v$), so in this subsection we apply the tools from \cref{sec:subsample}.

Let $\gamma>0$ be a constant which is small with respect to $g,\beta,\theta,\rho,\nu$. To summarize what we will prove in this subsection, our random set of triangles $\mc M^\ddagger\subseteq\mc A$ will satisfy the following properties.
\begin{enumerate}[{\bfseries{C\arabic{enumi}}}]
    \item Every triangle in $\mc M^\ddagger$ consists of a single edge in $G[U_{k+1}]$, and two edges between $U_k\setminus U_{k+1}$ and $U_{k+1}$. Also, the triangles in $\mc I\cup \mc D\cup \mc M^*\cup \mc M^\dagger\cup \mc M^\ddagger$ are edge-disjoint.
    \item With high probability, no edge in $G\setminus G[U_{k+1}]$ is left uncovered by $\mc M^*\cup \mc M^\dagger\cup \mc M^\ddagger$.
    \item $\mc I\cup \mc D\cup \mc M^* \cup \mc M^\dagger\cup \mc M^\ddagger$ contains no forbidden configuration $\mc S\in \bigcup_{j=4}^g\mf F_j$.
    \item \label{item:CD2-prob-guarantee} Conditioning on any outcome of $R,\mc M^*,\mc M^\dagger$ and the random data $\mc I, \mc D,G,\mc A$: every set of $s\in \mb N$ distinct triangles is present in $\mc M^\ddagger$ with probability at most $(n^\gamma/(n^{-\theta} p^2 n^{1-\rho}))^s$. That is to say, $\mc M^\ddagger$ is approximately stochastically dominated by a random set of triangles each independently selected with this probability.
\end{enumerate}
After this subsection, it will just remain to verify that the union $\mc M=\mc M^*\cup \mc M^\dagger\cup \mc M^\ddagger$ satisfies the conclusion of \cref{prop:iter}.

\subsubsection{Defining link graphs}
For a vertex $v\in U_k\setminus U_{k+1}$, let $N_{U_{k+1}}(v)\subseteq U_{k+1}$ be the set of vertices $u\in U_{k+1}$ such that $vu\in G$. Let $L_v$ be the graph with vertex set $N_{U_{k+1}}(v)$, and an edge $uw$ whenever $vuw\in \mc A$. Let $W_v$ be the set of vertices $u\in N_{U_{k+1}}(v)$ such that $uv$ has not yet been covered by $\mc M^*\cup \mc M^\dagger$.

Now, our goal is to find a set of triangles $\mc M^\ddagger$ which covers the remaining edges in $G\setminus G[U_{k+1}]$ in pairs (each triangle consists of two edges incident to some vertex $v\in U_k\setminus U_{k+1}$, and an edge in $G[N_{U_{k+1}}(v)]$). The triangles we are allowed to use correspond precisely to the edges in $L_v[W_v]$: in the language of our link graphs, our task is to find a perfect matching in each $L_v[W_v]$. However, not just any perfect matchings will do; we need to avoid forbidden configurations and our perfect matchings need to be edge-disjoint.

Our plan is to randomly sparsify the link graphs, and then delete all edges that could possibly cause problems with forbidden configurations or edge-disjointness. If we sparsify sufficiently harshly, we will be able to show that very few subsequent deletions are actually necessary, which will allow us to apply \cref{lem:match}.

\subsubsection{Typicality of link graphs}
For a vertex $v\in U_k\setminus U_{k+1}$, let $N_R(v)$ be the set of vertices $u\in U_{k+1}$ such that $vu\in R$. By iteration-typicality and a Chernoff bound using the randomness of $R$, each $|N_R(v)|=(1\pm o(1))pn^{-\theta}|U_{k+1}|=(1\pm o(1))pn^{1-\rho-\theta}$ and in each link graph $L_v$:
\begin{itemize}
    \item each vertex $u\in N_{U_{k+1}}(v)$ has $(1\pm o(1))p^2qn^{-\theta}|U_{k+1}|=(1\pm o(1))pq|N_R(v)|$ neighbors in $N_R(v)$, and
    \item each pair of vertices $u,w\in N_{U_{k+1}}(v)$ have $(1\pm o(1))p^3q^2n^{-\theta}|U_{k+1}|=(1\pm o(1))(pq)^2|N_R(v)|$ common neighbors in $N_R(v)$.
\end{itemize}
To see this, note that the degree of $u$ in $L_v$ is simply the number of vertices $x\in U_{k+1}$ such that $uvx$ is a triangle, and the number of common neighbors of $u$ and $w$ in $L_v$ is simply the number of vertices $x\in U_{k+1}$ such that $uvx$ and $wvx$ are both triangles. These quantities are both controlled by iteration-typicality. Then, $N_R(v)$ can be interpreted as a binomial random subset of $N_{U_{k+1}}(v)$, where each vertex is present with probability $n^{-\theta}$.

Now, as already observed in \cref{subsub:edge-intersections-bound}, it follows from iteration-typicality, \cref{item:nibble-prob-guarantee}, and \cref{cor:moments-asymptotic-simple} that with probability $1-n^{-\omega(1)}$, for each vertex $v\in U_k\setminus U_{k+1}$ there are at most $n^{o(1)-\beta}p|U_{k+1}|$ edges in $G\setminus R$ which are incident to $v$ and not covered by $\mc M^*$. When this is the case, for each $v$ the symmetric difference $N_R(v)\triangle W_v$ has size at most $n^{o(1)-\beta}p|U_{k+1}|$. So, if $\theta,\nu$ are small with respect to $\beta$, then $L_v[W_v]$ has $(1\pm o(1))pn^{-\theta}n^{1-\rho}$ vertices and is $(pq,o(1))$-typical (in the sense of \cref{def:typical}).

\subsubsection{Sparsification, deletion and matching}\label{subsub:sparsification}
Recall that $\gamma>0$ is a constant which is small with respect to $g,\beta,\theta,\rho$. For each $v\in U_k$ let $S_v$ be the randomly sparsified subgraph of $L_v[W_v]$ obtained by retaining each edge with probability $n^\gamma/(n^{-\theta}p^2qn^{1-\rho})$ (independently for each $v$ and each edge).

We delete some edges from the sparsified link graphs $S_v$. First, let $D_1$ be the set of edges which appear in multiple $S_v$. Second, recall that every edge of $L_v$ corresponds to a triangle in $\mc A$; let $\mc R$ be the collection of sets of edges in $\bigcup_v S_v$ whose corresponding triangles create a forbidden configuration when combined with triangles in $\mc I\cup \mc D\cup \mc M^*\cup \mc M^\dagger$. Let $D_2$ be the set of all edges in the configurations in $\mc R$. If we can find a perfect matching in each $S_v\setminus (D_1\cup D_2)$, we will be able to take the corresponding triangles as $\mc M^\ddagger$.

In order to apply \cref{lem:match} we need to bound the maximum degrees of each $S_v\cap D_1$ and $S_v\cap D_2$. First,  \cref{item:link-overlaps} implies that w.h.p.\ every edge appears in at most $n^{1-2\theta+o(1)}$ of our link graphs $L_v[W_v]$. Condition on an outcome of $R,\mc M^*$ such that this is the case, and fix a vertex $v\in U_k\setminus U_{k+1}$. Conditioning on the event that an edge $e$ appears in $S_v$, the probability that $e$ appears in some other $S_{v'}$ is at most $n^{-\theta+2\nu+\gamma+\rho+o(1)}=n^{-\Omega(1)}$, by a union bound over choices of $v'$ (recalling that $\nu,\gamma,\rho$ are small compared to $\theta$). By a Chernoff bound (and the independence of our random sparsifications), it follows that with probability $1-n^{-\omega(1)}$ the maximum degree of $S_v\cap D_1$ is at most $o(n^\gamma)$.

Similarly, using \cref{lem:moment-link,lem:moments-asymptotic} we can see that with probability $1-n^{-\omega(1)}$ the maximum degree of $S_v\cap D_2$ is at most $n^{o(1)}=o(n^\gamma)$. Indeed, for each $w\in N_{U_{k+1}}(v)$ we apply \cref{lem:moment-link} with $y=O(1)$, $z=O(n^\beta)$, $r=n^{-\theta}$ and $e=vw$. The weights of the form $\pi_{T,1}$ and $\pi_{T,2}$ describe the probability that a triangle $T$ is in $\mc I$ or $\mc D\cup\mc M^*\cup \mc M^\dagger$, and the weights of the form $\pi_{T,3}$ describe the probability that a triangle $T=\{v',u,u'\}$ is such that $uu'\in S_{v'}$. The weights of the form $\pi_{e'}$, for $e'$ between $U_k\setminus U_{k+1}$ and $U_{k+1}$ describe the probability that $e'$ is uncovered by $\mc I$ and appears in $R$, and the weights of the form $\pi_{e'}$, for $e$ inside $U_{k+1}$, simply describe the probability that $e'$ is uncovered by $\mc I$. Then, with notation as in \cref{lem:moment-link,lem:moments-asymptotic}, $X(\mf M_{e,j})$ is an upper bound on the number of edges in $D_2$ incident to $w$. The joint probability guarantees that we need to apply \cref{lem:moments-asymptotic} are provided by \cref{eq:CD1-consistency}, and the independently randomly sparsified link graphs $S_{v'}$. (Here we must also ensure $pn^\gamma\le 1$ to apply \cref{lem:moment-link}, which follows from the assumption $p\le|U_k|^{-\Omega(1)}$ and choosing $\gamma$ sufficiently small.)

We conclude that, with probability $1-n^{-\omega(1)}$, the graphs $S_v\setminus (D_1\cup D_2)$ are each suitable for application of \cref{lem:match} and we can find the desired perfect matchings. \cref{item:CD2-prob-guarantee} follows from the fact that the desired matchings are stochastically dominated by binomial sampling (conditional on suitability).

\subsection{Verifying iteration-goodness}\label{sub:iter-final}

Finally, we complete the proof of \cref{prop:iter} by verifying that the conditions of $(p,q,o(1))$-iteration-goodness are satisfied by including the triangles $\mc M=\mc M^*\cup \mc M^\dagger\cup \mc M^\ddagger$ (specifically, update the data to $(G_{U_{k+1}}(\mc M),\mc A_{U_{k+1}}(\mc I\cup \mc D\cup\mc M),\mc I,\mc D\cup\mc M)$). The only properties of iteration-goodness that are not obviously satisfied by definition are \cref{IG1} (strong well-distributedness) and \cref{IG3} (iteration-typicality).

\subsubsection{Strong well-distributedness}
To show that $(\mc I,\mc D\cup \mc M)$ is strongly $(p,O(1),n^{-\omega(1)})$-well-distributed, we perform a very similar calculation as in \cref{sub:CD1-consistency}. Consider any $s,t,r\in \mb N$, any distinct triangles $I_1,\ldots,I_s$, $D_1,\ldots,D_{t}$ in $K_N$, and any distinct edges $e_1,\ldots,e_r$ in $K_N$. Let $\mc T=\{D_1,\dots,D_t\}$, and consider a partition $\mc T=\mc T^0\cup \mc T^*\cup \mc T^\dagger\cup \mc T^\ddagger$.

We are interested in the probability of the event that $I_1,\dots,I_s$ are contained in $\mc I$, and $e_1,\dots,e_r$ are uncovered by $\mc I$, and $\mc T^0\subseteq\mc D,\mc T^*\subseteq\mc M^*, \mc T^\dagger\subseteq\mc M^\dagger, \mc T^\ddagger\subseteq\mc M^\ddagger$.
This probability is at most
\begin{align*}&C^{s+t+r}\left(p^{r+3|\mc T^*|}N^{-s}\prod_{T \in \mc T^0}\frac{p}{|U_{\on{lev}(T)}|}+n^{-\omega(1)}\right)\cdot \left((2n^{-\theta})^{2|\mc T^\ddagger|}+n^{-\omega(1)}\right)\\
&\qquad\qquad\qquad\cdot\left(\left(\frac{C^*}{p^2 n}\right)^{|\mc T^*|}\left(\frac{C^*}{n^{\beta}}\right)^{|\mc T^\dagger|}+n^{-\omega(1)}\right)\cdot\left(\frac{C^\dagger}{n^{-2\theta} p^2 n^{1-\rho}}\right)^{|\mc T^\dagger|}\cdot \left(\frac{n^\gamma}{n^{-\theta}p^2n^{1-\rho}}\right)^{|\mc T^\ddagger|}.
\end{align*}
(Here we have used that the probability a set of $2|\mc T^\ddagger|$ edges of $G$ between $U_k\setminus U_{k+1}$ and $U_{k+1}$ are all uncovered by $\mc M^*$ is at most $(2n^{-\theta})^{2|\mc T^\ddagger|}+n^{-\omega(1)}$. This incorporates the definition of $R$, and \ref{item:nibble-prob-guarantee}, \ref{item:CD1-prob-guarantee}, and \ref{item:CD2-prob-guarantee}, and assumes that $\theta$ is small compared to $\beta$).

Summing over all $4^s$ choices of the partition $\mc T=\mc T^0\cup \mc T^*\cup \mc T^\dagger\cup \mc T^\ddagger$ verifies the desired well-distributedness claim, provided $\theta,\nu,\rho$ are small with respect to $\beta$, and $\gamma,\nu,\rho$ are small compared to $\theta$.

\subsubsection{Iteration-typicality}

We are assuming that $(G,\mc A)$ is $(p,q,o(1),4)$-iteration-typical, and we need to prove that the updated data $(G_{U_{k+1}}(\mc M),\mc A_{U_{k+1}}(\mc I\cup \mc D\cup\mc M))$ is likely to still be $(p,q,o(1),4)$-iteration-typical. We will simply show that the process of updating the data affects iteration-typicality in a negligible fashion.

Specifically, let $G^\mr{diff}=G[U_{k+1}]\setminus G_{U_{k+1}}(\mc M)$ and let $\mc A^\mr{diff}$ be the set of triangles in $\mc A\setminus \mc A_{U_{k+1}}(\mc I\cup \mc D\cup\mc M)$ which lie in the updated graph $G_{U_{k+1}}(\mc M)$. It suffices to show that with probability $1-n^{-\omega(1)}$, for each $k+1\le \letter< \ell$ and $\letter^*\in \{\letter,\letter+1\}$:
\begin{enumerate}
    \item every vertex $v\in U_{\letter}$ has
    $o(p|U_{\letter}|)$ neighbors in $U_{\letter}$ and 
    $o(p|U_{\letter+1}|)$ neighbors in $U_{\letter+1}$ with respect to $G^\mr{diff}$, and
    \item for any edge subset $Q\subseteq G[U_k]$ spanning $|V(Q)|\le 4$ vertices:
    \begin{enumerate}
        \item there are $o(p^{|V(Q)|}|U_{\letter^*}|)$ vertices $u\in U_{\letter^*}$ for which there is a triangle $uvw\in \mc A^\mr{diff}$ for some $vw\in Q$, and
        \item there are $o(p^{|V(Q)|}|U_{\letter^*}|)$ vertices $u\in U_{\letter^*}$ for which there is an edge $uv\in G^\mr{diff}$ for some $v\in V(Q)$.
    \end{enumerate}
\end{enumerate}

First, (2a) follows from \cref{lem:moment-quasi,lem:moments-asymptotic}. Indeed, here we are applying \cref{lem:moment-quasi} with $y=O(1)$, $z=O(n^\beta)$, where the weights of the form $\pi_{T,1}$ and $\pi_{T,2}$ describe the probability that a triangle $T$ is in $\mc I$ or $\mc D\cup\mc M$, respectively, and the weights of the form $\pi_e$ describe the probability that an edge $e$ is uncovered by $\mc I$. With notation as in \cref{lem:moments-asymptotic,lem:moment-quasi}, $X(\mf R_{Q,i^*,j})$ is an upper bound on the number of configurations $\mc S\in \mf F_{j}$ which cause a triangle containing an edge of $Q$ to contribute to $\mc A^\mr{diff}$. The joint probability guarantees that we need to apply \cref{lem:moments-asymptotic} are provided by the strong well-distributedness fact we have just proved in the last subsection.

Then, for (1) and (2b) we simply observe that $G^\mr{diff}\subseteq G^\ddagger$, where $G^\ddagger$ is the set of edges in $E(\mc M^\ddagger)$ which are inside $U_{k+1}$. As in \cref{subsub:sparsification}, using \cref{item:link-overlaps,item:CD2-prob-guarantee}, we can see that $G^\ddagger$ is stochastically dominated by a random subgraph of $G[U_{k+1}]$ in which every edge is present with probability $n^{-\theta+2\nu+\gamma+\rho+o(1)}=o(p^{|V(Q)|})=o(p)$ independently (we are assuming that $\gamma,\rho,\nu$ are sufficiently small relative to $\theta$). So, (1) and (2b) both follow from a Chernoff bound.

\section{Proof of the main theorem}\label{sec:final}
\begin{proof}[Proof of \cref{thm:main}]
Before proceeding with the proof, let us recall the structure of the argument that will follow. We will first set up a vortex $V(K_N)=U_0\supseteq\cdots\supseteq U_\ell$ and place an absorbing structure $H$ on top of $U_\ell$. Then we will run an initial sparsification on $K_N\setminus H$ using a high-girth triple process (\cref{def:high-girth-nibble}) with appropriate data. We will use \cref{prop:initial-nibble} as well as guarantees from \cref{thm:random-high-girth-nibble} to control this step. This introduces a factor $q=\Omega_g(1)$ (which corresponds to $q_g>0$ in \cref{subsec:sparsification}) of sparsification on the available triangles. Then, we will run iterative absorption using \cref{prop:iter} to sequentially cover edges in stages until we have a leftover contained purely within $U_\ell$. Finally, we use the absorbing structure to complete the decomposition.

There are a number of key parameters involved in the argument (and it is important to avoid circular dependence when setting these parameters). The parameter $\beta$ will depend only on $g$, and control how long we can run any of our high-girth triple processes (\cref{def:high-girth-nibble}) involved in both initial sparsification and in the iterative absorption (\cref{prop:iter}). The parameter $\rho$ will depend only on $g$ and our choice of $\beta$, and will control the contraction of set sizes as we pass through the vortex (i.e., $|U_k|=\lfloor|U_{k-1}|^{1-\rho}\rfloor$). The parameter $\ell$ is chosen based on $g,\rho,\beta$ and is the length of the vortex; we just need to ensure $\ell$ is big enough so that the final set $U_\ell$, which has size $m$, is small enough that we can fit our absorbing structure $H$ with enough room for our argument. Next, $p$ is the density of the initial sparsification. It cannot be too sparse, and in fact it cannot be too sparse with respect to the final vortex set $U_\ell$ (with the polynomial rate of sparsity constrained by $\nu$ defined in \cref{prop:iter} and constrained by $\tilde B$ defined in \cref{prop:initial-nibble}). It cannot be too dense due to constraint focusing, and this requirement is encapsulated in \cref{prop:iter}; it simply has to decay at \emph{some} polynomial rate. Finally, in order for us to be able to iterate \cref{prop:iter}, note that at this stage $\ell$ can be thought of as a fixed constant (independent of $N$). We note that \cref{prop:iter} requires input control of $(p,q,\xi,C,b)$-iteration-goodness and outputs updated data which is $(p,q,o(1),O(1),n^{-\omega(1)})$-iteration-good. Although these parameters may therefore decay in quality slightly, the fact that we are performing only a constant number of steps means that we can carry this through.

Now we provide the formal proof. Fix $\beta<\min\{\beta_{\ref{thm:random-high-girth-nibble}}(g,O_g(1)),\beta_{\ref{prop:initial-nibble}}(g),\beta_{\ref{prop:iter}}(g)\}$ (with the same $O_g(1)$ as in the proof of \cref{prop:initial-nibble}), let $\rho=\rho_{\ref{prop:iter}}(g,\beta)$, and choose  $\ell=O(1)$ such that $(2C_{\ref{thm:absorbers}})\cdot (1-\rho)^\ell< \beta/(10g)$.
Fix a descending sequence of subsets (a ``vortex'') $V(K_N)=U_0\supseteq \dots\supseteq U_\ell$, with $|U_k|=\lfloor|U_{k-1}|^{1-\rho}\rfloor$ for each $1\le k\le\ell$. Then, let $m=|U_\ell|$, and fix an absorbing structure $H$ (as in \cref{thm:absorbers}) whose ``flexible set'' $X$ has size $m$. By the choice of $\ell$ this absorbing structure $H$ has at most $N^{\beta/(10g)}$ vertices; embed it in $K_N$ in such a way that $X$ coincides with $U_\ell$ and all other vertices of $H$ are in $U_0\setminus U_1$.

Let $G=K_N\setminus H$, let $\mc{B}$ be as in \cref{thm:absorbers}, and for each $4\le j\le g$, let $\mf F^\mc{B}_j$ be as defined in \cref{lem:absorber-well-spread}. Note that $|\mc{B}|^{2g}\le ((N^{\beta/(10g)})^3)^{2g}\le N^{\beta}$, and $\mf F^\mc{B}_j$ is $(O(1),|U_k|^\beta)$-well-spread with respect to $U_0\supseteq\cdots\supseteq U_k$ for all $0\le k\le\ell$ by \cref{lem:absorber-well-spread}. Let $\mc{A}$ be the set of triangles in $G$ which do not create an Erd\H{o}s configuration when added to $\mc{B}$.

Let $\nu= \nu_{\ref{prop:iter}}(g,\rho)$ and $\tilde B=\tilde B_{\ref{prop:initial-nibble}}(g)$, and let $p=\max(m^{-\nu},m^{-1/\tilde B})$. Now, let $\mc I$ be the partial Steiner triple system obtained from $(1-p)|E(G)|/3$ steps of the generalized high-girth process in \cref{def:high-girth-nibble}, and let $n=N$. Applying \cref{prop:initial-nibble} together with the analysis in the proof of \cref{thm:random-high-girth-nibble} (and taking an appropriate union bound over relevant choices of vertices $v$, vertex sets $U_i$ and edge sets $Q$) shows that $(\mc I,\emptyset)$ is strongly $(p,O(1),n^{-\omega(1)})$-well-distributed and $(p,q,o(1),4)$-iteration-typical for some $q = \Omega_g(1)$. (The precise value of $q$ comes from \cref{prop:initial-nibble}(2) applied at the end time $t=(1-p)|E(G)|/3$ of the initial sparsification; we have $q=\exp(-\rho(t))$ with $\rho$ as in \cref{def:traj}, which is easily verified to be constant order using $t=\Theta(n^2)$, $|\mc{A}(0)|=\Theta(n^3)$, and $|\mf{F}_j^\mc{B}|=\Theta_g(n^j)$.)

It follows that with probability $1-n^{-\omega(1)}$ the data $(G_{U_0}(\mc I),\mc A_{U_0}(\mc I),\mc I,\emptyset)$ are $(p,q,o(1),O(1),n^{-\omega(1)})$-iteration-good, suitable for applying \cref{prop:iter}, taking $\mf F_j=\mf F_j^\mc{B}$. (The required divisibility condition \cref{IG0} follows from the fact that both $K_N$ and $H$ are triangle-divisible.) We now apply \cref{prop:iter} repeatedly ($\ell$ times, which is a constant number independent of $N$), to obtain a set of triangles $\mc M$ in $G$ that covers all edges except those in $G[X]$ and avoids all configurations in $\bigcup_{j=4}^g\mf F^\mc{B}_j$. Finally, the crucial property \cref{AB1} of our absorbing structure $H$ allows us to transform $\mc I\cup\mc M$ into a Steiner triple system with girth greater than $g$, since the final leftover graph, which was obtained by removing triangles from the triangle-divisible graph $K_N$, is triangle-divisible. By the definitions of $\mc A$ and $\mf{F}_j^{\mc{B}}$, this last step does not introduce any forbidden configurations (all triangles in $\mc I\cup \mc M$ come from $\mc A$, the only other triangles in our Steiner triple system come from $\mc{B}$, and the definition of the ``induced'' forbidden configurations within $\mf{F}_j^\mc{B}$ appropriately rules out Erd\H os configurations with multiple triangles in $\mc A$ and some triangles in $\mc{B}$).
\end{proof}

\begin{proof}[Proof sketch of \cref{thm:counting-lower-bound}]
We proceed in almost exactly the same way as in \cite{Kee18}. To prove \cref{thm:counting-lower-bound}, we observe that in the above proof of \cref{thm:main}, we are (randomly) constructing an \emph{ordered} Steiner triple system (where the triangles in $\mc I$ are ordered according to the order they are selected in our generalized high-girth triple process, and we arbitrarily order the other triangles to appear after the triangles in $\mc I$). Our construction succeeds with probability $1-o(1)$, and the probability of any particular successful outcome is at most about \[\prod_{t=1}^{(1-p)|E(G)|/3} \frac1{A(t)},\]
where, as in the proof of \cref{prop:initial-nibble}, we let $A(t)=(1/3)|E(t)|f_\mr{edge}(t)$ be the approximate number of available triangles at step $t$ of the initial high-girth process. So, the number of possible outcomes of our construction is at least about $\prod_{t=1}^{(1-p)|E(G)|/3} A(t)$. We can approximate the logarithm of this product by an integral, and computations as in \cite[Section~6]{GKLO20} provide a lower bound on the desired number of ordered Steiner triple systems with girth greater than $g$. Dividing by the number of possible orderings $(\binom{N}{2}/3)!$ yields the desired result.
\end{proof}

\bibliographystyle{amsplain0.bst}
\bibliography{main.bib}

\end{document}